\documentclass[12pt, letterpaper]{article}
\RequirePackage{amsthm,amsmath,amsfonts,amssymb}
\RequirePackage[authoryear]{natbib} 
\RequirePackage[colorlinks,citecolor=blue,urlcolor=blue]{hyperref}

\usepackage{bm}
\usepackage{mathrsfs}  
\usepackage{appendix}
\usepackage{commath}
\usepackage{dsfont}
\usepackage{subfig}
\usepackage{xcolor}
\usepackage{mathtools}
\usepackage{enumitem}
\usepackage{fancyhdr}
\usepackage{mathrsfs} 
\usepackage{xr}
\usepackage[lmargin=0.9in, rmargin=0.9in]{geometry}
\usepackage[nodisplayskipstretch]{setspace}

\newcommand{\N}{\mathbb{N}}
\newcommand{\R}{\mathbb{R}}
\newcommand{\eps}{\varepsilon}
\renewcommand{\P}{\mathbb{P}}
\newcommand\E{\mathbb{E}}

\theoremstyle{plain}
\newtheorem{theorem}{Theorem}[section]
\newtheorem{lemma}[theorem]{Lemma}
\newtheorem{corollary}[theorem]{Corollary}
\newtheorem{proposition}[theorem]{Proposition}

\theoremstyle{remark}
\newtheorem{remark}{Remark}[section]
\newtheorem{example}{Example}[section]
\newtheorem{question}{Question}
\newtheorem{assumption}{Assumption}[section]

\begin{document}

\title{Consistency of~$p$-norm based tests in high dimensions: characterization, monotonicity, domination\footnote{We would like to thank Davy Paindaveine, Michael Wolf, and seminar participants in Bristol/Warwick, Durham, Ohio, Turku, Paris, Rome, Vienna, Wisconsin and York for helpful comments and discussions. David Preinerstorfer gratefully acknowledges support from the Program of Concerted Research Actions (ARC) of the Universit\'e libre de Bruxelles.}}

\date{}

\author{
	\begin{tabular}{c}
		Anders Bredahl Kock \\ 
		\small	University of Oxford \\
		\small	CREATES, Aarhus University\\
		\small	10 Manor Rd, Oxford OX1 3UQ
		\\
		\small	{\small	\href{mailto:anders.kock@economics.ox.ac.uk}{anders.kock@economics.ox.ac.uk}} 
	\end{tabular}
	\and
	\begin{tabular}{c}
		David Preinerstorfer \\ 
		{\small	SEW-SEPS} \\ 
		{\small	 University of St.~Gallen} \\
		{\small	Varnb\"uelstrasse 14, 9000 St.~Gallen} \\ 
		{\small	 \href{mailto:david.preinerstorfer@unisg.ch}{david.preinerstorfer@unisg.ch}}
	\end{tabular}
}

\maketitle

\begin{abstract}
Many commonly used test statistics are based on a norm measuring the evidence against the null hypothesis. To understand how the choice of a norm affects power properties of tests in high dimensions, we study the \emph{consistency sets} of~$p$-norm based tests in the prototypical framework of sequence models with unrestricted parameter spaces, the null hypothesis being that all observations have zero mean. The consistency set of a test is here defined as the set of all arrays of alternatives the test is consistent against as the dimension of the parameter space diverges. We characterize the consistency sets of~$p$-norm based tests and find, in particular, that the consistency against an array of alternatives cannot be determined solely in terms of the~$p$-norm of the alternative. Our characterization also reveals an unexpected monotonicity result: namely that the consistency set is strictly increasing in~$p \in (0, \infty)$, such that tests based on higher~$p$ strictly dominate those based on lower~$p$ in terms of consistency. This monotonicity allows us to construct novel tests that dominate, with respect to their consistency behavior, \emph{all}~$p$-norm based tests without sacrificing size.
\end{abstract}
\doublespacing
\section{Introduction} 

Since the advent of the Big Data era, the development of procedures for testing hypotheses in high-dimensional models has received considerable attention with applications in fields such as genomics, finance, economics and engineering. Many tests are based on a~$p$-norm, measuring the evidence against the null hypothesis. In contrast to classical low-dimensional testing problems, tests based on different~$p$-norms tend to be consistent against different types of alternatives in the high-dimensional setting, cf., e.g., \cite{ingster}. This observation led \cite{fan2015} to put forward their remarkable power enhancement principle, a combination procedure that improves an initial tests by combining it with another one. The resulting enhanced test has the same asymptotic size as the initial test but is consistent against more alternatives. Specifically, \cite{fan2015} enhanced the Euclidean norm based test by combining it with the supremum-norm based test, thus constructing a test that: (i) retains the good consistency properties of the Euclidean norm based test against dense alternatives and (ii) is also consistent against a large class of sparse alternatives. This possibility of increasing power by combining tests has led to an intense recent interest in constructions combining tests based on various~$p$-norms in many types of high-dimensional testing problems, cf.~\cite{xu2016adaptive}, \cite{yang2017weighted}, \cite{yu2020fisher}, \cite{he2021asymptotically}, \cite{yu2021power} [testing high-dimensional means and covariance matrices]; \cite{zhang2021adaptive} [change point detection]; \cite{jammalamadaka2020sobolev} [tests for uniformity on the sphere]; \cite{feng2020max} [tests for cross-sectional independence
in high-dimensional panel data models]. Furthermore, \cite{wang2018testing} pointed to the potential usefulness of combining tests based on the Euclidean and supremum-norm in the context of high-dimensional quantile regression. By combining either tests based on the Euclidean and supremum-norm (sometimes referred to as ``max-sum'' tests) or a finite number of norms, these papers establish that one can construct tests with better power properties compared to tests based on a single~$p$-norm in the concrete testing problems considered.

The classical results in \cite{ingster} as well as the more recent contributions in the context of the power enhancement principle of \cite{fan2015} discussed above have in common that \emph{sufficient} conditions for consistency (or inconsistency) of the tests under consideration are established. However, even for a test based on a single $p$-norm it is not known to date what the necessary and sufficient condition for consistency of such a test is in high dimensions. As a consequence, also the following questions are open:
\begin{itemize}
	\item Can one obtain tests that are consistent against more alternatives by using~$p$-norms other than the 1-, Euclidean or supremum-norm, e.g., the~$4$-norm?
	\item Is it possible to rank~$p$-norm based tests in terms of their consistency properties?
\end{itemize}
Similarly, in the context of combinations of~$p$-norm based tests, one may ask questions such as:
\begin{itemize}
\item When combining several~$p$-norm based tests: how many and which should one combine if one wants to obtain a test with optimal consistency properties?
\item Is there a test that simultaneously dominates all~$p$-norm based tests in terms of consistency?
\end{itemize}

\subsection{Contributions}

We study the questions raised above in the prototypical context of a sequence model. To reject the ``global null'' of all observations having zero mean, these tests gauge whether the~$p$-norm of the observation vector exceeds a given critical value. Our contributions are as follows:
\begin{itemize}
\item \emph{Characterization}: We characterize the \emph{consistency set} of each~$p$-norm based test, that is the set of arrays of alternatives that the test is consistent against. Somewhat surprisingly, we find that whether the~$p$-norm based test is consistent against an array of alternatives cannot be determined solely by the~$p$-norm of the alternative (which would in general only lead to sufficient but not necessary conditions for consistency). For~$p\in(2,\infty)$ it actually suffices that the~$p$-norm \emph{or} the Euclidean norm of the alternative is sufficiently large.
\item \emph{Monotonicity}: For~$0<p<q<\infty$ the~$q$-norm based test is consistent against any array of alternative that the~$p$-norm based test is consistent against --- and more. This finding sheds light on the consequences of employing tests based on the commonly used~$1$- and Euclidean norms, as these are strictly dominated in terms of the magnitude of their consistency sets. Note that the strict ranking of~$p$-norm based tests in terms of the magnitude of their consistency sets is in contrast to the absence of such a strict ranking in terms of local power properties, e.g., \cite{pinelis2010asymptotic} and \cite{pinelisI}. We also identify the structure of the additional alternatives that a test based on a larger power is consistent against. This strict dominance relation does not extend to~$p=\infty$. 
\item \emph{Domination}: Although no single~$p$-norm based test dominates all other ones in terms of the magnitude of its consistency set, we construct a single test that is consistent against any array of alternatives that any~$p$-norm based test is consistent against --- and more. Thus, it is possible to \emph{simultaneously} dominate all members of the commonly used class of~$p$-norm based tests. We also revisit the results given in \cite{ingster} in Appendix~\ref{sec:minimax} where we also establish some new adaptation results in the minimax framework building on our results obtained in earlier sections.
 %Finally, we elucidate some consequences and relations to minimax adaptive testing of our findings. 
\end{itemize}

An immediate consequence of our findings for the current practice of either using a Euclidean norm based test, a supremum-norm based test or a combination of these, is that a test that is consistent against strictly more alternatives than any of these can be explicitly constructed. In fact, the test we propose is consistent against strictly more alternatives than any test that combines a finite fixed number of~$p$-norm based tests.

\section{Framework}\label{sec:GLM}

\subsection{Model and testing problem}

We consider a sequence model
\begin{equation}\label{eqn:model}
	y_{i,d} = \theta_{i,d} + \varepsilon_i, \quad i = 1, \hdots, d,
\end{equation}
where the~$\theta_{i,d} \in \R$ are unknown parameters, the unobserved error terms~$\varepsilon_i$ have mean zero and are i.i.d., and where~$y_{1,d}, \hdots, y_{d,d}$ are the observations. \emph{For simplicity of presentation, and since it is the most important case, we shall assume throughout the main text of this article that~$\varepsilon_i \sim \mathbb{N}(0, 1)$.} All results generalize appropriately to non-normal errors under suitable assumptions concerning the tail of the distribution of~$\varepsilon_i$. Such results are established in Appendix~\ref{app:AUX}.

For notational simplicity, we write~$\bm{y}_d = (y_{i,1}, \hdots y_{i,d})$,~$\bm{\varepsilon}_d = (\varepsilon_1, \hdots, \varepsilon_d)$, and~$\bm{\theta}_d = (\theta_{1,d}, \hdots, \theta_{1,d}) \in \R^d$, upon which we may write~\eqref{eqn:model} equivalently as~$\bm{y}_d = \bm{\theta}_d + \bm{\varepsilon}_d$.

In the sequence model~\eqref{eqn:model}, we study properties of tests for the testing problem
\begin{equation}\label{eqn:tp}
	H_{0, d}: \bm{\theta}_d = \bm{0}_d \quad \text{ against } \quad H_{1, d}: \bm{\theta}_d \in \R^d \setminus \{\bm{0}_d\},
\end{equation}
where~$\bm{0}_d = (0, \hdots, 0) \in \R^d$. That is, we are concerned with what is sometimes referred to as the ``global null'' of no effect. Oftentimes, a test for this hypothesis is done in the initial stage of a whole series of potential further investigations that are carried out only if the global null is rejected. It is crucial to use a test with good power properties at the initial stage. 

The sequence model~\eqref{eqn:model}
is an often used prototypical framework in high-dimensional statistics. On the one hand, the simplicity of the model allows one to focus on asymptotic power properties as the dimension~$d$ diverges to~$\infty$ in a clean framework. On the other hand,  the framework assumes away certain complications, such as the estimation of nuisance (e.g., variance) parameters, or the interplay between sample size and the dimensionality of the parameter vector, which we do not address.

We also note that the results derived in a Gaussian sequence model immediately correspond to results in a Gaussian regression model by a sufficiency argument detailed (for completeness) in the following remark. This is a model of fundamental importance for applications. 
\begin{remark}\label{rem:linmod}
Let~$\bm{X}_d \in \R^{n \times d}$ be a matrix of full column rank, in particular implying~$d \leq n$. In the Gaussian linear regression model~$\mathbf{z}_d = \bm{X}_d \bm{\beta}_d + \mathbf{u}_d$ with~$\mathbf{u}_d \sim \mathbb{N}(\bm{0}_d, \bm{I}_d)$, the OLS estimator~$\hat{\bm{\beta}}_d$ is sufficient for~$\bm{\beta}_d$ and, denoting~$\bm{M}_d := (\bm{X}_d'\bm{X}_d)^{1/2}$, the same holds for~$\bm{M}_d\hat{\bm{\beta}}_d$, which satisfies~$$\bm{M}_d\hat{\bm{\beta}}_d := \bm{M}_d\bm{\beta}_d + \bm{M}_d^{-1} \bm{X}_d'\mathbf{u}_d \quad \text{ with } \quad \bm{M}_d^{-1}\bm{X}_d'\mathbf{u}_d \sim \mathbb{N}(\bm{0}_d, \bm{I}_d),$$
which is clearly of the form~\eqref{eqn:model}. Note also that~$\bm{\theta}_d := \bm{M}_d\bm{\beta}_d = \bm{0}_d$ if and only if~$\bm{\beta}_d = \bm{0}_d$. That is, this problem is statistically equivalent to testing~\eqref{eqn:tp} in a (re-parameterized) Gaussian sequence model.
\end{remark}

\begin{remark}
On a \emph{conceptual and informal} level, results obtained in Gaussian sequence models carry over much more broadly to any situation where the distribution of a properly standardized estimator is approximately Gaussian. The reasoning follows along similar lines as in Remark~\ref{rem:linmod}: Consider a situation where an estimator~$\hat{\bm{\beta}}_d$ for a target parameter~$\bm{\beta}_d \in \R^d$ is available the distribution of which satisfies
\begin{equation*}
\hat{\bm{\beta}}_d \approx \mathbb{N}(\bm{\beta}_d, \bm{\Omega}_d).
\end{equation*}
Suppose further that an invertible estimator~$\hat{\bm{\Omega}}_d \approx \bm{\Omega}_d$ is available, such that
\begin{equation*}
\hat{\bm{\Omega}}_d^{-1/2} \hat{\bm{\beta}}_d \approx \mathbb{N}(\bm{\Omega}_d^{-1/2}\bm{\beta}_d, \bm{I}_d).
\end{equation*}
Then, testing~$\bm{\beta}_d = \bm{0}_d$ on the basis of~$\hat{\bm{\beta}}_d$ and~$\hat{\bm{\Omega}}_d$ is approximated by testing~$\bm{\theta}_d :=  \bm{\Omega}_d^{-1/2}\bm{\beta}_d = \bm{0}_d$ in a Gaussian sequence model. Precise sets of conditions under which the above approximation statements hold depend on the interplay of the dimension of the target parameter and sample size and on particularities of the specific setup under consideration. 
\end{remark}

\subsection{Tests, asymptotic size, and consistency sets}

Suppose that for every~$d \in \N$ we are given a (possibly randomized) test~$\varphi_d$ for the testing problem~\eqref{eqn:tp}. That is,~$\varphi_d$ is a (Borel measurable) function from~$\R^d$ to~$[0, 1]$. We denote the set of all such sequences of tests~$\{\varphi_d\}_{d \in \N}$ by~$\mathbb{T}$. With some abuse of notation, we shall often abbreviate~$\{\varphi_d\}_{d \in \N}$ by~$\varphi_d$, to which we then also refer to as a sequence of tests or just as a test. %It will always be clear from the context whether a sequence of tests or a test for a fixed~$d$ is meant.

We say that the sequence of tests~$\varphi_d \in \mathbb{T}$ has \emph{asymptotic size}~$\alpha \in [0, 1]$ if and only if
\begin{equation}\label{eqn:size}
	\lim_{d \to \infty} \mathbb{E}\left(\varphi_d(\bm{\varepsilon}_d)\right) = \alpha.
\end{equation}
The subset of all sequences of tests with asymptotic size~$\alpha$ is denoted by~$\mathbb{T}_{\alpha} \subseteq \mathbb{T}$. We will mostly be concerned with~$\alpha \in (0, 1)$.

We consider an asymptotic framework where the coordinates of the parameter vector~$\bm{\theta}_d = (\theta_{1,d}, \hdots, \theta_{d,d})$ depend on~$d$. That is, we consider (uniform) power properties along triangular arrays of alternatives~$\bm{\vartheta} = \{\bm{\theta}_d: d\in \N\}$, where~$\bm{\theta}_d \in \R^d$ for every~$d \in \N$. The set of all such triangular arrays will be denoted by~$\bm{\Theta} := \bigtimes_{d = 1}^{\infty} \R^d$. As is common, we say that the sequence of tests~$\varphi_d$ is \emph{consistent} against an array~$\bm{\vartheta} = \{\bm{\theta}_d: d\in \N\} \in \bm{\Theta}$ if and only if
\begin{equation}\label{eqn:consistent}
	\lim_{d \to \infty} \mathbb{E}
	\left(
	\varphi_d(\bm{\theta}_d + \bm{\varepsilon}_d)
	\right) = 1.
\end{equation}
To any sequence of tests~$\varphi_d \in \mathbb{T}$ we associate its \emph{consistency set}~$\mathscr{C}(\varphi_d) \subseteq \bm{\Theta}$, i.e., the subset of all arrays of alternatives~$\bm{\vartheta} \in \bm{\Theta}$ that the sequence of tests~$\varphi_d$ is consistent against. The characterization and comparison of consistency sets is the main focus of the present article.

\subsection{Optimal consistency properties}\label{sec:optcons}

For a given~$\alpha \in (0, 1)$, one can compare the quality of two sequences of tests~$\psi_d \in \mathbb{T}_{\alpha}$ and~$\varphi_d \in \mathbb{T}_{\alpha}$ on the basis of their consistency sets.\footnote{In principle, one can also compare two sequences of tests with different asymptotic sizes by comparing their consistency sets, but the comparison is less meaningful in such situations since the consistency sets could depend on the asymptotic size.} Note that, by definition,~$\mathscr{C}(\psi_d) \supseteq \mathscr{C}(\varphi_d)$ if and only if~$\psi_d$ is consistent against \emph{all} alternatives that~$\varphi_d$ is consistent against. If this is the case,~$\psi_d$ weakly dominates~$\varphi_d$ in terms of consistency, whereas~$\psi_d$ strongly dominates~$\varphi_d$ in terms of consistency if even~$\mathscr{C}(\psi_d) \supsetneqq \mathscr{C}(\varphi_d)$.

In search for an ``optimal'' test, one would hope that there exists a single sequence of tests~$\psi_d \in \mathbb{T}_{\alpha}$ that is consistent against all alternatives that some~$\varphi_d \in \mathbb{T}_{\alpha}$ is consistent against, i.e.,~$\psi_d$ satisfies
\begin{equation}\label{eqn:reqpsi}
	\mathscr{C}(\psi_d) \supseteq \bigcup \{\mathscr{C}(\varphi_d): \varphi_d \in \mathbb{T}_{\alpha}\}.
\end{equation}
Such a test~$\psi_d$ (in case it exists) would be ``consistency-optimal'' in the sense that no other test with the same asymptotic size exists that is consistent against more alternatives.

Whereas the optimality property in~\eqref{eqn:reqpsi} is often achieved by the usual candidates of tests in standard finite-dimensional testing problems, the consistency properties of tests are much more delicate in high-dimensional testing  problems, where a~$\psi_d \in \mathbb{T}_{\alpha}$ satisfying~\eqref{eqn:reqpsi} typically does \emph{not} exist. This is illustrated in the following result, which can be obtained by a similar reasoning as in Section~1.1 of~\cite{kp1}, to which we refer for further discussions and results in the context of models that are locally asymptotically normal. 

\begin{theorem}\label{thm:impo}
	Let~$\alpha \in (0, 1)$. For every~$\psi_d \in \mathbb{T}_{\alpha}$ there exists a~$\overline{\psi}_d \in \mathbb{T}_{\alpha}$, such that:
	\begin{enumerate}
		\item~$\overline{\psi}_d \geq \psi_d$ holds for every~$d \in \N$, guaranteeing that~$\overline{\psi}_d$ has uniformly non-inferior asymptotic power compared to~$\psi_d$; 
		\item~$\overline{\psi}_d$ is consistent against an array of alternatives against which~$\psi_d$ has asymptotic power at most~$\alpha$. 
	\end{enumerate}
	In particular it holds that~$\mathscr{C}(\psi_d) \subsetneqq \mathscr{C}(\overline{\psi}_d)$.
\end{theorem}

Theorem~\ref{thm:impo} establishes that the consistency set of \emph{any} test (of asymptotic size~$\alpha \in (0, 1)$) can be strictly improved.\footnote{Actually one can show that any test can be strictly improved against a highly sparse array of alternatives; cf.~Theorem~\ref{thm:imposparse} in Appendix~\ref{app:MT}. In this sense, not even the most stringent sparsity assumptions overturn the conclusion of Theorem~\ref{thm:impo}.} Inspection of the proof, which largely builds on an argument in Section~3.4.2 of \cite{ingster}, shows that the improvement is achieved through the power enhancement principle of~\cite{fan2015}. Hence, no consistency-optimal test as fancied in the second paragraph of this section and~\eqref{eqn:reqpsi} in particular can exist.

The literature mainly offers two solutions in situations where (for a given optimality criterion) no optimal test exists: The first is to restrict one's attention to a certain subclass of alternatives. Restricting the class of alternatives is an approach that is often employed in high-dimensional and nonparametric problems, e.g., by focusing on sparse alternatives or all alternatives for which a given norm exceeds a certain threshold. One then explores and constructs tests that work well for such alternatives (but do perhaps not work so well for others). This approach then typically studies minimax rates of detection, and minimax consistent tests. We offer more discussion and provide some results in that direction in Appendix~\ref{sec:minimax}, but that is not our main focus.

A different approach is to leave the set of alternatives unrestricted as is, but instead focus on a specific class of tests, and to study optimality for this restricted class of tests. In essence, one fixes a subset~$\mathbb{T}^*_{\alpha} \subseteq \mathbb{T}_{\alpha}$ and looks for a test whose consistency set contains the consistency sets of all tests in~$\mathbb{T}^*_{\alpha}$; that is one searches for a~$\psi_d \in \mathbb{T}_{\alpha}$ that satisfies~\eqref{eqn:reqpsi} but with~$\mathbb{T}_{\alpha}$ replaced by~$\mathbb{T}^*_{\alpha}$. Whether or not a test~$\psi_d$ with the desired property exists clearly depends on~$\mathbb{T}^*_{\alpha}$. In the present article we explore this question for an important subclass of tests, namely~$p$-norm based tests.

\subsection{$p$-norm based tests}

For~$\bm{x} = (x_1, \hdots, x_d) \in \R^d$ and~$p \in (0, \infty]$, we define
\begin{equation}
	\|\bm{x}\|_p := 
	\begin{cases}
		\left(\sum_{i = 1}^d |x_i|^p\right)^{\frac{1}{p}} & \text{if } p < \infty, \\
		\max_{i = 1, \hdots, d} |x_i| & \text{else}.
	\end{cases}
\end{equation}
For~$p \geq 1$ the function~$\|\cdot\|_p$ defines a norm on~$\R^d$; whereas for~$p \in (0, 1)$ it only defines a quasinorm on~$\R^d$. For the sake of brevity, we will refer to~$\|\cdot\|_p$ as the~``$p$-norm'' also in case~$p \in (0, 1)$.  Given a radius~$r \in [0, \infty]$, we denote~$\mathbb{B}_p^d(r) := \{\bm{y} \in \R^d: \|\bm{y}\|_p \leq r\}$, i.e., the closed ``ball'' with respect to~$\|\cdot\|_p$ of radius~$r$ centered at the origin. 

It is well known that the likelihood-ratio test for~\eqref{eqn:tp} rejects if~$\|\bm{y}_d\|_2$, the Euclidean norm of the vector of observations~$\bm{y}_d$, exceeds a critical value (chosen to satisfy a given size constraint). More generally, any~$p$-norm delivers a test for~\eqref{eqn:tp}, i.e., one rejects the null hypothesis if~$\|\bm{y}_d\|_p$ exceeds some critical value. Given~$p \in (0, \infty]$ and a sequence of critical values~$\kappa_{d}$, we abbreviate the sequence of tests~$\mathds{1}\{\|\cdot\|_p \geq \kappa_{d}\}$ by~$\{p, \kappa_{d}\}$, and refer to such a test as a \emph{$p$-norm based} test. Consequently, the consistency set of such a sequence of tests is written as~$\mathscr{C}(\{p, \kappa_{d}\})$. 

By a classical result of~\cite{birnbaum1955} and~\cite{stein1956},
tests that reject if the~$p$-norm of~$\bm{y}_d$ exceeds a given critical value are \emph{admissible} for every~$d \in \N$ and every~$p \in [1, \infty]$; the reason is that their acceptance region is convex. Furthermore, the tests just described are all \emph{unbiased}, for every~$d\in \N$, due to Anderson's theorem, cf.~\cite{anderson1955integral}. Hence, all~$p$-norm based tests with~$p \in [1, \infty]$ are reasonable from a non-asymptotic point of view.

Following the Neyman-Pearson approach, we shall mostly be interested in the situation where the critical values~$\kappa_d$ are chosen such that the asymptotic size of~$\{p, \kappa_d\}$ is in~$(0, 1)$, i.e.,
\begin{equation}\label{eqn:sizep}
	\lim_{d \to \infty} \mathbb{P} \left(\|\bm{\varepsilon}_d \|_p \geq \kappa_{d}
	\right) = \alpha \in (0, 1).
\end{equation}
One has that~\eqref{eqn:sizep} is true if and only if 
\begin{equation}\label{eqn:cvpalpha}
	\kappa_{d} = 
	\begin{cases}
		\left[\left(\Phi^{-1}(1-\alpha) + o(1)\right)\sqrt{d} \sigma_p + d \mu_p \right]^{1/p}	 & \text{ if } p \in (0, \infty),\\[7pt]
		\sqrt{2\log(d)}-\frac{\log\log(d)+\log(4\pi) +2\log(-\log(1-\alpha)/2) + o(1)}{2\sqrt{2	\log(d)}} & \text{ if } p = \infty.
	\end{cases}
\end{equation}
Here~$\Phi$ denotes the cdf of the standard normal distribution and for~$p \in (0, \infty)$ we abbreviated~$\sigma_p^2 := \mathbb{V}ar(|\varepsilon_1|^p)$ and~$\mu_p := \mathbb{E}(|\varepsilon_1|^p)$. The equality in~\eqref{eqn:cvpalpha} follows since for~$p \in (0, \infty)$ the test statistic converges in distribution under the null (upon suitable centering and scaling) to the standard normal distribution, whereas for~$p = \infty$ it converges to a slight modification (only taking care of absolute values) of Gumbel's double exponential distribution (cf.~Lemma~\ref{lem:asyn} and Lemma~\ref{lem:crit_val} for formal statements).

\section{Consistency sets of~$p$-norm based tests}

We shall now derive a characterization of~$\mathscr{C}(\{p, \kappa_{d}\})$ for sequences~$\kappa_{d}$ satisfying~\eqref{eqn:sizep}. This characterization will in particular allow us to compare the consistency sets for different~$p \in (0, \infty]$. The present section is divided into the cases~$p \in (0, \infty)$ and~$p = \infty$, which lead to fundamentally different results. 

Characterizing the consistency set of a~$p$-norm based test requires us to establish a necessary and sufficient condition for the test to be consistent against an array of alternatives. For~$p$-norm based tests a set of sufficient conditions is also given in~\cite{ingster}, Sections~3.1.2--3.1.4. However, these conditions are not necessary. That is, there are arrays~$\bm{\vartheta}$ against which a~$p$-norm based test is consistent, but which do not satisfy the sufficient condition for consistency there.

\subsection{Case~$p \in (0, \infty)$}\label{sec:cpreal}

For any~$p \in \R$, denote by~$g_p: \R \to \R$ the function
\begin{equation}\label{eqn:rhodef}
	g_p(x) := x^2 \mathds{1}_{[-1,1]}(x) + |x|^p \mathds{1}_{\R \setminus [-1,1]}(x).
\end{equation}
In particular,~$g_p(x) = x^2 \wedge |x|^p$ for~$p \in (0, 2)$ and~$g_p(x) = x^2 \vee |x|^p$ for~$p \in [2, \infty)$, where for two real numbers~$x$ and~$y$ we write~$\min(x, y) = x \wedge y$ and~$\max(x, y) = x \vee y$ for readability.
%\begin{figure}
%\centering
%\includegraphics[width=0.5\linewidth]{rhofun}
%\caption{Graphs of the function~$g_q$ for several values of~$q$.}
%\label{fig:rhofun}
%\end{figure}

The consistency set of a~$p$-norm based test with non-trivial asymptotic size is characterized next.
\begin{theorem}\label{thm:conspreal}
	For~$p \in (0, \infty)$ and~$\{p, \kappa_d\} \in \mathbb{T}_{\alpha}$,~$\alpha \in (0, 1)$, we have
	\begin{equation}\label{eqn:cpreal}
		\bm{\vartheta} \in \mathscr{C}(\{p,\kappa_{d} \}) \quad \Leftrightarrow \quad 
		\frac{\sum_{i = 1}^d g_p(\theta_{i,d})}{\sqrt{d}} \to \infty.
	\end{equation}
\end{theorem}
It is a somewhat surprising aspect of Theorem~\ref{thm:conspreal} that, apart from the case~$p = 2$, the consistency set~$\mathscr{C}(\{p,\kappa_{d} \})$ cannot be entirely characterized in terms of the asymptotic behavior of the~$p$-norm of the elements of an array~$\bm{\vartheta} \in \bm{\Theta}$. Regardless of~$p$, coordinates of~$\bm{\theta}_d$ that are small in absolute value enter the consistency criterion in~\eqref{eqn:cpreal} via their squares, whereas coordinates with large absolute values enter differently and in dependence on~$p$. We illustrate this in Figure~\ref{fig:cont} by showing contour plots of the level sets of the function~$(x_1, x_2) \mapsto g_p(x_1)/\sqrt{2} + g_p(x_2)\sqrt{2}$ for~$p = 1, 3$, which are genuinely different from the level sets of the corresponding~$p$-norm in dimension~$d = 2$. Note that the level sets close to the origin are circular irrespective of the value of~$p$, i.e., in correspondence to level sets of a Euclidean norm, whereas level sets further away from the origin approach those of the~$p$-norm the test is based on. 

Theorem~\ref{thm:conspreal} follows from a more general result that also holds in non-Gaussian settings and which is given in Theorem~\ref{thm:consprealGEN} of Appendix~\ref{app:AUX}. All results in the present Section~\ref{sec:cpreal} carry over to this more general setting, cf.~the discussion after the proof of Theorem~\ref{thm:consprealGEN}.

\begin{figure}%
	\centering
	\subfloat{{\includegraphics[width=5cm]{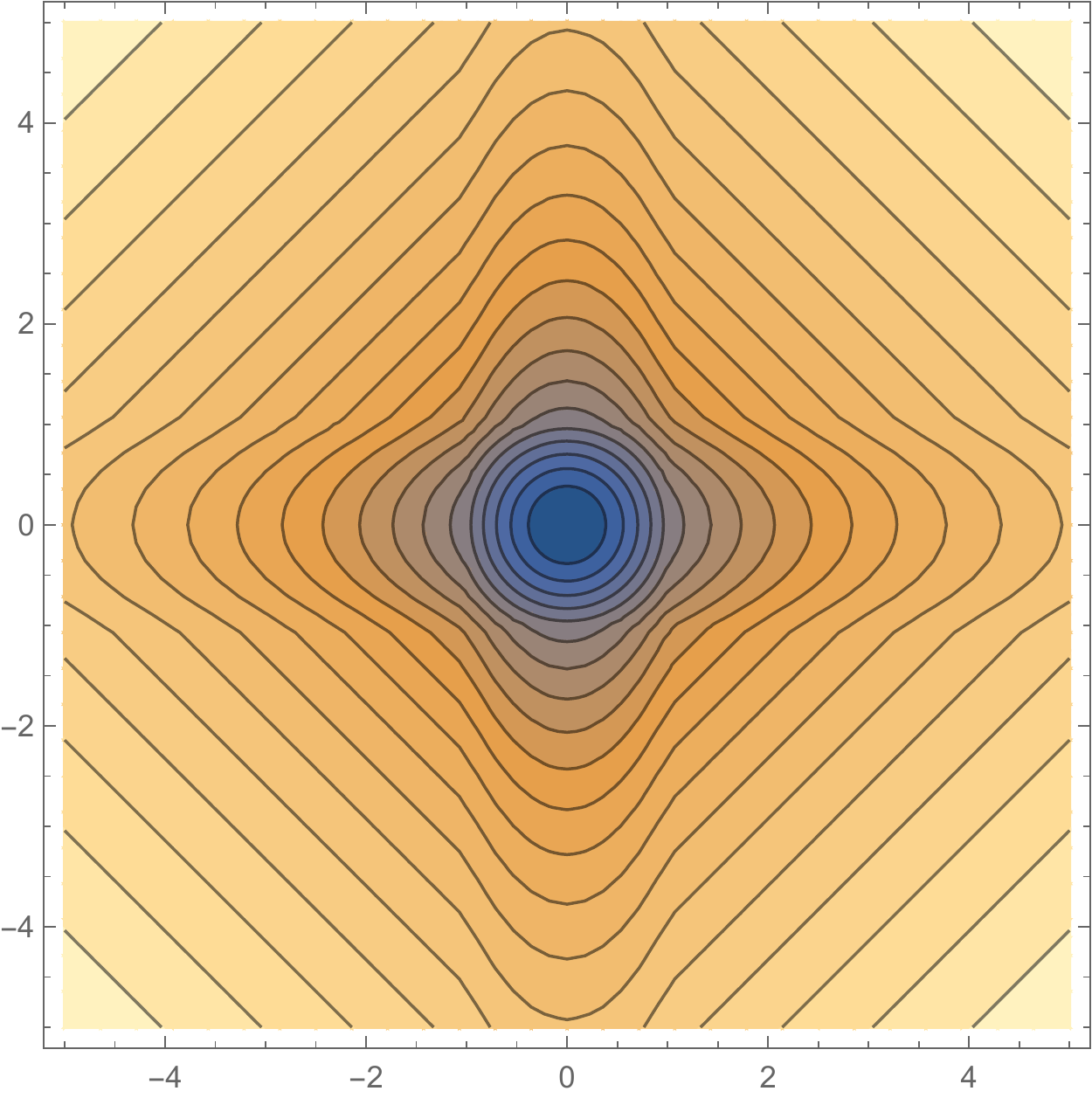} }}%
	\qquad
	\subfloat{{\includegraphics[width=5cm]{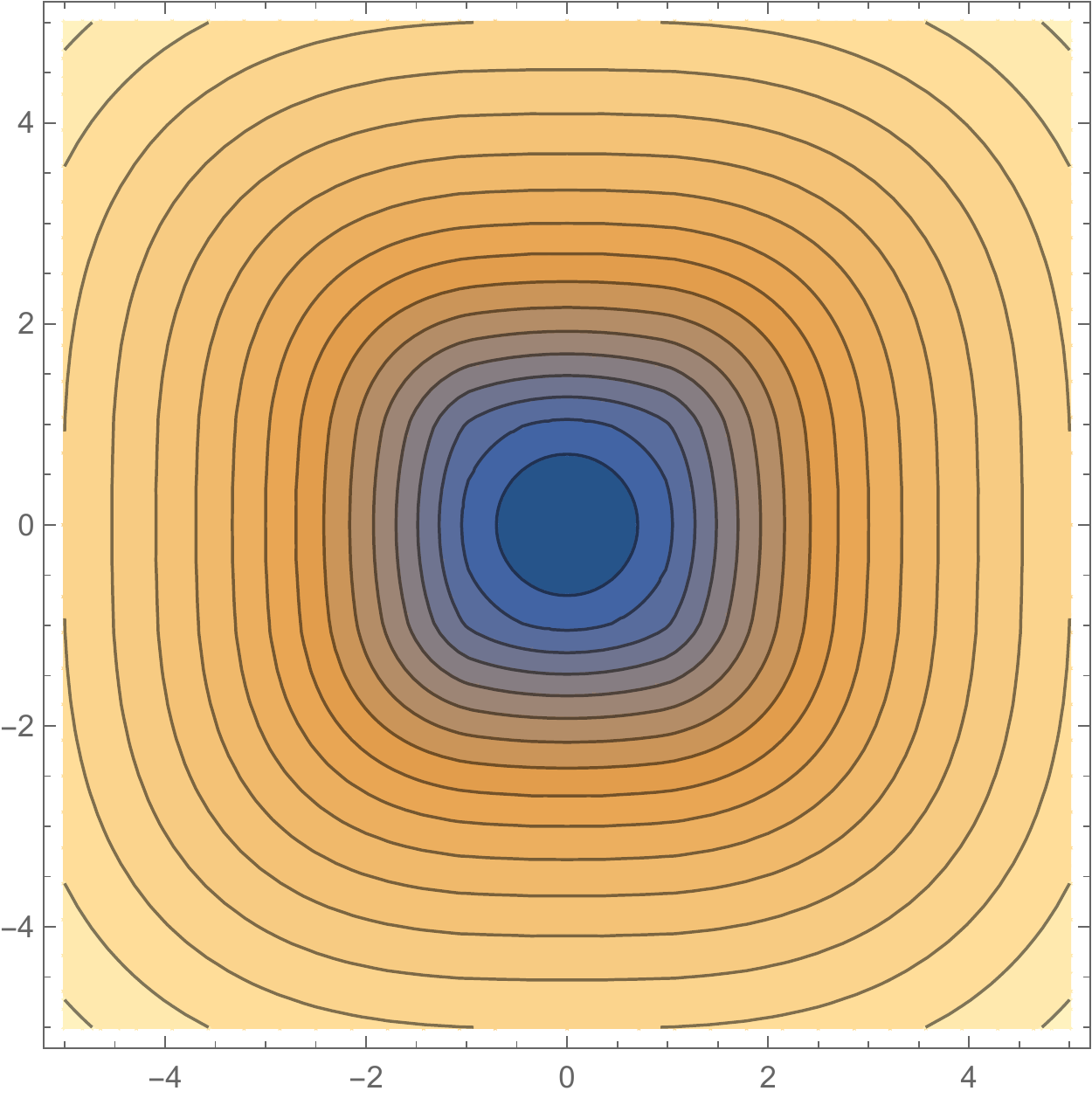} }}%
	\caption{Contour plots of the function characterizing the consistency of the~$p$-norm based test for~$d = 2$ and~$p = 1$ (left panel) and~$p = 3$ (right panel).}%
	\label{fig:cont}%
\end{figure}
\begin{remark}\label{rem:arb}
	There is nothing special about~the interval~$[-1,1]$ in the definition of~$g_p$ in Equation~\eqref{eqn:rhodef}. In principle, one could replace~the interval~$[-1,1]$ in the definition of~$g_p$ by any interval~$[-M, M]$,~for a fixed~$M>0$, and the statement in~\eqref{eqn:cpreal} would still be correct with this re-defined function~$g^{(M)}_p$, say, in place of~$g_p$. This follows immediately from~$$0 < \inf_{x \in \R \setminus \{0\}} \frac{g^{(M)}_p(x)}{g_p(x)} \leq \sup_{x \in \R \setminus \{0\}} \frac{g^{(M)}_p(x)}{g_p(x)} < \infty.$$ We have chosen~$M = 1$ in the formulation of Theorem~\ref{thm:conspreal} for concreteness and because it is the most convenient choice for later use in the proof of Theorem~\ref{thm:incl}. The content of this remark will be instrumental in establishing Theorem~\ref{thm:cons_prop} further below. 
\end{remark}

We emphasize (also for later use in Section~\ref{sec:dom}) that the condition on the right in Equation~\eqref{eqn:cpreal} does not depend on~$\alpha$. That is, as long as the sequence of critical values~$\kappa_d$ is chosen such that the asymptotic size of the corresponding test equals~$\alpha \in (0, 1)$ (cf.~also~\eqref{eqn:cvpalpha}), the consistency set does not depend on the concrete value of~$\alpha$.

An immediate consequence of Theorem~\ref{thm:conspreal} is the following observation, which we shall make use of later. In particular, it contains an equivalent way of writing~\eqref{eqn:cpreal} in case~$p \in [2, \infty)$ which further clarifies that the consistency of the~$p$-norm based test against any array of alternatives cannot be settled based on the sequence~$\|\bm{\theta}_d\|_p$ alone. The condition~$\|\bm{\theta}_d\|_p^p/\sqrt{d} \to \infty$, which is also given in~\cite{ingster} Corollary 3.6, is only sufficient.
\begin{corollary}\label{cor:rewrite}
	For~$p \in (0, \infty)$ and~$\{p, \kappa_d\} \in \mathbb{T}_{\alpha}$,~$\alpha \in (0, 1)$, we have:
	\begin{enumerate}
		\item If~$p \in [2, \infty)$, then
		\begin{equation}\label{eqn:equi}
			\bm{\vartheta} \in \mathscr{C}(\{p,\kappa_{d} \}) \quad \Leftrightarrow \quad d^{-1/2} \left(\|\bm{\theta}_d\|^2_2 \vee \|\bm{\theta}_d\|^p_p\right) \to \infty.
		\end{equation}
		\item If~$p \in (0, 2)$,
		\begin{equation}\label{eqn:equi2}
			\bm{\vartheta} \in \mathscr{C}(\{p,\kappa_{d} \}) \quad \Rightarrow \quad d^{-1/2} \left( \|\bm{\theta}_d\|^2_2 \wedge \|\bm{\theta}_d\|^p_p\right)	 \to \infty;
		\end{equation}
		but the condition on the right in~\eqref{eqn:equi2} does not imply the condition to the left.
	\end{enumerate} 
\end{corollary}
We shall now highlight a ``monotonicity'' result concerning the consistency set of~$p$-norm based tests, which can be obtained from Theorem~\ref{thm:conspreal}. This result shows that in terms of consistency properties, and choosing among~$p$-norm based tests with~$p \in (0, \infty)$ and asymptotic size in~$(0, 1)$,
it is best to choose~$p$ ``large.'' \footnote{Because the statement of Theorem~\ref{thm:incl} involves tests based on different exponents and the sequences of critical values effecting asymptotic size control depend on the respective exponent, we here explicitly indicate this dependence simply to distinguish the two sequences. A similar convention will be applied in the statements below when needed.}

\begin{theorem}\label{thm:incl}
	For sequences of tests~$\{p, \kappa_{d,p}\}$ and~$\{q, \kappa_{d,q}\}$ with asymptotic sizes in~$(0, 1)$, and~$0 < p < q < \infty$, we have~$$\mathscr{C}(\{p, \kappa_{d, p}\}) \subsetneqq \mathscr{C}(\{q, \kappa_{d, q}\}).$$
\end{theorem}

In other words, the larger~$p \in (0, \infty)$, the larger the consistency set of~$\{p, \kappa_{d, p}\}$. Somewhat unfortunately, however, the theorem shows --- at the same time --- that it is impossible to choose~$p$ ``large enough.'' That tests based on larger powers have larger consistency sets may appear contradictory in light of the inequality~$\|\cdot\|_p \geq \|\cdot\|_q$ for~$1 \leq p \leq q$, which implies that the test statistic gets \emph{smaller} for large powers. Note, however, that this is compensated by the critical values,~cf.~\eqref{eqn:cvpalpha}. We finally emphasize that the asymptotic sizes of the two sequences of tests in Theorem~\ref{thm:incl} do not need to be identical. This is a consequence of the independence of the consistency set of~$\alpha \in (0, 1)$, cf.~Theorem~\ref{thm:conspreal}. 

Returning to the discussion in Section~\ref{sec:optcons}, the monotonicity statement established in Theorem~\ref{thm:incl} immediately raises the question whether there exists a test of asymptotic size~$\alpha \in (0, 1)$ that is consistent against all sequences of alternatives that any~$p$-norm based test is consistent against; that is:
\begin{question}\label{q:existdom}
	Given~$\alpha \in (0, 1)$, does there exist a~$\psi_d \in \mathbb{T}_{\alpha}$ such that~\begin{equation}\label{eqn:better}
		\mathscr{C}(\{p, \kappa_{d}\}) \subseteq \mathscr{C}(\psi_d)
	\end{equation} for every~$p \in (0, \infty)$ and every sequence~$\kappa_d$ such that~$\{p, \kappa_{d}\} \in \mathbb{T}_{\alpha}$?
\end{question}
\begin{remark}\label{q:equivform}
	It follows immediately from Theorem~\ref{thm:incl} that the requirement on~$\psi_d$ in~\eqref{eqn:better} of Question~\ref{q:existdom} can equivalently be replaced by 
	\begin{equation}
		\mathscr{C}(\{p, \kappa_{d}\}) \subsetneqq \mathscr{C}(\psi_d);
	\end{equation}
	that is, if a test as in Question~\ref{q:existdom} exists, then it must not only contain the consistency set of every~$p$-norm based test ($p \in (0, \infty))$) as a subset, but this set inclusion must actually be strict for every~$p$. In this sense, if a test~$\psi_d$ as in Question~\ref{q:existdom} exists, it \emph{strictly} dominates all~$p$-norm based tests in terms of their consistency behavior.
\end{remark}

We will come back to Question~\ref{q:existdom} in later sections. There, it will also be shown that despite the monotonicity established in Theorem~\ref{thm:incl} for~$p \in (0, \infty)$, the supremum-norm based (i.e.,~$p = \infty$) test does not have the property of~$\psi_d$ sought in Question~\ref{q:existdom}. 

In the context of Theorem~\ref{thm:incl} one may ask whether the arrays of alternatives that the~\mbox{$q$-norm} based test but not the~$p$-norm based test is consistent against have particular structural properties. Thus, we next study the structure of
\begin{equation}\label{eqn:mincon}
	\mathscr{C}(\{q, \kappa_{d,q}\}) \setminus \mathscr{C}(\{p, \kappa_{d,p}\}) \quad \text{ for } p < q.
\end{equation}

\subsubsection{Structure of elements of~\eqref{eqn:mincon}}

Let~$2 \leq p < q < \infty$
and fix critical values such that~$\{p, \kappa_{d,p}\}$ and~$\{q, \kappa_{d,q}\}$ have asymptotic sizes in~$(0, 1)$. Then Corollary~\ref{cor:rewrite} shows that~$\bm{\vartheta}$ is an element of the set in~\eqref{eqn:mincon} if and only if (i) there exists a subsequence along which~$d^{-1/2} \left(\|\bm{\theta}_d\|^2_2 \vee \|\bm{\theta}_d\|^p_p\right)$ converges to a real number and (ii) if along any such subsequence it holds that~$d^{-1/2}\|\bm{\theta}_{d}\|_q^q \to \infty$. A concrete example of such an array is~$\bm{\vartheta}=\{(d^{\frac{1}{2p}},0,\hdots,0): d \in \N\}$, i.e., sparse alternatives with a dominating coordinate diverging to~$\infty$ at an appropriate rate. That all elements~$\bm{\vartheta}$ in the set in~\eqref{eqn:mincon} are approximately sparse and highly unbalanced, at least along suitably chosen subsequences, will be shown next; the proof makes use of the observation in Remark~\ref{rem:arb}.
\begin{theorem}\label{thm:cons_prop}
	Let~$0 < p < q < \infty$, and let~$\{p, \kappa_{d,p}\}$ and~$\{q, \kappa_{d,q}\}$ have asymptotic sizes in~$(0, 1)$. Then, for every~$\bm{\vartheta} \in \mathscr{C}(\{q, \kappa_{d,q}\}) \setminus \mathscr{C}(\{p, \kappa_{d,p}\})$ the following holds:
	\begin{enumerate}
		\item For every~$\delta \in (0, \infty)$ there exists a subsequence~$d'$ of~$d$, such that
		\begin{equation}\label{eqn:corstr1}
			\sup_{d'} \frac{\sum_{i=1}^{d'}\mathds{1}\cbr[0]{|\theta_{i,d'}|> \delta}}{\sqrt{d'}} < \infty \quad \text{ and } \quad \max_{1\leq i\leq d'}|\theta_{i,d'}|\to\infty.
		\end{equation}
		\item If~$p \geq 2$, then there exists a subsequence~$d'$ of~$d$, such that
		\begin{equation} \label{eqn:corstr2}
			\sup_{\delta \in (0, \infty)} \delta^p \times \sup_{d'} \frac{\sum_{i=1}^{d'}\mathds{1}\cbr[0]{|\theta_{i,d'}|> \delta}}{\sqrt{d'}} < \infty \quad \text{ and } \quad \max_{1\leq i\leq d'}|\theta_{i,d'}|\to\infty.
		\end{equation}
	\end{enumerate}

\end{theorem}
The first part of Theorem~\ref{thm:cons_prop} shows that those arrays~$\bm{\vartheta} \in \bm{\Theta}$ that are consistently detected by some~$p$-norm based test, but only if~$p$ is chosen large enough, are, along a subsequence, (i) ``approximately sparse,'' in that they have a vanishing fraction of entries larger than a given~$\delta>0$, and (ii) ``highly unbalanced,'' in the sense that~$$\min_{1\leq i\leq d}|\theta_{i,d}|/\max_{1\leq i\leq d}|\theta_{i,d}| \approx 0.$$ 
Furthermore, the second part of Theorem~\ref{thm:cons_prop} shows that in case~$\bm{\vartheta}$ is not consistently detected by a~$p$-norm based test with~$p \geq 2$ (but is detected by some~$q$-norm based test with~$q > p$), a bit more can be said concerning~(i). Namely that then the fraction of entries larger than a given~$\delta>0$ decays (at least) at the order~$\delta^{-p}$, uniformly over a \emph{common} subsequence~$d'$.

\subsection{Case~$p = \infty$}\label{sec:pinf}

We shall now present a result that characterizes the consistency set of tests~$\{\infty, \kappa_d\} \in \mathbb{T}_{\alpha}$,~$\alpha \in (0, 1)$. Structurally, the result is comparable to the one given in Theorem~\ref{thm:conspreal}. Define
\begin{equation}\label{eqn:ginfty}
	g_{\infty}(x) :=
	\begin{cases}
		w(x) & \text{ if } x < 1 \\
		e^{-x^2/2}/x & \text{ else},
	\end{cases}
\end{equation}
for a fixed continuous function~$w: \R \to (0, \infty)$ that satisfies~$w(z) \to \infty$ as~$z \to -\infty$; e.g.,~$w(z) = e^{(z^2-2z)/2}$, in which case~$g_{\infty}$ is also continuous and strictly decreasing. The result is as follows, where we recall that we denote the standard Gaussian cdf by~$\Phi$ and set~$\overline{\Phi} := 1-\Phi$. Its proof is based on a result applicable beyond the Gaussian setting which can be found in Proposition~\ref{prop:abs_noabs} in Appendix~\ref{app:conspinf}.

\begin{theorem}\label{thm:conspinf}
	For~$\kappa_{d}$ such that~$\{\infty, \kappa_d\} \in \mathbb{T}_{\alpha}$,~$\alpha \in (0, 1)$, we have
	\begin{equation}\label{eqn:cpinf}
		\begin{aligned}
			\bm{\vartheta} \in \mathscr{C}(\{\infty,\kappa_{d} \}) & ~ \Leftrightarrow ~ 
			\sum_{i = 1}^d \frac{\overline{\Phi}\left(\mathfrak{c}_d - |\theta_{i,d}|\right)}{\Phi\left(\mathfrak{c}_d - |\theta_{i,d}|\right)} \to \infty
			~ \Leftrightarrow ~
			\sum_{i = 1}^d g_{\infty}\left(\mathfrak{c}_d - |\theta_{i,d}|\right) \to \infty,
		\end{aligned}
	\end{equation}
	where~$\mathfrak{c}_d := \sqrt{2\log(d)} - \frac{\log \log(d)}{2 \sqrt{2\log(d)}}$ for~$d \geq 2$ (and one may set~$\mathfrak{c}_1 := 0$ for completeness).
\end{theorem}

First of all, we note that the sequence~$\mathfrak{c}_d$ in~\eqref{eqn:cpinf} plays the role of centering. This is in contrast to~\eqref{eqn:cpreal}, where the function~$g_p$ is symmetric about~$0$ and  standardization is achieved by multiplying by~$1/\sqrt{d}$.  

In analogy to the discussion after Theorem~\ref{thm:conspreal} (cf.~also Figure~\ref{fig:cont}), we observe that Theorem~\ref{thm:conspinf} shows that the consistency set of a supremum-norm based test is not characterized exclusively by the supremum-norm of the deviation from the null hypothesis. Figure~\ref{fig:cont2} further clarifies this by showing the contour sets for the function~$(x_1, x_2) \mapsto \overline{\Phi}(\mathfrak{c}_2 - |x_1|)/\Phi(\mathfrak{c}_2 - |x_1|) + \overline{\Phi}(\mathfrak{c}_2 - |x_2|)/\Phi(\mathfrak{c}_2 - |x_2|)$. We note that the contour sets deviate from the ones of a supremum-norm, in particular along the ``diagonals'' or close to the origin.
\begin{figure}
	\centering
	\includegraphics[width=0.4\linewidth]{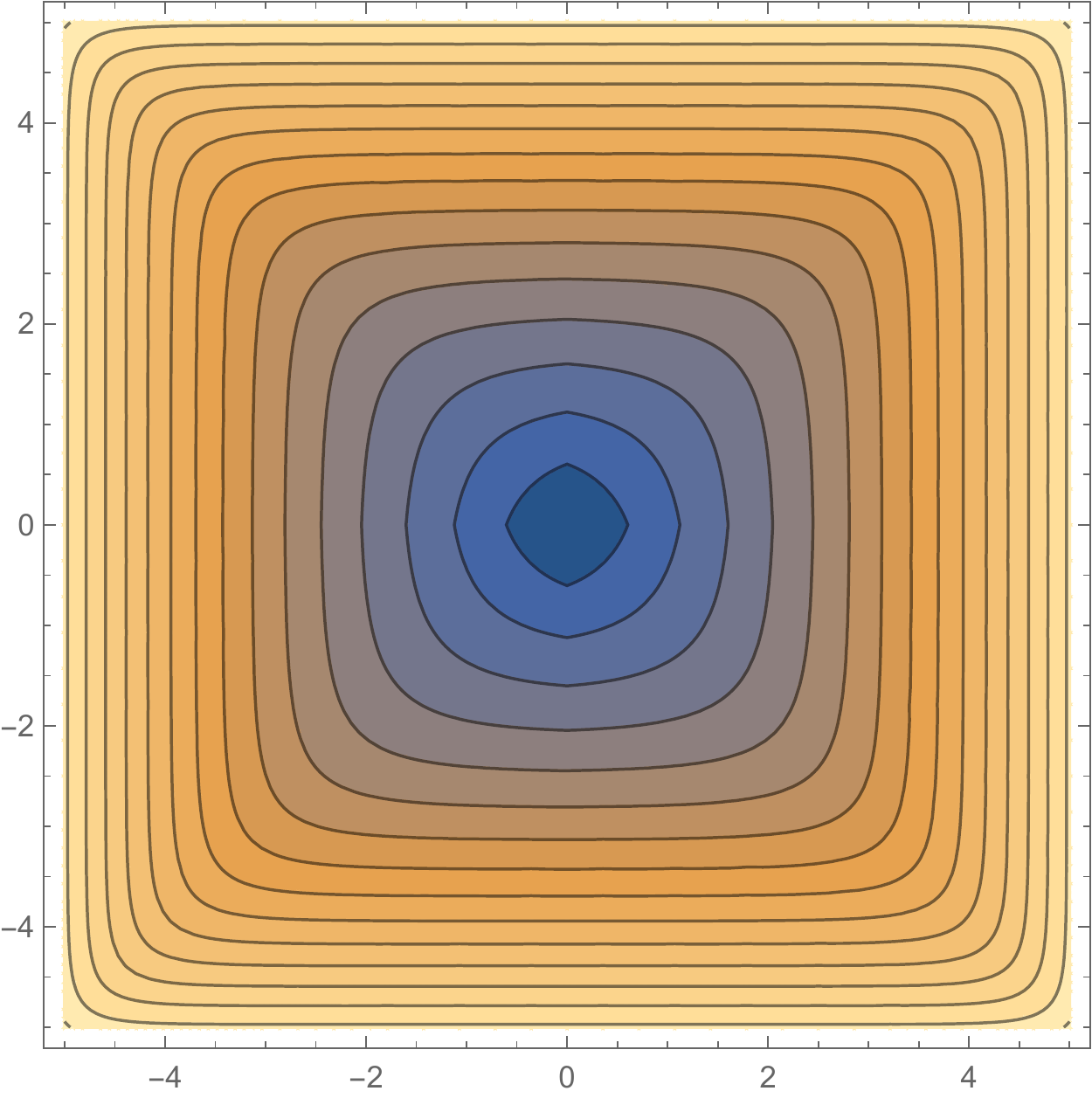}
	\caption{Contour plots of the function characterizing the consistency of the supremum-norm based test for~$d = 2$.}
	\label{fig:cont2}
\end{figure}

Theorem~\ref{thm:conspinf} also reveals that as long as the critical values~$\kappa_d$ are chosen such that the asymptotic size of a supremum-norm based test is in~$(0, 1)$, the consistency set remains the same, i.e., does not depend on the concrete asymptotic size. Although this already follows from the first equivalence in Theorem~\ref{thm:conspinf}, the second equivalence statement has the advantage that it is easier to use in order to check whether a particular~$\bm{\vartheta} \in \bm{\Theta}$ is in~$\mathscr{C}(\{\infty, \kappa_d\})$.

Concerning ``dense'' alternatives, i.e., arrays of the form~$\{\tau_d \bm{\iota}_d : d \in \N\}$ for a real sequence~$\tau_d$ and~$\bm{\iota}_d$ denoting the vector of ones of length~$d$, Theorem~\ref{thm:conspinf} implies that such alternatives are in~$\mathscr{C}(\{\infty,\kappa_{d} \})$ if and only if~$\sqrt{\log(d)}|\tau_d| \to \infty$; see Appendix~\ref{app:supt} for details.  Furthermore, Theorem~\ref{thm:conspinf} (together with the conditions assumed on~$w$) immediately implies the well-known fact that~$\bm{\vartheta} \in \mathscr{C}(\{\infty,\kappa_{d} \})$ if~$\|\bm{\theta}_d\|_{\infty} - \sqrt{2\log(d)} \to \infty$; see \citet[p.~92]{ingster}. 

The discussion in the previous paragraph shows that the supremum-norm based test with asymptotic size~$\alpha$ is consistent against the array~$\{(\sqrt{3\log(d)},0,\hdots,0) : d \in \N\}$. In contrast, it follows from Corollary~\ref{cor:rewrite} that this is not the case for any~$p$-norm based test for~$p\in(0,\infty)$. This echoes the conventional wisdom that supremum-norm based tests ``are more powerful against sparse alternatives.'' On the other hand, Corollary~\ref{cor:rewrite} shows that  for~$p\in(0,\infty)$ all~$p$-norm based tests are consistent against the array~$\{ \bm{\iota}_d/\sqrt{\log(d)} : d \in \N\}$. The above discussion implies, however, that this is not the case for the supremum-norm based test. Hence, there exists no ``ranking'' of the consistency sets of the~$p$-norm based tests for~$p\in(0,\infty)$ and the one of the supremum-norm based test. In particular, the monotonicity result in Theorem~\ref{thm:incl} does \emph{not} extend to~$\infty$. We summarize this observation:
\begin{corollary}\label{cor:supcomp}
	Let~$p \in (0, \infty)$. Then, for~$\{p, \kappa_{d, p}\}$ and~$\{\infty, \kappa_{d, \infty}\}$ with asymptotic sizes in~$(0, 1)$, it holds that~
	\begin{equation}\label{eqn:infcons}
		\mathscr{C}(\{p, \kappa_{d, p}\}) \not \subseteq \mathscr{C}(\{\infty, \kappa_{d, \infty}\}) \quad \text{ and } \quad \mathscr{C}(\{\infty, \kappa_{d, \infty}\}) \not \subseteq \mathscr{C}(\{p, \kappa_{d, p}\}).
	\end{equation}
\end{corollary}
Corollary~\ref{cor:supcomp} implies in particular that supremum-norm based tests do not have the property desired in Question~\ref{q:existdom}.

While the previous section focused on the consistency set for~$p$-norm based tests with~$p$ fixed, we conclude this section with a result concerning the consistency set of~$p$-norm based tests with~$p$ growing with~$d$. The following theorem relates the resulting consistency set to the one of the supremum-norm based test, and will be instrumental in the proof of Theorem~\ref{thm:domp} below.
\begin{theorem}\label{thm:pdconssup}
	Let the sequences~$p_d \in (0, \infty)$ and~$\kappa_d$ be such that the sequence of tests~$\{p_d, \kappa_d\} := \mathds{1}\{\|\cdot \|_{p_d} \geq \kappa_d\}$ is in~$\mathbb{T}_{\alpha}$,~$\alpha \in (0, 1)$. Under the condition that
	\begin{equation}\label{eqn:fasterlogdgauss}
		\liminf_{d \to \infty} \frac{p_d}{\log(d)} > 2,
	\end{equation}
	it holds that
	\begin{equation}
		\mathscr{C}(\{\infty, \kappa_{d, \infty}\}) \subseteq \mathscr{C}(\{p_d, \kappa_d\})
	\end{equation}
	for every sequence of critical values~$\kappa_{d, \infty}$ such that~$\{\infty, \kappa_{d, \infty}\}$ has asymptotic size in~$(0, 1)$.
\end{theorem}
The proof of Theorem~\ref{thm:pdconssup} uses the fact that~$\mathscr{C}(\{\infty, \kappa_d\})$ does not depend on the concrete value of the asymptotic size, as long as this size is in~$(0, 1)$, which we established in Theorem~\ref{thm:conspinf}. Furthermore, it uses a ``squeezing-argument'' based on results in \cite{bogachev2006limit}. That article provides asymptotic approximations for~$p_d$-norm based test statistics under the null and for sequences~$p_d \to \infty$ (under weak assumptions on the sequences, which necessitate the squeezing argument), cf.~also~\cite{schlather2001} and~\cite{janssen2010limit} for related results. A general version of Theorem~\ref{thm:pdconssup} not assuming Gaussianity is provided in Proposition~\ref{prop:pdsup} and Remark~\ref{rem:consnotdep} in Appendix~\ref{app:conspinf}. Although Theorem~\ref{thm:pdconssup} proves that~$p_d$-norm based tests weakly dominate supremum-norm based tests, the discussion in Section~4.5 of~\cite{GF} reveals that tests based on a sequence of powers~$p_d$ with an asymptotic behavior as in the theorem just given do not have good power properties against dense alternatives.

\section{Power enhancements and related procedures}
To maximize power against sparse alternatives, it is often suggested to use a test based on the supremum-norm. At the other extreme, the typical choice maximizing power against dense alternatives is the likelihood-ratio test. Therefore, one could hope that the supremum-norm based test and the likelihood-ratio test together ``suffice'' in the sense that whenever some~$p$-norm based test is consistent against~$\bm{\vartheta}$, the supremum-norm based test or the likelihood-ratio test is consistent against~$\bm{\vartheta}$ as well. That is, one may conjecture that~$\mathscr{C}(\{2, \kappa_{d, 2}\})\cup \mathscr{C}(\{\infty, \kappa_{d, \infty}\})$ contains
\begin{equation}\label{eqn:unionp}
	\bigcup_{p\in(2,\infty)}\mathscr{C}(\{p, \kappa_{d, p}\});
\end{equation}
all critical values being chosen so that the corresponding asymptotic sizes are in~$(0, 1)$. The fact established in Theorem~\ref{thm:cons_prop} that for any~$q \in (2, \infty)$ all elements of the set~$\mathscr{C}(\{q, \kappa_{d,q}\}) \setminus \mathscr{C}(\{2, \kappa_{d,2}\})$ are approximately sparse along a subsequence may make such a conjecture appear even more plausible, since these are the types of alternatives against which the supremum-norm based test has particularly good power properties. Based on Theorems~\ref{thm:conspreal} and~\ref{thm:conspinf} we now show, however, that the likelihood-ratio test and the supremum-norm based test do \emph{not} suffice. Specifically, we show that
\begin{equation}
	\bm{\vartheta}^{\dagger} \in \bigcap_{p \in (2, \infty)} \mathscr{C}(\{p, \kappa_{d,p}\}) \setminus \left[
	\mathscr{C}(\{2, \kappa_{d, 2}\})\cup \mathscr{C}(\{\infty, \kappa_{d, \infty}\})
	\right],
\end{equation} 
where~$\bm{\vartheta}^{\dagger} \in \bm{\Theta}$ is defined as
\begin{equation}\label{eqn:deadenh}
	\bm{\theta}^{\dagger}_d := (\tau_d, \hdots, \tau_d, 0, \hdots, 0) \quad \text{ for } \tau_d=\frac{\sqrt{2\log(d)}}{\log\log(d)},
\end{equation}
and where the number of non-zero entries of~$\bm{\theta}_d^{\dagger}$ is~$\lceil \sqrt{d}/\log(d)\rceil$ (at least for~$d$ large enough).
The statement is as follows. 

\begin{theorem}\label{thm:2-p-inf}
	For every~$p \in (2, \infty)$ and any sequences of critical values such that~$\{2, \kappa_{d, 2}\} \in \mathbb{T}_{\alpha_2}$,~$\{\infty, \kappa_{d, \infty}\}\in \mathbb{T}_{\alpha_{\infty}}$, and~$\{p, \kappa_{d, p}\} \in \mathbb{T}_{\alpha_p}$, the asymptotic sizes all being in~$(0, 1)$, it holds that
	\begin{equation}\label{eqn:2-p-inf}
		\bm{\vartheta}^{\dagger} \in \mathscr{C}(\{p, \kappa_{d, p}\}) \quad \text{ but } \quad \bm{\vartheta}^{\dagger} \notin \mathscr{C}(\{2, \kappa_{d, 2}\})\cup \mathscr{C}(\{\infty, \kappa_{d, \infty}\}).
	\end{equation}
	Furthermore, the following convergences, stronger than the statement on the right in~\eqref{eqn:2-p-inf}, hold
	\begin{equation}\label{eq:strongerstatement}
		\mathbb{P}\left(
		\|\bm{\theta}_d^{\dagger} + \bm{\varepsilon}_d \|_2 \geq \kappa_{d,2}
		\right) \to \alpha_2 \quad \text{ and } \quad \mathbb{P}\left(
		\|\bm{\theta}_d^{\dagger} + \bm{\varepsilon}_d \|_{\infty} \geq \kappa_{d,\infty}
		\right) \to \alpha_{\infty}.
	\end{equation}
\end{theorem}

Theorem~\ref{thm:2-p-inf} shows that there exist arrays of alternatives against which (i) \emph{all}~$p$-norm based tests are consistent for~$p \in (2,\infty)$, and (ii) the~$2$- and supremum-norm based
tests have asymptotic power equaling their asymptotic size.

As a corollary to this observation, we can now show that (under reasonable assumptions concerning the asymptotic sizes of the tests involved) it is impossible to obtain a test with the desired properties in Question~\ref{q:existdom} by ``combining'' the likelihood-ratio test and the supremum-norm based test into a test~$\varphi_d$ such that~$\mathscr{C}(\varphi_d)$ contains the union in~\eqref{eqn:unionp}. Here we say that a test~$\varphi_d$ is a combination of the likelihood-ratio test and the supremum-norm based test if there exist sequences of critical values~$\kappa_{d,2}$ and~$\kappa_{d,\infty}$ such that
\begin{equation}\label{eqn:PEup}
	\varphi_d \leq  \max\left(\{2, \kappa_{d,2}\}_d, \{\infty, \kappa_{d,\infty}\}_d\right) \quad \text{ for every } d \in \N.
\end{equation}
A remarkable special case of such a combination procedure was investigated in~\cite{fan2015}, where it was suggested to improve the likelihood-ratio test by the supremum-norm based test in the context of their power enhancement principle. Their construction satisfies Equation~\eqref{eqn:PEup} for critical values~$\kappa_{d,2}$ such that the corresponding likelihood-ratio test has asymptotic size~$\alpha \in (0, 1)$, and~$\kappa_{d,\infty}$ such that the corresponding supremum-norm based test has asymptotic size~$0$. Our result concerning such procedures is as follows.
\begin{corollary}\label{cor:PE}
	If~$\{2, \kappa_{d,2}\} \in \mathbb{T}_{\alpha_2}$ and~$\{\infty, \kappa_{d,\infty}\} \in \mathbb{T}_{\alpha_{\infty}}$ with~$0\leq \alpha_2 + \alpha_{\infty} < 1$, then~$\bm{\vartheta}^{\dagger} \notin \mathscr{C}(\varphi_d)$ for every~$\varphi_d \in \mathbb{T}$ that satisfies Equation~\eqref{eqn:PEup}.
\end{corollary}
Thus, any combination of the likelihood-ratio test and the supremum-norm based test is inconsistent against~$\bm{\vartheta}^{\dagger}$, and hence cannot be used to answer Question~\ref{q:existdom} in the affirmative.

\section{Tests that dominate all~$p$-norm based tests}\label{sec:dom}

We now answer Question~\ref{q:existdom} in the affirmative by constructing sequences of tests~$\psi_d$ that are consistent against any deviation from the null
that some~$p$-norm based test is consistent against. A similar statement also holds in the non-Gaussian, cf.~Appendix~\ref{app:existdomgen}.

The idea underlying our construction is related to the test proposed in \cite{xu2016adaptive}, who went beyond the classical power enhancement principle and suggested to combine a \emph{fixed} number of $p$-norm based tests. Importantly, however, we intend to combine \emph{all} $p$-norm based tests into a better test. In principle, this would require us to combine an \emph{uncountable} amount of tests indexed by~$p$ in the non-compact set~$(0, \infty)$, which seems impossible at first sight. But the monotonicity result in Theorem~\ref{thm:incl} can be used as a ``discretization device,'' allowing us to get all consistency sets corresponding to powers~$p \in (0, \infty)$ by suitably combining a finite, but in~$d$ \emph{increasing}, number~$m_d$, say, of powers~$p_1, \hdots, p_{m_d}$, say, cf.~Equation~\eqref{eq:max_test} below.  Besides the monotonicity result, a crucial aspect exploited in the proof is the independence of the consistency set of~$p$-norm based tests of their asymptotic size, cf.~Theorems~\ref{thm:conspreal} and~\ref{thm:conspinf} and the ensuing discussions. 
\begin{theorem}\label{thm:domp}
	Let~$p_d$ be a strictly increasing and unbounded sequence in~$(0, \infty)$ and let~$m_d$ be a non-decreasing and unbounded sequence in~$\N$. Choose~$\alpha \in (0, 1)$ and fix an array~$$\mathcal{A}  = \left\{\alpha_{j,d} \in (0, 1): d \in \N,~ j = 1, \hdots, m_d\right\}$$ such that, for~$\mathbb{M} \subseteq \N$ unbounded, it holds that
	\begin{equation}\label{eqn:Aarray}
		\sum_{j = 1}^{m_d} \alpha_{j,d} = \alpha \text{ for every } d\in \N,~~\lim_{d \to \infty} \alpha_{m_d, d} > 0,~\text{ and }~ \lim_{d \to \infty} \alpha_{j,d}  > 0 \text{ for every } j \in \mathbb{M},
	\end{equation}
	where the conditions implicitly impose the existence of the respective limits. For every~$d \in \N$ and every~$j = 1, \hdots, m_d$, choose~$\kappa_{j, d} > 0$ and~$c_d \in (0, 1]$ such that
	\begin{equation}\label{eqn:exactsize}
		\mathbb{P}\left( \|\bm{\varepsilon}_d \|_{p_j} \geq \kappa_{j, d} \right) = \alpha_{j,d} \quad \text{ and } \quad \mathbb{E}\left(\psi_d(\bm{\varepsilon}_d)\right) = \alpha,
	\end{equation}
	where
	\begin{equation}\label{eq:max_test}
		\psi_d(\cdot) := \mathds{1}\left\{ \max_{j = 1, \hdots, m_d} \kappa_{j,d}^{-1}
		\| \cdot \|_{p_j} \geq c_d
		\right\}.
	\end{equation}
	Then, the following statements hold:
	\begin{enumerate}
		\item The sequence of tests~$\psi_d$ has the property requested in Question~\ref{q:existdom}, that is, it has asymptotic size~$\alpha$ and satisfies
		\begin{equation}\label{eqn:better_2} 
			\mathscr{C}(\{p, \kappa_{d}\}) \subseteq \mathscr{C}(\psi_d), \text{ for every } p \in (0, \infty) \text{ and every } \{p, \kappa_d\} \in \mathbb{T}_{\alpha}.
		\end{equation}
		\item Under the additional condition that
		\begin{equation}\label{eqn:powslow}
			\liminf_{d \to \infty} \frac{p_{m_d}}{\log(d)} > 2,
		\end{equation}
		it furthermore holds that~$$\mathscr{C}(\{\infty, \kappa_{d}\}) \subseteq \mathscr{C}(\psi_d), \text{ for every } \{\infty, \kappa_d\} \in \mathbb{T}_{\alpha}.$$
	\end{enumerate}  
\end{theorem}

\begin{remark}\label{rem:domp}
	Let us note first that for every~$d \in \N$ the acceptance region of the test~$\psi_d$ in~\eqref{eq:max_test} is a symmetric convex set if~$p_1 \geq 1$. It then follows that~$\psi_d$ is admissible and unbiased (cf.~\cite{birnbaum1955} and~\cite{stein1956}, and~\cite{anderson1955integral}).
	
	Second, in case~\eqref{eqn:powslow} is satisfied, the test~$\psi_d$ does not only have the property of being consistent against every array of alternatives that any~$p$-norm based test with~$p \in (0, \infty)$ and asymptotic size in~$(0, 1)$ is consistent against. It moreover also dominates any supremum-norm based test with asymptotic size~$\alpha$. The proof of this property crucially relies on Theorem~\ref{thm:pdconssup}.
	
	Third, by Theorems~\ref{thm:conspreal} and~\ref{thm:conspinf}, the consistency sets of~$p$-norm based tests with asymptotic size in~$(0, 1)$ do not depend on the actual value of the asymptotic size. It therefore follows that an apparently stronger (but equivalent) version of Theorem~\ref{thm:domp} is true, in which~$\mathbb{T}_{\alpha}$ in Parts~1 and~2 is replaced by~$\bigcup_{\tilde{\alpha} \in (0, 1)}\mathbb{T}_{\tilde{\alpha}}$.
\end{remark}

The sequence~$m_d$ regulates the number of norms the test in Equation~\eqref{eq:max_test} is based on, whereas the sequence~$p_d$ determines the concrete powers~$p \in (0, \infty)$ used in the construction. The condition in Equation~\eqref{eqn:powslow} requires that the maximal power used in the test~\eqref{eq:max_test} grows sufficiently quickly in~$d$ to guarantee that the supremum-norm based test is dominated by making use of Theorem~\ref{thm:pdconssup}.

Intuitively, the role of the array~$\mathcal{A}$ is to regulate the sizes of the ``individual'' tests involved in the construction in Equation~\eqref{eq:max_test}. Furthermore,~$c_d$ is a correction term guaranteeing that the test in that display has size exactly equal to~$\alpha$ for every sample size~$d$. The choice of~$c_d = 1$ would in general only lead to a test of size not greater than~$\alpha$. Thus, working with a smaller~$c_d$ leads to higher power compared to~$c_d = 1$, which would correspond to an overly conservative test. The critical values~$\kappa_{j,d}$ and the multiplier~$c_d$ as in Theorem~\ref{thm:domp} can be found by a simple line search. The probabilities that need to be obtained in such computations can be approximated numerically.  This is computationally relatively cheap, because the number~$m_d$ of tests involved can be chosen to grow slowly in~$d$, cf.~Example~\ref{ex:spec} below. 
\begin{remark}\label{rem:array}
	Given sequences~$p_d$ and~$m_d$ as in Theorem~\ref{thm:domp}, concrete choices of arrays~$\mathcal{A}$ satisfying~$\sum_{j = 1}^{m_d} \alpha_{j,d} = \alpha$ for every~$d \in \N$, and Equation~\eqref{eqn:Aarray} in Theorem~\ref{thm:domp} with~$\mathbb{M} = \mathbb{N}$, can easily be obtained from any probability mass function~$\delta_j > 0$,~$j \in \N$, and~$\gamma \in (0, 1)$, via the transformation (for~$d$ large enough so that~$m_d \geq 2$)
	\begin{equation}\label{eqn:arrayspec}
		\alpha_{j,d} = 
		\begin{cases}
			\gamma \alpha \frac{\delta_j}{\sum_{i = 1}^{m_d-1}
				\delta_i} & \text{ for } j = 1, \hdots, m_d-1, \\
			(1-\gamma)\alpha & \text{ for } j = m_d.
		\end{cases}
	\end{equation}
\end{remark}

\begin{example}\label{ex:spec}
	A specific example of a test as in Theorem~\ref{thm:domp} is given next. The test we discuss commences with the likelihood-ratio test in the sense that~$p_1 = 2$ (this will be important in Example~\ref{ex:closeex} below). Choose~$\mathcal{A}$ as in Remark~\ref{rem:array} with the geometric probability mass function~$\delta_j = \delta_0(1-\delta_0)^{j-1}$ for some~$\delta_0 \in (0, 1)$ and~$\gamma \in (0, 1)$. Choose~$m_d = \lceil 3\log(d) \rceil +1$, and let~$p_d = d+1$, so that~\eqref{eqn:powslow} is satisfied. For critical values~$\kappa_{j, d}$ as defined in~\eqref{eqn:exactsize} (but based on the concrete array, and the concrete sequences~$p_d$ and~$m_d$ just defined) and the corresponding~$c_d \leq 1$, Theorem~\ref{thm:domp} shows that the corresponding sequence of tests
	satisfies the property sought for in Question~\ref{q:existdom} (and also dominates any supremum-norm based test).
\end{example}

The array~$\mathcal{A}$ that the test~$\psi_d$ in~\eqref{eq:max_test} is based on also regulates the uniform ``closeness'' of the asymptotic power function of~$\psi_d$ to the power function of each~$p_j$-norm based test involved in its construction. The following result quantifies this relation. How the result can be used to mimic the asymptotic power properties of a specific~$p$-norm based test when working with a test~$\psi_d$ will be discussed subsequently.
\begin{theorem}\label{thm:opt_sum}
	In the context of Theorem~\ref{thm:domp}, fix~$j \in \mathbb{M}$ and set 
	\begin{equation}\label{eqn:alphabardef}
		\lim_{d \to \infty} \alpha_{j, d} =: \underline{\alpha}_j > 0.
	\end{equation}
	For any sequence of critical values~$\kappa_{d}$ such that~$\{p_j, \kappa_{d}\} \in \mathbb{T}_{\alpha}$, the sequence of tests~$\psi_d$ as defined in~\eqref{eq:max_test} satisfies
	\begin{align}\label{eq:powerfunction}
		\limsup_{d\to\infty}\sup_{\bm{\theta}_d\in\R^d} \left[\P\del[1]{\|\bm{\theta}_d+\bm{\eps}_d\|_{p_j} \geq \kappa_{d}}-\E(\psi_d(\bm{\theta}_d+\bm{\eps}_d))\right] \leq \frac{
			\Phi^{-1}(1-\underline{\alpha}_j) - \Phi^{-1}(1-\alpha)}{\sqrt{2\pi}},
	\end{align}
	and
	\begin{equation}\label{eq:powerfunction2}
		\lim_{d\to\infty}\sup_{\bm{\theta}_d\in\R^d} \left[\E(\psi_d(\bm{\theta}_d+\bm{\eps}_d)) - \P\del[1]{\|\bm{\theta}_d+\bm{\eps}_d\|_{p_j}
			\geq \kappa_{d}}\right] = 1-\alpha.
	\end{equation}
\end{theorem}

Recall from Equation~\eqref{eqn:better_2} that the asymptotic power of~$\psi_d$ is~$1$ whenever that of some~$p$-norm based test is. The inequality in~\eqref{eq:powerfunction} sheds further light on the asymptotic power function of~$\psi_d$ by establishing that it is \emph{nowhere} much below that of~$\{p_j, \kappa_d\}$ for an index~$j \in \mathbb{M}$ such that~$\underline{\alpha}_j \approx \alpha$, as the upper bound in that inequality is then approximately~$0$. Thus, there are no arrays of alternatives for which much is lost by using~$\psi_d$ instead of such a~$\{p_j, \kappa_d\}$. What is more, Equation~\eqref{eq:powerfunction2} (and its proof) shows that much can be gained by using~$\psi_d$, as arrays of alternatives exist against which~$\psi_d$ is consistent, but against which~$\{p_j, \kappa_d\}$ has asymptotic power equal to its size.

In the following example we illustrate how this reasoning can be incorporated in the construction of~$\psi_d$.

\begin{example}\label{ex:closeex}
	In case one has reasons to favor a specific~$p$-norm based test, e.g., the likelihood-ratio test corresponding to~$p = 2$, but does not want to abandon the idea of using a test~$\psi_d$ that dominates all~$p$-norm based tests, one can decide on the following compromise, which is possible as a consequence of Theorems~\ref{thm:domp} and~\ref{thm:opt_sum}: choose the components in the construction of~$\psi_d$ in such a way that~$\psi_d$ dominates all~$p$-norm based tests \emph{and} such that the power function of the obtained test~$\psi_d$ is \emph{everywhere} at most slightly smaller than that of the preferred~$p$-norm based test. 
	
	To see how this can be achieved, we focus on the case where~$p = 2$. Let~$\{2, \kappa_d\}$ be a sequence of likelihood-ratio tests with asymptotic size~$\alpha \in (0, 1)$. To obtain a test~$\psi_d$ as in Theorem~\ref{thm:domp} and whose power is (asymptotically) nowhere much less than that of~$\{2, \kappa_d\}$, we can reconsider the test~$\psi_d$ constructed in Example~\ref{ex:spec}. Note that~$p_1 = 2$. Furthermore,~$\underline{\alpha}_1 =  \alpha \gamma \delta_0$, cf.~Equation~\eqref{eqn:alphabardef}. It hence follows from Theorem~\ref{thm:opt_sum} (with~$j = 1$) that
	\begin{align*}
		\limsup_{d\to\infty}\sup_{\bm{\theta}_d\in\R^d} \left[\P\del[1]{\|\bm{\theta}_d+\bm{\eps}_d\|_{2} \geq \kappa_{d}}-\E(\psi_d(\bm{\theta}_d+\bm{\eps}_d))\right] \leq \frac{
			\Phi^{-1}(1-\alpha \gamma \delta_0) - \Phi^{-1}(1-\alpha) }{\sqrt{2\pi}}.
	\end{align*}
	The upper bound can be made arbitrarily close to~$0$ by choosing~$\delta_0$ and~$\gamma$ close to~$1$. At the same time, and in contrast to the likelihood-ratio test~$\{2, \kappa_d\}$,~$\psi_d$ has the favorable property of being consistent against every alternative that some~$p$-norm based test is consistent against.
\end{example}

\section{Numerical results}\label{sec:num}

To investigate the non-asymptotic properties of the tests under consideration, we provide a numerical comparison of the power of $p$-norm based tests (we consider~$p = 1, 2, 3, 4$ and~$p = \infty$) and a test~$\psi_d$, say, corresponding to the construction in Theorem~\ref{thm:domp}. We set~$\alpha = 0.05$ and consider the dimensions~$d = 50.000$ and~$d = 250.000$. The specific version of~$\psi_d$ used in the computations is the following:
\begin{enumerate}
	\item We employed~$m_d = \lceil \log(\log(d^6)) \rceil$, which equals~$5$ for both dimensions considered.
	\item We used $p_d = \exp(d-1) + 1$, i.e., the test was based on~$p_1 = 2$, $p_2 = e + 1$, $p_3 = e^2 +1$,~$p_4  =  e^3 + 1$, and~$p_5 = p_{m_d} =e^4 + 1$.
	\item To generate~$\mathcal{A}$, we used the approach in Remark~\ref{rem:array} with~$\gamma = 1/2$ and~$\delta_j$ ($j = 1, \hdots, m_d - 1$) the probability mass function from a geometric distribution with success parameter~$1/2$. For both dimensions considered, this results in roughly~$\alpha_{1, d} = 0.013$,~$\alpha_{2, d} = 0.007$,~$\alpha_{3, d} = 0.003$,~$\alpha_{4, d} = 0.002$~ and~$\alpha_{5, d} =  \alpha_{m_d, d} = 0.025$.
\end{enumerate}
Note that with this choice Equation~\eqref{eqn:powslow} holds and that~$m_d$, i.e., the number of exponents used in the construction, grows slowly with~$d$ which is numerically favorable.  As a consequence of Theorem~\ref{thm:opt_sum}, and since~$p_1 = 2$, we obtain that the asymptotic power of the Euclidean norm based test can nowhere exceed the asymptotic power of~$\psi_d$ by more than~$(\Phi^{-1}(0.9875) - \Phi^{-1}(0.95))/\sqrt{2\pi} \approx 0.24$. The critical values for the tests considered were obtained through Monte Carlo (with~$50.000$ replications throughout). 

We compare the power functions (determined via Monte Carlo using $1.000$ replications) of each of the above-mentioned tests against three types of alternatives: (i) dense vectors~$\bm{\theta}_d$, i.e., vectors of the form~$a \times (1, \hdots, 1)$; (ii) semi-sparse vectors~$\bm{\theta}_d$, i.e., vectors of the form~$a \times \bm{\theta}_d^{\dagger}$ for~$\bm{\theta}_d^{\dagger}$  as defined in Equation~\eqref{eqn:deadenh}; and (iii) sparse vectors~$\bm{\theta}_d$, i.e., vectors of the form~$a \times (1, 0, \hdots, 0)$. The power functions are provided in Figure~\ref{fig:plot} and are plotted against~$a$. The results show that the choice of the exponent in a~$p$-norm based test has a strong effect on the type of signal one has high power against. While the supremum-norm based test performs well for sparse and semi-sparse signals, lower exponents perform better for dense signals, and vice versa. The combination procedure~$\psi_d$ strikes a balance between this extreme difference in performance. It performs best in the semi-sparse setup, while it is very competitive with the best performing tests in the other setups, particularly so for non-centrality parameters where power is high, echoing our theoretical results.
\begin{figure}
	\centering
	\includegraphics[width=\linewidth]{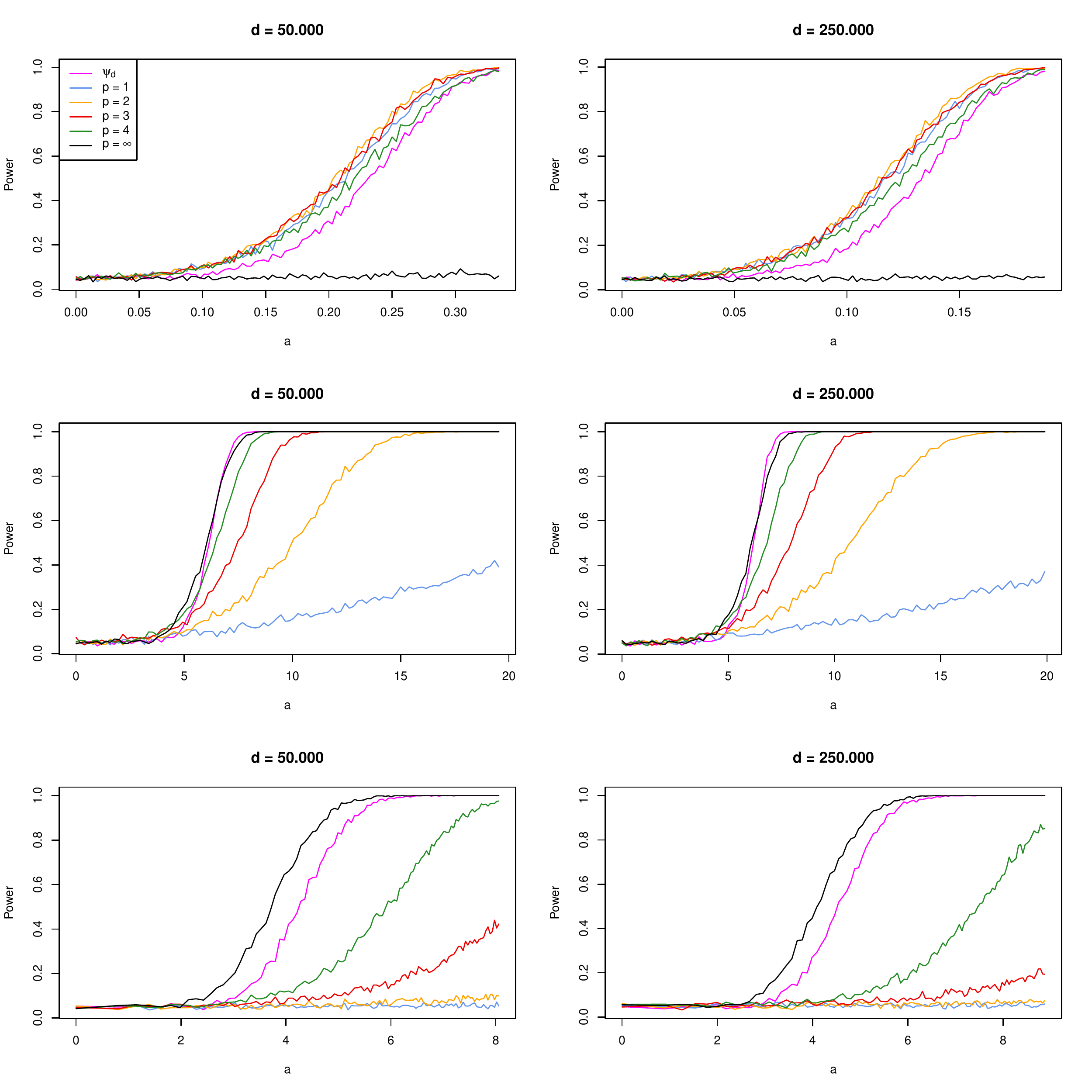}
	\caption{Power against dense alternatives (first row), semi-sparse alternatives (second row), and sparse alternatives (last row).}
	\label{fig:plot}
\end{figure}

\section{Conclusion}

Combining two tests to enhance power is the underlying paradigm of the power enhancement principle recently put forward by~\cite{fan2015}. There, it was suggested to combine tests based on the Euclidean and supremum-norms. In the present article, we have characterized the consistency sets of~$p$-norm based tests for all~$p \in (0, \infty]$. Our characterizations allowed us to reveal an unexpected monotonicity relation of the consistency sets, and, somewhat surprisingly, to asymptotically dominate \emph{all} these tests in terms of set inclusion of consistency sets, but also in a minimax sense. This was achieved by suitably combining a finite number of tests that grows with the dimension of the testing problem.

Even though the Gaussian sequence model is a prototypical framework for high-dimensional inference and results in that setup carry over to many other settings at least on a conceptual level, suitable generalizations of our results continue to hold also in non-Gaussian settings, as we show in detail in Appendix~\ref{app:AUX}.

\bibliographystyle{chicago}% (uses file ''ims.bst'')
\bibliography{refs}% expects file ''refs.bib''
\newpage
\appendix

\begin{center}
\Huge{Appendix}
\end{center}

\medskip

Appendix~\ref{app:MT} contains the proofs of all results in the main body of the paper, Appendix~\ref{sec:minimax} contains results on minimax adaptive testing the proofs of which are contained in Appendix~\ref{app:minimaxproofs}, whereas Appendix~\ref{app:AUX} contains results in a more general non-Gaussian framework.

\section{Proofs of results in the main part of the paper}\label{app:MT}

Throughout Appendix~\ref{app:MT}, all assumptions imposed in the main part of the paper concerning the distribution of the random variables~$\varepsilon_i$ are maintained; that is, the~$\varepsilon_i$ are~i.i.d.~standard normal. 

\subsection{Proof of Theorem~\ref{thm:impo}}

We establish a slightly stronger result than~Theorem~\ref{thm:impo}, showing that the sequence of tests~$\overline{\psi}_d$ in Theorem~\ref{thm:impo} can actually be constructed such that the power of the sequence of tests~$\psi_d$ is improved against a~$1$-sparse vector of alternatives. This is of additional interest, as it implies that the result pertains even in situations where the testing problem is reduced to~$1$-sparse parameter spaces, i.e., when instead of considering the parameter space~$\bm{\Theta} = \bigtimes_{d = 1}^{\infty} \R^d$ one only considers all arrays~$\bm{\vartheta}$ in~$$\bm{\Sigma}_1 := \bigtimes_{d = 1}^{\infty} \bigcup_{i = 1}^d \left\{ \theta \bm{e}_i(d) : \theta \in \R \right\} \subseteq \bm{\Theta},$$
where~$\bm{e}_i(d)$ denotes the~$i$-th element of the canonical basis of~$\R^d$. Note that for every fixed~$d \in \N$, every element of~$\bigcup_{i = 1}^d \left\{ \theta \bm{e}_i(d) : \theta \in \R \right\}$ has at most one non-zero coordinate, i.e., is~$1$-sparse.

\begin{theorem}\label{thm:imposparse}
	Let~$\alpha \in (0, 1)$. For every~$\psi_d \in \mathbb{T}_{\alpha}$ there exists a~$\overline{\psi}_d \in \mathbb{T}_{\alpha}$, such that:
	\begin{enumerate}
		\item~$\overline{\psi}_d \geq \psi_d$ holds for every~$d \in \N$, guaranteeing that~$\overline{\psi}_d$ has uniformly non-inferior asymptotic power compared to~$\psi_d$; 
		\item~$\overline{\psi}_d$ is consistent against an array~$\bm{\vartheta} \in \bm{\Sigma}_1$ against which~$\psi_d$ has asymptotic power at most~$\alpha$. 
	\end{enumerate}
	In particular it holds that~$\mathscr{C}(\psi_d) \subsetneqq \mathscr{C}(\overline{\psi}_d)$.
\end{theorem}

The proof is a special case of the example discussed in Section~1.1 of~\cite{kp1}. The simplification is due to no sufficiency argument being necessary in the present framework. We provide a complete argument for the convenience of the reader.
\begin{proof}
	The first part of the proof is based on an argument in Section~3.4.2 of \cite{ingster}: For every~$d \in \N$, set~$\mathbb{Q}_{d, 0} := \mathbb{N}(\bm{0}_d, \bm{I}_{d})$ (the distribution of~$\bm{y}_d$ under the null) and define the mixture~$\mathbb{Q}_d = d^{-1} \sum_{i = 1}^{d} \mathbb{Q}_{d,i}$, where~$\mathbb{Q}_{d,i} := \mathbb{N}(a_d \bm{e}_{i}(d), \bm{I}_{d})$ (the distribution of~$\bm{y}_d$ under the 1-sparse alternative~$\bm{\theta}_d = a_d \bm{e}_{i}(d)$) for~$a_d := \sqrt{\log(d)/2}$. The likelihood-ratio statistic of~$\mathbb{Q}_d$ w.r.t.~$\mathbb{Q}_{d,0}$ is given by $L_d(z_1, \hdots, z_d) = d^{-1}\sum_{i = 1}^{d} e^{a_d z_i -a_d^2/2}$. Denote the expectation operators w.r.t.~$\mathbb{Q}_d$ and~$\mathbb{Q}_{d,i}$ by~$\mathbb{E}^Q_{d}$ and~$\mathbb{E}^Q_{d,i}$, respectively. It holds that
	\begin{equation}\label{eq:gauss1}
		\big|\mathbb{E}_{d,0}^Q(\psi_d) - d^{-1} \sum_{i = 1}^{d} \mathbb{E}_{d,i}^Q(\psi_d)\big|^2  
		= \left|\mathbb{E}^Q_{d,0}(\psi_d(1 - L_d))\right|^2 \leq \mathbb{E}^Q_{d,0}\left((1 - L_d)^2\right) = \mathbb{E}^Q_{d,0}\left(L_d^2\right) -1,
	\end{equation}
	where we used~$d^{-1} \sum_{i = 1}^{d} \mathbb{E}_{d,i}^Q(\psi_d) = \mathbb{E}_{d}^Q(\psi_d) = \mathbb{E}_{d,0}^Q(\psi_d L_d)$, Jensen's inequality and $\mathbb{E}_{d,0}^Q(L_d) = 1$. From the moment-generating-function of a normal distribution we obtain
	\begin{equation}\label{eq:gauss2}
		\mathbb{E}^Q_{d,0}\left(L_d^2\right) -1 = d^{-2} \sum_{i = 1}^{d} \sum_{j = 1}^{d} \mathbb{E}^Q_{d,0}\left(e^{a_d(z_i+z_j) - a_d^2}\right) - 1 \leq d^{-1/2} \to 0.
	\end{equation}
	Given this argument, we now see that the sequence to the left in~\eqref{eq:gauss1} converges to~$0$, and thus the existence of an array~$\bm{\vartheta} = \{ a_d \bm{e}_{i(d)}(d): d \in \N\} \in \bm{\Sigma}_1$ follows, against which~$\psi_d$ has asymptotic power at most~$\alpha$. Next, define (for every~$d \in \N$) the test~$\nu_d = \mathds{1} \{y \in \R^d: |\bm{e}_{i(d)}'y| \geq \sqrt{a_d}\}$. Obviously, its asymptotic size equals~$0$, and it is consistent against~$\bm{\vartheta}$. Hence~$\overline{\psi}_d := \min(\psi_d + \nu_d, 1)$ has the properties required in the theorem.
\end{proof}

\begin{remark}
	Theorem~\ref{thm:imposparse} is an ``existence result'' in the sense that it shows that there exists a sequence of tests~$\overline{\psi}_d$ with certain properties. Inspection of the proof just given, however, shows how such a sequence of tests~$\overline{\psi}_d$ can actually be obtained by ``enhancing'' the test~$\psi_d$ with a test~$\nu_d$ as defined in course of the	 proof. Note that the sequence~$i(d)$ needed in such a construction can in principle be found numerically by choosing~$i(d)$ for every~$d \in \N$ as an index~$i \in \{1, \hdots, d\}$ such that the test~$\psi_d$ has minimal power against~$\mathbb{N}(a_d \bm{e}_{i}(d), \bm{I}_{d})$ (note that the~$d$ values of the power function needed can be obtained by Monte-Carlo simulations). The same remark applies a fortiori to Theorem~\ref{thm:impo}.
\end{remark}

\subsection{$p$-norm based tests with~$p \in (0, \infty)$}

\subsubsection{Proof of Theorem~\ref{thm:conspreal}}

By Remark~\ref{rem:pfindist}, Assumption~\ref{ass:pfindist} holds for~$F = \Phi$, so that the theorem is an immediate consequence of Theorem~\ref{thm:consprealGEN} (since~$\mathbb{E}(|\varepsilon|^q) < \infty$ then holds for every~$q \in (0, \infty)$).

\subsubsection{Proof of Corollary~\ref{cor:rewrite}}

To show~\eqref{eqn:equi}, recalling the definition of~$g_p$ from~\eqref{eqn:rhodef}, we just combine Theorem~\ref{thm:conspreal} with~$\frac{1}{2}(|\cdot|^p + |\cdot|^2) \leq g_p(\cdot) \leq |\cdot|^p + |\cdot|^2$ for~$p \in [2, \infty)$. Similarly, 
to show~\eqref{eqn:equi2}, we use Theorem~\ref{thm:conspreal} together with~$g_p(\cdot) \leq |\cdot|^p$ and~$g_p(\cdot) \leq |\cdot|^2$ for~$p \in (0, 2)$. To verify the remaining statement, let~$\bm{\vartheta} \in \bm{\Theta}$ be defined via
\begin{equation*}
	\theta_{i,d} = 
	\begin{cases}
		d^{-1/4} & \text{ for } i = 1, \hdots, d-1, \\
		d^{1/(2p)} & \text{ for } i = d.
	\end{cases}
\end{equation*}
Then~$\|\bm{\theta}_d\|^2_2/\sqrt{d} \geq d^{\frac{1}{p} - \frac{1}{2}} \to \infty$ and~$\|\bm{\theta}_d\|^p_p/\sqrt{d} \geq \frac{d-1}{\sqrt{d}} d^{-p/4} \to \infty$. However,~$$d^{-1/2} \sum_{i = 1}^d g_p(\theta_{i,d}) = \frac{d-1}{d} + 1 \leq 2,$$ and we conclude~$\bm{\vartheta} \notin \mathscr{C}(\{p, \kappa_d\})$ with Theorem~\ref{thm:conspreal}.

\subsubsection{Proof of Theorem~\ref{thm:incl}}

Just observe that~$0 < p < q < \infty$ implies~$g_p \leq g_q$, which by Theorem~\ref{thm:conspreal} delivers~$\mathscr{C}(\{p, \kappa_{d, p}\}) \subseteq \mathscr{C}(\{q, \kappa_{d, q}\})$. That the inclusion is strict follows from~$$\bm{\vartheta} := \{(d^{1/(2p)}, 0, \hdots, 0): d \in \N\} \in \mathscr{C}(\{q, \kappa_{d, q}\}) \setminus \mathscr{C}(\{p, \kappa_{d, p}\}),$$
which is easily seen using the equivalence from Theorem~\ref{thm:conspreal}.

\subsubsection{Proof of Theorem~\ref{thm:cons_prop}}
Fix~$p<q$, $\{p, \kappa_{d,p}\}$ and~$\{q, \kappa_{d,q}\}$ as in the statement of this theorem. Let~$\bm{\vartheta}\in\mathscr{C}(\{q, \kappa_{d,q}\})\setminus\mathscr{C}(\{p, \kappa_{d,p}\})$, which is possible by Theorem~\ref{thm:incl}. From Theorem~\ref{thm:conspreal} and Remark~\ref{rem:arb} we conclude that for every~$\delta \in (0, \infty)$ it holds that
\begin{align}\label{eq:qnormaux}
	G_{d,p}(\delta) := \frac{\sum_{i=1}^d\theta_{i,d}^2\mathds{1}\cbr[0]{|\theta_{i,d}|\leq \delta}+\sum_{i=1}^d|\theta_{i,d}|^p\mathds{1}\cbr[0]{|\theta_{i,d}|>\delta}}{\sqrt{d}}
	=
	\frac{\sum_{i = 1}^dg^{(\delta)}_p(\theta_{i,d})}{\sqrt{d}} \not\to \infty.
\end{align}
By~\eqref{eq:qnormaux}, for every~$\delta \in (0, \infty)$, we can choose a subsequence~$d'_{\delta}$ of~$d$ (we highlight the dependence of the subsequence on~$\delta$ in the proof to avoid confusion), such that~$\sup_{d'_{\delta}} G_{d'_{\delta},p}(\delta) \leq D(\delta) < \infty$ for some~$D(\delta)$. Although~$d_{\delta}'$ and~$D(\delta)$ depend on~$\delta$, in general, if~$p \geq 2$ these dependencies can be avoided, because~$\sup_{d_1'} G_{d_1',p}(\delta) \leq \sup_{d_1'} G_{d_1',p}(1) =: D(1) < \infty$ then holds for all~$\delta \in (0, \infty)$. To see the latter, just note that in case~$p \geq 2$ the function~$\delta \mapsto G_{d,p}(\delta)$, for~$\delta \in (0, \infty)$, has a maximum at~$\delta = 1$ for every~$d \in \N$. In what follows we set~$d_{\delta}' = d_1'$ and~$D(\delta) = D(1)$ in case~$p \geq 2$.

Irrespective of the value of~$p\in(0, \infty)$, we obtain for every~$\delta \in (0, \infty)$ that
\begin{equation}\label{eqn:qnstr0}
	\delta^p \times \sup_{d_{\delta}'} 
	\frac{\sum_{i=1}^{d_{\delta}'}\mathds{1}\cbr[0]{|\theta_{i,d_{\delta}'}|>\delta}}{\sqrt{d_{\delta}'}}
	\leq
	\sup_{d_{\delta}'}
	\frac{\sum_{i=1}^{d_{\delta}'}|\theta_{i,d_{\delta}'}|^p\mathds{1}\cbr[0]{|\theta_{i,d_{\delta}'}|>\delta}}{\sqrt{d_{\delta}'}} 
	\leq \sup_{d_{\delta}'} G_{d_{\delta}', p}(\delta) \leq  D(\delta).
\end{equation}
In case~$p \geq 2$, we can furthermore take the supremum over~$\delta \in (0, \infty)$ (since the subsequence chosen does not depend on~$\delta$) to get
\begin{equation}\label{eqn:qnstr1}
	\sup_{\delta \in (0, \infty)}
	\delta^p \times \sup_{d'_1}
	\frac{\sum_{i=1}^{d'_1}\mathds{1}\cbr[0]{|\theta_{i,d'_1}|>\delta}}{\sqrt{d'_1}} \leq D(1).
\end{equation}
Thus, the first statements in Equation~\eqref{eqn:corstr1} and~\eqref{eqn:corstr2}, respectively, hold for the subsequences~$d_{\delta}'$ and~$d_1'$. It remains to verify that for every~$\delta \in (0, \infty)$ we have that~$\max_{i = 1, \hdots, d'_{\delta}} |\theta_{i, d'_{\delta}}| \to \infty$ (note that~$d'_{\delta} = d_1'$ in case~$p \geq 2$ by construction). Fix~$\delta \in (0, \infty)$. Theorem~\ref{thm:conspreal} and Remark~\ref{rem:arb} show that~$\bm{\vartheta} \in \mathscr{C}(\{q, \kappa_{d,q}\})$ is equivalent to
\begin{equation}\label{eqn:qnstr}
	\frac{\sum_{i=1}^d\theta_{i,d}^2\mathds{1}\cbr[0]{|\theta_{i,d}|\leq\delta}+\sum_{i=1}^d|\theta_{i,d}|^q\mathds{1}\cbr[0]{|\theta_{i,d}|>\delta}}{\sqrt{d}}
	=
	\frac{\sum_{i = 1}^d g_q^{(\delta)}(\theta_{i,d})}{\sqrt{d}} \to \infty
\end{equation}
The last inequality in~\eqref{eqn:qnstr0} delivers~$\sum_{i=1}^{d'_{\delta}}\theta_{i,d'_{\delta}}^2\mathds{1}\cbr[0]{|\theta_{i,d'_{\delta}}|\leq\delta}/\sqrt{d'_{\delta}} \leq D(\delta)$ for every~$d'_{\delta}$. Thus,~\eqref{eqn:qnstr} implies
\begin{equation*}
	\frac{\sum_{i=1}^{d'_{\delta}}|\theta_{i,d'_{\delta}}|^q\mathds{1}\cbr[0]{|\theta_{i,d'_{\delta}}|>\delta}}{\sqrt{d'_{\delta}}} \to \infty.
\end{equation*}
But for every~$d'_{\delta}$
\begin{equation*}
	\frac{\sum_{i=1}^{d'_{\delta}}|\theta_{i,d'_{\delta}}|^q\mathds{1}\cbr[0]{|\theta_{i,d'_{\delta}}|>\delta}}{\sqrt{d'_{\delta}}} \leq \max_{i = 1, \hdots, d'_{\delta}} |\theta_{i,d'_{\delta}}|^q
	\frac{\sum_{i=1}^{d'_{\delta}}\mathds{1}\cbr[0]{|\theta_{i,d'_{\delta}}|>\delta}}{\sqrt{d'_{\delta}}},
\end{equation*}
which, by~\eqref{eqn:qnstr0}, is upper bounded by~$\max_{i = 1, \hdots, d'_{\delta}} |\theta_{i,d'_{\delta}}|^q D(\delta) \delta^{-p}$, which hence diverges to~$\infty$.

\subsection{Supremum-norm based tests}

In some proofs in this section we will use the following classic inequality (e.g.,~\cite{feller}, Section 7.1) for the standard normal cdf~$\Phi$:
\begin{equation}\label{eqn:Gtail}
	\left(x^{-1} - x^{-3}\right) \frac{e^{-x^2/2}}{\sqrt{2\pi}} \leq 1-\Phi(x) \leq x^{-1} \frac{e^{-x^2/2}}{\sqrt{2\pi}} \quad \text{ for every } x > 0.
\end{equation}

We also need the following observation concerning~$\kappa_d$ sequences leading to~$\{\infty, \kappa_d\} \in \mathbb{T}_{\alpha}$ for some~$\alpha \in (0, 1)$ (which provides a more common way of writing critical values for the supremum-based test in the normal case as compared to the general case treated in Lemma~\ref{lem:cvsup} further below). To obtain the result, we recall the classical limit theorem (e.g., Theorem 1.5.3 in \cite{leadbetter})
\begin{equation}\label{eqn:Glt}
	\P\del[2]{a_d\sbr[1]{\max_{1\leq i\leq d}\eps_i-b_d} \leq x}
	\to
	\Lambda(x) \quad \text{ for every } x \in \R,
\end{equation}
where~$\Lambda(x) = \exp\left(-\exp(-x)\right)$,~$x \in \R$, denotes Gumbel's double exponential cdf, and where $a_d=\sqrt{2\log(d)}$ and $b_d=\sqrt{2\log(d)}-\frac{\log\log(d)+\log(4\pi)}{2\sqrt{2\log(d)}}$. For~$x \in \R$ set~$u_d = u_d(x) := x/a_d + b_d$, so that
\begin{align*}
	\P\del[2]{\max_{1\leq i\leq d}|\eps_i|\leq u_d}
	=
	\P\del[2]{\max_{1\leq i\leq d}\eps_i\leq u_d,\ \min_{1\leq i\leq d}\eps_i\geq -u_d}
	\to \exp(-2\exp(-x)),
\end{align*}
where we used~\eqref{eqn:Glt} and the asymptotic independence of the minimum and maximum (e.g., Theorem 1.8.2 in \cite{leadbetter}). That is,
\begin{equation}\label{lem:absG}
	\P\del[2]{a_d\sbr[1]{\max_{1\leq i\leq d}|\eps_i|-b_d} \leq x}
	\to
	\exp(-2\exp(-x)) \quad \text{ for every } x \in \R,	
\end{equation}
The following now immediately follows from~\eqref{lem:absG}.
\begin{lemma}\label{lem:crit_val}
	It holds that $\{\infty, \kappa_{d}\} \in \mathbb{T}_{\alpha}$,~$\alpha \in (0, 1)$, if and only if
	\begin{align}\label{eqn:crit_val}
		\kappa_d = \sqrt{2\log(d)}-\frac{\log\log(d)+\log(4\pi)}{2\sqrt{2\log(d)}}-\frac{\log\del[0]{-\log(1-\alpha)/2}}{\sqrt{2\log(d)}}+o\del[2]{\frac{1}{\sqrt{\log(d)}}}.
	\end{align}
\end{lemma}

We can now proceed to proving the characterization of the consistency set of supremum-norm based tests. Throughout the proof we interpret sums over empty index sets as~$0$.

\subsubsection{Proof of Theorem~\ref{thm:conspinf}}

Let~$\{\infty, \kappa_d\}$ be as in the theorem statement. Proposition~\ref{prop:abs_noabs} delivers that~$\bm{\vartheta} \in \mathscr{C}(\{\infty, \kappa_d\})$ is equivalent to the condition that every subsequence~$d'$ of~$d$ has a subsequence~$d''$ along which~$$\overline{\Phi}(\kappa_{d} - \|\bm{\theta}_d\|_{\infty}) \to 1 \quad \text{ or } \quad \sum_{i = 1}^d \overline{\Phi}\left(
\kappa_{d} - |\theta_{i,d}|
\right) \to \infty.$$ As a first step, we now show that this condition, and hence also~$\bm{\vartheta} \in \mathscr{C}(\{\infty, \kappa_d\})$, is equivalent to the same condition but with~$\mathfrak{c}_d = \sqrt{2\log(d)} - \frac{\log \log(d)}{2 \sqrt{2\log(d)}}$ replacing~$\kappa_d$.

Let~$d'$ be a subsequence of~$d$. Lemma~\ref{lem:crit_val} shows that~$\kappa_d - \mathfrak{c}_d = u_d/\sqrt{2\log(d)}$ for a bounded sequence~$u_d$. Therefore,
\begin{equation}\label{eqn:domd'}
	\limsup_{d' \to \infty} \overline{\Phi}(\kappa_{d'} - \|\bm{\theta}_{d'}\|_{\infty}) = 1 \quad \Leftrightarrow \quad \limsup_{d' \to \infty} \overline{\Phi}(\mathfrak{c}_{d'} - \|\bm{\theta}_{d'}\|_{\infty}) = 1,
\end{equation}
and setting~$\mathcal{I}_d := \{i \in \{1, \hdots, d\}: \kappa_d - |\theta_{i,d}| \leq 3\}$, it is obvious that 
\begin{equation}\label{eqn:eql1}
	\limsup_{d' \to \infty} \sum_{i \in \mathcal{I}_{d'}}\overline{\Phi}\left(
	\kappa_{d'} - |\theta_{i,d'}|
	\right) = \infty ~  \Leftrightarrow ~
	\limsup_{d' \to \infty} \sum_{i \in \mathcal{I}_{d'}}\overline{\Phi}\left(
	\mathfrak{c}_{d'} - |\theta_{i,d'}|
	\right) = \infty,
\end{equation}
(as both divergences are equivalent to~$|\mathcal{I}_d| \to \infty$ along a subsequence of~$d'$). For~$\mathcal{I}_d^c := \{1, \hdots, d\} \setminus \mathcal{I}_d$ we now show that
\begin{equation}\label{eqn:eql2}
	\limsup_{d' \to \infty} \sum_{i \in \mathcal{I}_{d'}^c}\overline{\Phi}\left(
	\kappa_{d'} - |\theta_{i,d'}|
	\right) = \infty ~  \Leftrightarrow ~
	\limsup_{d' \to \infty} \sum_{i \in \mathcal{I}_{d'}^c}\overline{\Phi}\left(
	\mathfrak{c}_{d'} - |\theta_{i,d'}|
	\right) = \infty.
\end{equation}
To this end, let~$d \geq 2$ be large enough so that
\begin{equation}\label{eqn:udbound}
	|u_d| \leq \sqrt{2\log(d)}.
\end{equation}
Furthermore, let~$x$ and~$y$ be in~$(2, \infty)$ and such that~$x-y = u_d/\sqrt{2\log(d)}$. We make the following two observations:
\begin{enumerate}
	\item[O1:] If~$\max(x, y) \geq \sqrt{2\log(d)}$, then, denoting~$z := \sqrt{2\log(d)} - |u_d|/\sqrt{2\log(d)}$, we have$$\overline{\Phi}(x) \vee \overline{\Phi}(y) \leq \overline{\Phi}\left(z \right) \leq e^{-z^2/2} \leq d^{-1} e^{|u_d|},$$
	where we used that~$\overline{\Phi}(w) \leq e^{-w^2/2}$ for~$w \geq 0$.
	\item[O2:] If~$\max(x, y) \leq \sqrt{2\log(d)})$, then~$|x^2-y^2| \leq 2|u_d|$, and
	\begin{equation*}
		\frac{x^{-1}}{y^{-1} - y^{-3}} = \left[
		(y + u_d/\sqrt{2\log(d)})(y^{-1} -y^{-3})
		\right]^{-1}
		\leq 
		\left[
		(1 - y^{-1})(1 -y^{-2})
		\right]^{-1} \leq \frac{8}{3},
	\end{equation*}
	where we used~\eqref{eqn:udbound} in the first (displayed) inequality. Now,~\eqref{eqn:Gtail} and boundedness of~$u_d$ shows that
	\begin{equation*}
		0 < \frac{3}{8} e^{-\sup_{d \in \N} |u_d|} \leq \frac{\overline{\Phi}(x)}{\overline{\Phi}(y)} \leq \frac{8}{3} e^{\sup_{d \in \N} |u_d|} < \infty;
	\end{equation*}
	(by symmetry one only needs to check the upper bound). 
\end{enumerate}
Let~$\mathcal{I}_{d, 1}^c$ denote those indices~$i$ in~$\mathcal{I}_{d}^c$ for which~$\max(\kappa_d - |\theta_{i,d}|, \mathfrak{c}_d - |\theta_{i,d}|) \geq \sqrt{2\log(d)}$, and denote the set of remaining indices in~$\mathcal{I}_{d}^c$ by~$\mathcal{I}_{d, 2}^c$. It follows from O1 together with~$|\mathcal{I}_{d}^c| \leq d$, that (to show the equivalence in~\eqref{eqn:eql2}) it suffices to verify the equivalence with~$\mathcal{I}_{d'}^c$ replaced by~$\mathcal{I}_{d', 2}^c$. But the latter equivalence follows from O2, which establishes~\eqref{eqn:eql2}. Combining~\eqref{eqn:domd'},~\eqref{eqn:eql1} and~\eqref{eqn:eql2}, we have shown that~$\bm{\vartheta} \in \mathscr{C}(\{\infty, \kappa_d\})$ is equivalent to every subsequence~$d'$ of~$d$ having a subsequence~$d''$ along which
\begin{equation}\label{eqn:subsCH}
	\overline{\Phi}(\mathfrak{c}_d - \|\bm{\theta}_d\|_{\infty}) \to 1 \quad \text{ or } \quad \sum_{i = 1}^d \overline{\Phi}\left(
	\mathfrak{c}_d- |\theta_{i,d}|
	\right) \to \infty.
\end{equation}
By Lemma~\ref{lem:crit_val}, the sequence of critical values~$\mathfrak{c}_d$ satisfies~$\{\infty, \mathfrak{c}_d\} \in \mathbb{T}_{\alpha^*}$ for some~$\alpha^* \in (0, 1)$. Hence, the just-derived equivalence~\eqref{eqn:subsCH} together with (the equivalence~$(3 \Leftrightarrow 4)$) in Proposition~\ref{prop:abs_noabs} applied to~$\{\infty, \mathfrak{c}_d\}$ shows that~
\begin{equation}\label{eq:FInsup}
	\bm{\vartheta} \in \mathscr{C}(\{\infty, \kappa_d\}) \Leftrightarrow \sum_{i = 1}^d \frac{\overline{\Phi}\left(\mathfrak{c}_d - |\theta_{i,d}|\right)}{\Phi\left(\mathfrak{c}_d - |\theta_{i,d}|\right)} \to \infty.
\end{equation}
Finally, note that by construction~$\overline{\Phi}/g_{\infty} > 0$ is uniformly continuous on compact subsets of~$\mathbb{R} \setminus \{1\}$ with positive left- and right-sided limits at~$1$, and~$\overline{\Phi}(x)/g_{\infty}(x) \to 1/\sqrt{2\pi}$ as~$x \to \infty$. Therefore, for every~$z \in \R$ there exists a~$C(z) \in (0, \infty)$ such that
\begin{equation}
	\frac{\overline{\Phi}(x)}{g_{\infty}(x)} \in [C(z)^{-1}, C(z)] \quad \text{ for every } x \geq z.
\end{equation}
Furthermore, a sequence~$x_d$ satisfies~$g_{\infty}(x_d) \to \infty$ if and only if~$x_d \to - \infty$. 
It follows that~$\sum_{i = 1}^d g_{\infty}\left(\mathfrak{c}_d - |\theta_{i,d}|\right) \to \infty$ if and only if every subsequence~$d'$ of~$d$ permits a subsequence~$d''$ along which~$\mathfrak{c}_d - \|\bm{\theta}\|_{\infty} \to - \infty$ (equivalently~$\overline{\Phi}(\mathfrak{c}_d - \|\bm{\theta}_d\|_{\infty}) \to 1$) or~$\sum_{i = 1}^d \overline{\Phi}(\mathfrak{c}_d - |\theta_{i,d}|) \to \infty$. In other words, we have shown that the condition in~\eqref{eqn:subsCH} is equivalent to~$\sum_{i = 1}^d g_{\infty}\left(\mathfrak{c}_d - |\theta_{i,d}|\right) \to \infty$, which concludes the proof.

\subsubsection{Proof of Theorem~\ref{thm:pdconssup}}

Proposition~\ref{prop:pdsup} (applicable due to Remark~\ref{rem:NRVnormal}) and Remark~\ref{rem:consnotdep} deliver the result.
\subsubsection{Proof of Theorem~\ref{thm:2-p-inf}}
It suffices to verify~$\bm{\vartheta}^{\dagger} \in \mathscr{C}(\{p, \kappa_{d, p}\})$ and~\eqref{eq:strongerstatement}. Recall that~$\bm{\theta}^{\dagger}_d := (\tau_d, \hdots, \tau_d, 0, \hdots, 0)$, $\tau_d=\sqrt{2\log(d)}/\log\log(d)$, and that the number of non-zero entries of~$\bm{\theta}_d$ is~$\lceil \sqrt{d}/\log(d)\rceil$ (at least for~$d$ large enough).

\textbf{1.:} That~$\bm{\vartheta}^{\dagger} \in \mathscr{C}(\{p, \kappa_{d, p}\})$ follows from Corollary~\ref{cor:rewrite}, noting that eventually~$$\|\bm{\theta}_d\|_p^p /\sqrt{d} \geq 2^{p/2} \log^{p/2 - 1}(d)/(\log \log(d))^p \to \infty,$$
the divergence following from~$p > 2$.

\textbf{2.:} That~$\mathbb{P}(
\|\bm{\theta}_d^{\dagger} + \bm{\varepsilon}_d \|_2 \geq \kappa_{d,2}
) \to \alpha_2$ follows from Theorem~\ref{thm:consprealGEN} (cf.~also Remark~\ref{rem:pfindist}) and~$$\sum_{i = 1}^d g_2(\theta_{i,d})/\sqrt{d} = \|\bm{\theta}^{\dagger}_d\|_2^2/\sqrt{d} \to 0.$$

\textbf{3.:} We now show that~$\mathbb{P}(
\|\bm{\theta}_d^{\dagger} + \bm{\varepsilon}_d \|_{\infty} \leq \kappa_{d,\infty}
) \to 1- \alpha_{\infty}$.  Setting~$k_d:=\lceil \sqrt{d}/\log(d)\rceil$, write~$\mathbb{P}(
\|\bm{\theta}_d^{\dagger} + \bm{\varepsilon}_d \|_{\infty} \leq \kappa_{d,\infty}
)
$ as
\begin{align}\label{eq:splittingup}
	\P(|\varepsilon_1 + \tau_d| \leq \kappa_{d, \infty})^{k_d}
	\sbr[1]{1-2\Phi(-\kappa_{d, \infty})}^{d-k_d}.
\end{align}
Concerning the second factor in~\eqref{eq:splittingup},~$B_d$, say, observe that 
\begin{align*}
	B_d = \sbr[1]{1-2\Phi(-\kappa_{d, \infty})}^{d-k_d}
	=
	\frac{\sbr[1]{1-2\Phi(-\kappa_{d, \infty})}^{d}}{\sbr[1]{1-2\Phi(-\kappa_{d, \infty})}^{d\times \frac{k_d}{d}}}
	\to 1-\alpha_{\infty},
\end{align*}
where we used that~$\{\infty, \kappa_{d, \infty}\}\in \mathbb{T}_{\alpha_{\infty}}$ is equivalent to~$\sbr[0]{1-2\Phi(-\kappa_{d, \infty})}^{d}\to 1-\alpha_{\infty}$, and that, as~$k_d/d\to 0$, one also has~$\sbr[0]{1-2\Phi(-\kappa_{d, \infty})}^{d\times \frac{k_d}{d}} \to 1$. It remains to show that the first factor in~\eqref{eq:splittingup},~$A_d$, say, converges to~$1$. Note that~$\P(|\varepsilon_1 + \tau_d| \leq \kappa_{d, \infty})$ can be written as
\begin{equation}
	\Phi(\kappa_{d, \infty} - \tau_d) - \Phi(-\kappa_{d, \infty} - \tau_d) \geq 1-2\Phi(-\kappa_{d, \infty} + \tau_d) = 1- 2\overline{\Phi}(\kappa_{d, \infty} - \tau_d).
\end{equation}
By~\eqref{eqn:crit_val} (and since~$\{\infty, \kappa_{d, \infty}\}\in \mathbb{T}_{\alpha_{\infty}}$ with~$\alpha_{\infty} \in (0, 1)$) we eventually have~$\kappa_{d, \infty}- \tau_d \geq \sqrt{\log(d)}$. Thus, we can eventually apply Bernoulli's inequality (that is, $(1+x)^m \geq 1+mx$ for~$m \in \N \cup \{0\}$ and~$x \geq -1$) and use~\eqref{eqn:Gtail} to conclude that
\begin{equation*}
	A_d \geq 1-2k_d \overline{\Phi}\left(\sqrt{\log(d)}\right) \geq 1-2 k_d /\sqrt{d} \to 1.
\end{equation*}

\subsubsection{Proof of Corollary~\ref{cor:PE}}

First consider the case where~$\alpha_2$ and~$\alpha_{\infty}$ are both greater than~$0$. Then, we just note that~\eqref{eqn:PEup} implies
\begin{equation*}
	\mathbb{E} \varphi_d(\bm{\theta}_d^{\dagger} + \bm{\varepsilon}_d) \leq \mathbb{P}(\|\bm{\theta}_d^{\dagger} + \bm{\varepsilon}_d\|_2 \geq \kappa_{d, 2}) + \mathbb{P}(\|\bm{\theta}_d^{\dagger} + \bm{\varepsilon}_d\|_{\infty} \geq \kappa_{d, \infty}) \to \alpha_2 + \alpha_{\infty} < 1,
\end{equation*}
the convergence following from the last statement in Theorem~\ref{thm:2-p-inf}. If any of the asymptotic sizes is~$0$, just pass to smaller critical values so that the asymptotic sizes are in~$(0, 1)$ and sum up to a number smaller than~$1$. Then, use the monotonicity of the rejection probabilities in the critical values and apply the already established statement.

\subsubsection{Proof of Theorem~\ref{thm:domp}}

The theorem follows immediately from Theorem~\ref{thm:dompgen} (cf.~Remark~\ref{rem:pfindist}) together with Remarks~\ref{rem:NRVnormal}, \ref{rem:consnotdep} and Theorem~\ref{thm:conspinf}.

\subsubsection{Proof of Theorem~\ref{thm:opt_sum}}

This is an immediate consequence of Theorem~\ref{thm:opt_sum_gen} (cf.~Remark~\ref{rem:pfindist}), noting that 
by Anderson's theorem~$$\E(\psi_d(\bm{\theta}_d+\bm{\eps}_d)) - \P\del[1]{\|\bm{\theta}_d+\bm{\eps}_d\|_{p_j}\geq \kappa_{d}} \leq 
1 - \P\del[1]{\|\bm{\eps}_d\|_{p_j}\geq \kappa_{d}}
\leq 1-\alpha + o(1),$$ which establishes~\eqref{eq:powerfunction2} from~\eqref{eq:powerfunction2gen}.

\subsection{Verification of a claim in Section~\ref{sec:pinf}}\label{app:supt}

Let~$\{\infty, \kappa_d\} \in \mathbb{T}_{\alpha}$,~$\alpha \in (0, 1)$. In Section~\ref{sec:pinf} it was claimed that for a real sequence~$\tau_d$ and~$\bm{\iota}_d$ the vector of ones of length~$d$ it holds that~$\bm{\vartheta} := \{\tau_d \bm{\iota}_d : d \in \N\} \in \mathscr{C}(\{\infty,\kappa_{d} \})$  if and only if~$\sqrt{\log(d)}|\tau_d| \to \infty$. To show this via Theorem~\ref{thm:conspinf}, it suffices to verify that for any sequence~$\tau_d \geq 0$ it holds that
\begin{equation}\label{eqn:aeqdisp}
	d g_{\infty}\left(\mathfrak{c}_d - \tau_d\right) \to \infty \quad \Leftrightarrow \quad \sqrt{\log(d)}\tau_d \to \infty.
\end{equation}
We choose~$w$ such that~$g_{\infty}$ is continuous and strictly decreasing (e.g., as in the example given after Equation~\eqref{eqn:ginfty}).
We first verify~\eqref{eqn:aeqdisp} for all bounded sequences~$\tau_d \geq 0$. In this case, for~$d$ large enough,~$d g_{\infty}\left(\mathfrak{c}_d - \tau_d\right) = de^{-(\mathfrak{c}_d - \tau_d)^2/2}/(\mathfrak{c}_d - \tau_d)$, the logarithm of which writes, using~$\mathfrak{c}_d := \sqrt{2\log(d)} - \log \log(d)/(2 \sqrt{2\log(d)})$ and expanding the square,
\begin{align*}
	\log(d) - (\mathfrak{c}_d - \tau_d)^2/2 - \log(\mathfrak{c}_d - \tau_d)  
	&= 
	\log\left(\sqrt{\log(d)}/(\mathfrak{c}_d - \tau_d)\right) + \sqrt{2\log(d)}\tau_d + O(1) \\
	&= \sqrt{2\log(d)}\tau_d + O(1).
\end{align*}
Finally, if there existed an unbounded sequence~$\tau_d \geq 0$ with the property that one of the sequences in~\eqref{eqn:aeqdisp} diverges and the other doesn't, there would exist a subsequence~$d'$ along which this property is preserved and along which~$\tau_d$ diverges to~$\infty$ (otherwise this would contradict what has already been shown). Then,~$\sqrt{\log(d')} \tau_{d'} \to \infty$ and thus~$\sqrt{\log(d')} \min(\tau_{d'}, 1)$ diverges, implying~$d'g_{\infty}(\mathfrak{c}_{d'} - \min(\tau_{d'}, 1)) \to \infty$ (by what we have already established for bounded sequences, along subsequences). Since~$g_{\infty}$ is strictly decreasing, we obtain~$d' g_{\infty}(\mathfrak{c}_{d'} - \tau_{d'}) \to \infty$, contradicting the property upon~$\tau_{d}$ was chosen to satisfy in the first place, as we have just shown that both sequences in~\eqref{eqn:aeqdisp} diverge along~$d'$.

\section{Minimax adaptive testing}\label{sec:minimax}

The main focus of the present article is to study the consistency behavior of~$p$-norm based tests for the \emph{unrestricted} testing problem~\eqref{eqn:tp}, and to answer questions like ``how do these tests compare'' and ``can they be dominated'' in terms of their consistency properties. 

Classical results on~$p$-norm based tests in the literature on minimax-optimal testing, on the other hand, start with a specific set of alternatives, i.e., complements of~$p$-norm balls centered at the origin. They then characterize the separation from the null necessary so that uniform consistency is possible, and show that a~$p$-norm based test can be constructed that is minimax rate optimal against such alternatives. Such results provide a justification for using a~$p$-norm based test if one cares about power against alternatives in the complement of a~$p$-norm ball as just described, but they do not answer the questions that we have focused on in previous sections.

In the present section, we want to illustrate that tools similar to the ones used in previous sections can be used to establish adaptivity results in the minimax framework. Essentially, we show that a~$p$-norm based test is not only minimax optimal against complements of centered~$p$-norm balls, but is adaptively minimax optimal over all complements of centered~$q$-norm balls with~$q \leq p$ (assuming that the radii of the balls admit minimax consistent tests), thus extending a result in~\cite{ingster} for the case where~$p \leq 2$. This corresponds to the monotonicity result in Theorem~\ref{thm:incl}. We also prove an adaptivity result over the whole range of all~$p \in (0, \infty)$ which parallels Theorem~\ref{thm:domp}.

Let us first revisit classical results and introduce some notation. Given~$p\in(0,\infty)$ and a radius~$r \in (0, \infty)$, let~$\mathbb{V}^d_p(r):=\cbr[0]{\bm{\theta}_d\in\R^d:\|\bm{\theta}_d\|_p\geq r}$. For a sequence of such radii~$r_{p,d}$,~\mbox{$d\in\N$}, we now consider (for every~$d \in \N$) the testing problem
\begin{align}\label{eq:mmtestprob}
	H_{0,d}:\bm{\theta}_d=\bm{0}_{d}\quad \text{vs.}\quad H_{1,d}:\bm{\theta}_d\in \mathbb{V}^d_p(r_{p,d}),
\end{align}
and denote the minimal sum of Type~1 and Type~2 errors for this testing problem by
\begin{align}\label{eq:gamma}
	\gamma(r_{p,d},p):=\inf_{\varphi_d} \left[\E \varphi_d(\bm{\eps}_d)+1-\inf_{\bm{\theta}_d\in \mathbb{V}^d_p(r_{p,d})}\E\varphi_d(\bm{\theta}_d+\bm{\eps}_d)\right];
\end{align}
here the outer infimum is taken over all Borel measurable functions from~$\R^d$ to~$[0,1]$. 

Compared to the testing problem~\eqref{eqn:tp} studied in previous sections, the alternatives considered in~\eqref{eq:mmtestprob} are now separated from the null, the type of separation depending on~$p$. The main questions concerning the sequence of testing problems~\eqref{eq:mmtestprob} (which have long been answered) are: (i) for which sequences of radii does~$\gamma(r_{p,d}, p) \to 0$; and (ii) for which tests this is achieved. 

Following \cite{ingster}, a sequence of radii~$r_{p,d}^*$,~$d\in\N$, is called a sequence of asymptotic \emph{minimax rates/critical radii} for the sequence of testing problems~\eqref{eq:mmtestprob} if the following holds for any sequence of radii~$r_{p,d}$ as above and as~$d\to\infty$:
\begin{align*}
	&\gamma(r_{p,d},p)\to 0 \quad \text{if and only if } \quad r_{p,d}/r^*_{p,d}\to \infty\\
	&\gamma(r_{p,d},p)\to 1 \quad \text{if and only if } \quad r_{p,d}/r^*_{p,d}\to 0.
\end{align*}
Proposition~3.9 (and its proof) in \cite{ingster} settles questions (i) and~(ii) above as follows: (i) The following sequences constitute sequences of critical radii
\begin{align}\label{eqn:ingrates}
	r^*_{p,d}=
	\begin{cases}
		d^{\frac{4-p}{4p}}\quad &\text{if } p\in(0,2]\\
		d^{\frac{1}{2p}}\quad &\text{if }p\in(2,\infty).
	\end{cases}
\end{align}
(ii.a) For every~$p\in[2,\infty)$ and for every sequence of radii~$r_{p,d}$ such that~$r_{p,d}/r^*_{p,d}\to\infty$, there exists a sequence of critical values~$\kappa_{p,r_{p,d},d}$, such that
\begin{align*}
	\P\left(\|\bm{\eps}_d\|_p\geq \kappa_{p,r_{p,d},d}\right)+1-\inf_{\bm{\theta}_d\in \mathbb{V}^d_p(r_{p,d})}\P\left(\|\bm{\theta}_d+\bm{\eps}_d\|_p\geq \kappa_{p,r_{p,d},d}\right)\to 0.
\end{align*}
That is, there exists a~$p$-norm based test (the critical values depending on the sequence of radii~$r_{p,d}$) that is \emph{minimax rate consistent} in the sequence of testing problems~\eqref{eq:mmtestprob}.\footnote{In light of the minimax rate being~$r_{p,d}^*=d^{\frac{4-p}{4p}}$ for~$p\in(0,2)$ it is tempting to conjecture that~\eqref{eqn:equi2} could be replaced by~$\bm{\vartheta} \in \mathscr{C}(\{p,\kappa_{d} \})$ being equivalent to~$\frac{\|\bm{\theta}_d\|^2_2}{\sqrt{d}} \wedge \frac{\|\bm{\theta}_d\|^p_p}{d^{\frac{4-p}{4}}} \to \infty$. This, however, is not the case as can be seen by considering~$\bm{\theta}_d=(d^{\frac{1}{2p}}\log(d),0,\hdots,0)$.
}
(ii.b) For every~$p\in(0,2)$ and for every sequence of radii~$r_{p,d}$ such that~$r_{p,d}/r^*_{p,d}\to\infty$, there exists a~$2$-norm based test that is minimax rate consistent in the sequence of testing problems~\eqref{eq:mmtestprob}. 

Inspection of the proof of Proposition~3.9 and using Corollary~3.4 in~\cite{ingster} shows that actually more can be said in case (ii.b): Given a set of radial sequences~$$\{r_{q,d} \in (0, \infty): q \in (0, 2], d \in \N\} \quad \text{ such that } \quad \inf_{q \in (0,2]} r_{q,d}/r^*_{q,d}\to\infty,$$ there exists a~$2$-norm based test that is minimax rate consistent for the sequence of testing problems
\begin{equation}\label{eq:mmtestprobU}
	H_{0,d}:\bm{\theta}_d=\bm{0}_{d}\quad \text{vs.}\quad H_{1,d}:\bm{\theta}_d\in \bigcup_{q \in (0, 2]} \mathbb{V}^d_q(r_{q,d});
\end{equation}
that is, the~$2$-norm based test \emph{adapts} to~$q\in(0,2]$.
The following theorem now shows that such a result actually extends from the~$2$-norm based test to every~$p\in [2,\infty)$, which parallels the monotonicity phenomenon in Theorem~\ref{thm:incl}. Throughout the remainder of this section, we shall again use the notation~$\sigma^2_p := \mathbb{V}ar(|\varepsilon_1|^p)$ and~$\mu_p := \mathbb{E}(|\varepsilon_1|^p)$.
\begin{theorem}\label{thm:minimax_p}
	Fix~$p \in (2, \infty)$ and suppose a set of radial sequences~$$\{r_{q,d} \in (0, \infty): q \in (0, p],~d \in \N\} \quad \text{ satisfies } \quad r_d := \inf_{q \in (0,p]} r_{q,d}/r^*_{q,d}\to\infty.$$
	Then, the sequence of tests~$\{p, \kappa_{d}\}$ with~$\kappa_{d}:=\sbr[1]{r_{d}\sqrt{d\sigma^2_p}+d\mu_p}^{1/p}$ is minimax rate consistent in the sequence of testing problems
	\begin{equation}\label{eq:mmtestprobUp}
		H_{0,d}:\bm{\theta}_d=\bm{0}_{d}\quad \text{vs.}\quad H_{1,d}:\bm{\theta}_d\in \bigcup_{q \in (0, p]} \mathbb{V}^d_q(r_{q,d});
	\end{equation}
	that is, for $\mathbb{V}_{p,d}:=\bigcup_{q \in (0, p]} \mathbb{V}^d_q(r_{q,d})$ it holds that
	\begin{align}\label{eq:mmresp}
		\P\del[1]{\enVert[0]{\bm{\eps}_d}_p\geq \kappa_{d}}+1-\inf_{\bm{\theta}_d\in\mathbb{V}_{p,d}}\P\del[1]{\enVert[0]{\bm{\theta}_d+\bm{\eps}_d}_p\geq \kappa_{d}}\to 0.
	\end{align}
\end{theorem}

Given the adaptivity result just obtained, one may ask whether one can construct a \emph{single} test that is minimax rate consistent in any testing problem of the type~\eqref{eq:mmtestprobUp}. That is, does there exist a \emph{single} test that is minimax rate consistent in the sequence of testing problems~\eqref{eq:mmtestprobUp} for \emph{every}~$p\in(0,\infty)$ simultaneously? As we shall establish next, this question can be answered affirmatively using a construction related to the one in Theorem~\ref{thm:domp}.
\begin{theorem}\label{thm:mmsum}
	Suppose a set of radial sequences~$$\{r_{q,d} \in (0, \infty): q \in (0, \infty),~d \in \N\} \quad \text{ satisfies } \quad r_d := \inf_{q \in (0,\infty)} r_{q,d}/r^*_{q,d}\to\infty.$$
	Let~$p_d$ be a non-decreasing and diverging sequence of natural numbers satisfying
	\begin{equation}\label{eqn:pdas}
		p_d \Phi(-r_d) \to 0\quad \text{ and } \quad \frac{p_d}{\sqrt{d}} \left(3/2\right)^{3p_d/2} \to 0.
	\end{equation}
	Then, setting~$\kappa_{j,d} :=\sbr[1]{r_d\sqrt{d\sigma^2_j}+d\mu_j}^{1/j}$, the sequence of tests
	\begin{equation}\label{eq:sum_test}
		\psi^*_d(\cdot) := \mathds{1}\left\{
		\max_{j = 1,\hdots, p_d} \kappa_{j, d}^{-1} \|\cdot \|_j \geq 1
		\right\}
	\end{equation}
	is minimax rate consistent in the sequence of testing problems~\eqref{eq:mmtestprobUp} for every~$p \in (0, \infty)$; \mbox{that is,} denoting~$\mathbb{V}_{p,d}:=\bigcup_{q \in (0, p]} \mathbb{V}^d_q(r_{q,d})$, it holds that
	\begin{align}\label{eqn:mmad}
		\E \psi_d^*(\bm{\eps}_d)+1-\inf_{\bm{\theta}_d\in\mathbb{V}_{p,d}}\E\psi_d^*(\bm{\theta}_d+\bm{\eps}_d)\to 0\qquad \text{for every }p\in(0,\infty).
	\end{align}
\end{theorem}%

For concreteness, we have decided to base the construction in~\eqref{eq:sum_test} on all powers from~$1$ to~$p_d$. Inspection of the proof shows that this is not crucial, and that one can also achieve the same minimax rate consistency property as in Theorem~\ref{thm:mmsum} by maximizing over an expanding subset of powers. 

\section{Proofs of the results in Appendix~\ref{sec:minimax}}\label{app:minimaxproofs}

The following notation will be employed freely throughout the remainder of this section. For every~$p\in(0,\infty)$ define~$\lambda_p:\R\to[0,\infty)$ via~$s\mapsto \E|\eps_1+s|^p$ and abbreviate~$\sigma^2_p:=\mathbb{V}ar|\eps_1|^p$. We denote the positive square root of~$\sigma^2_p$ by~$\sigma_p$.

We also note 
the following bounds concerning absolute moments of a standard normally distributed random variable, which are due to the following version of Stirling's approximation (cf.~Theorem~1.5 in~\cite{batir}, applied with~``$x = (r-1)/2$'')
\begin{equation}\label{eqn:batir}
	\sqrt{2 e} \left(r/(2e)\right)^{r/2} \leq \Gamma\left((r+1)/2\right) < \sqrt{2\pi} \left(r/(2e)\right)^{r/2}, ~~\text{ for every } r > 1,
\end{equation}
and will be used in the proof of Theorem~\ref{thm:mmsum} below.
\begin{lemma}\label{lem:mom_gauss}
	Let~$Z$ be a standard normal random variable. Then, for all~$r > 1$,
	\begin{align*}
		\sqrt{2e/\pi} r^{r/2}e^{-r/2}
		\leq
		\E|Z|^r
		<
		\sqrt{2} r^{r/2}e^{-r/2}.
	\end{align*}
\end{lemma}

\begin{proof}
	Write~$\E|Z|^r
	=
	\frac{1}{\sqrt{\pi}}\Gamma((r+1)/2)2^{r/2}$
	and apply~\eqref{eqn:batir}.
\end{proof}

\subsubsection{Proof of Theorem~\ref{thm:minimax_p}}  Let~$p\in(2,\infty)$ and the set of radial sequences be as in the statement of the theorem. We start with some preliminary observations. Note that for all~$q\in(0,p]$ and~$d\in\N$ one has~$r_{q,d}\geq r_dr^*_{q,d}$. Furthermore, for every~$\bm{\theta}_d\in\R^d$, one has that~$\|\bm{\theta}_d + \bm{\varepsilon}_d\|_p \geq \kappa_{d}$ is equivalent to~
\begin{equation*}
	t_{p,d} = \frac{\|\bm{\theta}_d + \bm{\varepsilon}_d\|_p^p - \sum_{i = 1}^d \lambda_p(\theta_{i,d})}{\sqrt{d \sigma^2_p}} + \frac{\sum_{i = 1}^d\left( \lambda_p(\theta_{i,d}) - \lambda_p(0)\right)}{\sqrt{d \sigma^2_p}} \geq r_d. 
\end{equation*}
Note that~$t_{p,d}$ depends on~$\bm{\theta}_d$, but we don't highlight this in our notation. By Lemma~\ref{lem:suff0behav} (cf.~Remark~\ref{rem:pfindist}) there exists a~$c_p>0$ such that
\begin{align}\label{eq:common}
	\E(t_{p,d})
	=
	\frac{\sum_{i = 1}^d\left( \lambda_p(\theta_{i,d}) - \lambda_p(0)\right)}{\sqrt{d \sigma^2_p}} 
	\geq 
	c_p^{-1}\frac{\sum_{i=1}^dg_p(\theta_{i,d})}{\sqrt{d \sigma^2_p}} = c_p^{-1}\frac{\sum_{i=1}^d( \theta^2_{i,d} \vee |\theta_{i,d}|^p)}{\sqrt{d \sigma^2_p}},
\end{align}
the last equality following from~$g_p(x) = x^2 \vee |x|^p$ for~$p \in [2, \infty)$,~cf.~\eqref{eqn:rhodef}. We can also bound~
\begin{equation}\label{eqn:gnq}
	\sum_{i=1}^d (\theta^2_{i,d} \vee |\theta_{i,d}|^p) \geq \|\bm{\theta}_d\|_s^s, \text{ for every } s \in [2, p].
\end{equation} 

We next claim that for~$d$ large enough
\begin{align}\label{eqn:explbdalt}
	\inf_{\bm{\theta}_d\in \mathbb{V}_{p,d}}\E(t_{p,d})
	\geq
	2r_d,
\end{align}
and first establish two inequalities taking care of different subsets of~$\mathbb{V}_{p,d}$:

\textbf{1.} Let~$\bm{\theta}_d\in \mathbb{V}_d^q(r_{q,d})$ for~$q\in[2,p]$: Then,~$r^*_{q,d}=d^{1/2q}$ by~\eqref{eqn:ingrates} and hence~$r_{q,d}\geq r_dd^{1/2q}$. By~\eqref{eq:common} and~\eqref{eqn:gnq} (applied with~$s = q$) one has that
\begin{align*}
	\E(t_{p,d}) \geq c_p^{-1}\frac{\sum_{i=1}^dg_p(\theta_{i,d})}{\sqrt{d \sigma^2_p}}
	\geq
	c_{p}^{-1}\frac{||\bm{\theta}_d||_q^q}{\sqrt{d \sigma^2_p}}
	\geq
	c_{p}^{-1}\frac{r_d^q}{\sigma_p}.
\end{align*}

\textbf{2.} Let~$\bm{\theta}_d\in \mathbb{V}_d^q(r_{q,d})$  for~$q\in(0,2)$: Then,~$r^*_{q,d}=d^{\frac{4-q}{4q}}$ by~\eqref{eqn:ingrates}  and hence~$r_{q,d}\geq r_dd^{\frac{4-q}{4q}}$. By Jensen's inequality~$||\bm{\theta}_d||_2\geq d^{1/2-1/q}||\bm{\theta}_d||_q$. By~\eqref{eq:common} and~\eqref{eqn:gnq} (applied with~$s = 2$), 
\begin{align*}
	\E(t_{p,d}) \geq c_p^{-1}\frac{\sum_{i=1}^dg_p(\theta_{i,d})}{\sqrt{d \sigma^2_p}}
	\geq
	c_{p}^{-1}\frac{||\bm{\theta}_d||_2^2}{\sqrt{d \sigma^2_p}}
	\geq 
	c_{p}^{-1}\frac{d^{1-\frac{2}{q}}||\bm{\theta}_d||_q^2}{\sqrt{d\sigma^2_p}}
	\geq
	c_{p}^{-1}\frac{d^{1-\frac{2}{q}}r_d^2d^{\frac{4-q}{2q}}}{\sqrt{d\sigma^2_p}} 
	=
	c_{p}^{-1}\frac{r_d^2}{\sigma_p}.
\end{align*}

Since the lower bounds in the two inequality chains obtained are both greater than~$2r_d$ for~$d$ large enough (uniformly over the respective set of~$\bm{\theta}_d$ considered in each case) the claim in~\eqref{eqn:explbdalt} follows.

Next, by Part~\textbf{2.} of Lemma~\ref{lem:pbounds} (cf.~\eqref{eqn:vbdII} and note that~$p \leq 2p-2$), there exists a~$C_p \in (0, \infty)$ such that for all~$\bm{\theta}_d\in\R^d$
\begin{align*}
	\mathbb{V}ar(t_{p,d})
	=
	\sum_{i = 1}^d \mathbb{V}ar\left(
	|\theta_{i,d} + \varepsilon_i|^p/\sqrt{d\sigma^2_p}
	\right)
	\leq
	1+
	\frac{C_p\sum_{i = 1}^d
		g_{2p-2}(\theta_{i,d})}{d\sigma^2_p}.
\end{align*}
Using~\eqref{eqn:vbdIIL} in the proof of Lemma~\ref{lem:asyn} and~$\sum_{i=1}^dg_p(\theta_{i,d})\geq 1$ for all~$\bm{\theta}_d\in\mathbb{V}_{p,d}$ for~$d$ sufficiently large (the latter following from the displayed inequalities in \textbf{1.} and \textbf{2.} above) implies that for~$d$ sufficiently large
\begin{align*}
	\frac{\sum_{i = 1}^d
		g_{2p-2}(\theta_{i,d})}{d\sigma^2_p}
	\leq
	\frac{2}{(d\sigma^2_p)^{1/p}}\del[3]{\sum_{i=1}^dg_{p}(\theta_{i,d})/\sqrt{d\sigma^2_p}}^{\frac{2p-2}{p}}
	\leq
	\frac{1}{(d\sigma^2_p)^{1/p}}\del[3]{c_p \E(t_{p,d})}^{\frac{2p-2}{p}},
\end{align*}
the last inequality following from~\eqref{eq:common}. Thus, for~$\tau := (2p-2)/p \in(0,2)$ and for~$d$ sufficiently large,
\begin{align}\label{eq:varboundbyexp}
	\mathbb{V}ar(t_{p,d})
	\leq
	1+\del[1]{\E(t_{p,d})}^{\tau} \text{ for every } \bm{\theta}_d\in \mathbb{V}_{p, d}.
\end{align}
Corollary~3.1 in \cite{ingster} together with~\eqref{eqn:explbdalt}, $r_d \to \infty$, and~\eqref{eq:varboundbyexp} now proves the result.

\subsubsection{Proof of Theorem~\ref{thm:mmsum}}

For every~$d \in \N$ and~$j = 1, \hdots, p_d$ we can write~$\|\bm{\varepsilon}_d\|_j\geq \kappa_{j,d}$ as
\begin{align*}
	\|\bm{\varepsilon}_d\|_j^j - d \lambda_j(0) \geq r_d\sqrt{d \sigma^2_j}.
\end{align*}
Combined with~$\psi_d^*(\cdot)\leq \sum_{j = 1}^{p_d} \mathds{1}\{
\| \cdot \|_{j} \geq \kappa_{j, d}\}$, which follows from the definition of~$\psi_d^*$, we obtain
\begin{equation*}
	\E \psi_d^*(\bm{\eps}_d)
	\leq
	\sum_{j=1}^{ p_d}\left[\Phi\del[0]{-r_d}+\sup_{x\in\R}\left|\P\left(\|\bm{\eps}_d\|_j^j - d \lambda_j(0) \leq x \sqrt{d \sigma^2_j}\right)-\Phi(x)\right|\right].
\end{equation*}
Using the Berry-Esseen theorem and Lemma~\ref{lem:mom_gauss}, together with a straightforward computation (note that~$\sigma^2_p = \mathbb{E}(|\varepsilon_1|^{2p}) - (\mathbb{E}(|\varepsilon_1|^{p}))^2$), it now follows that for constants~$C, D \in (0, \infty)$ (both independent of~$d$)
\begin{equation*}
	\E \psi_d^*(\bm{\eps}_d) \leq p_d \Phi(-r_d) +C\frac{p_d}{\sqrt{d}}\sup_{p\in[1,p_d]}\frac{\E|\eps_1|^{3p}}{\sigma^3_p} 
	\leq
	p_d \Phi(-r_d)+CD\frac{ p_d}{\sqrt{d}}\left(
	3/2
	\right)^{3p_d/2} \to 0,
\end{equation*}
the convergence following from~\eqref{eqn:pdas}. To establish~\eqref{eqn:mmad}, it hence remains to be shown that for any~$p\in(0,\infty)$, one has that 
\begin{align*}
	\inf_{\bm{\theta}_d\in \mathbb{V}_{p,d}}\E\psi_d^*(\bm{\theta}_d+\bm{\eps}_d)\to 1, \quad \text{as }d\to\infty.
\end{align*}
To this end, fix~$p\in(0,\infty)$ and let~$d$ be sufficiently large to ensure that~$J := \max(3, \lceil p \rceil) \leq p_d$ (recall that~$p_d \uparrow \infty$ by assumption). Then, it follows from the definition of~$\psi_d^*$ that
\begin{align*}
	\psi_d^*(\cdot)
	\geq
	\mathds{1}\cbr[0]{||\cdot||_{J}\geq \kappa_{J,d}},
\end{align*}
from which we can conclude that
\begin{align*}
	\inf_{\bm{\theta}_d\in \mathbb{V}_{p,d}}\E\psi_d^*(\bm{\theta}_d+\bm{\eps}_d)
	\geq
	\inf_{\bm{\theta}_d\in \mathbb{V}_{d,p}}\P(||\bm{\theta}_d+\bm{\eps}_d||_{J}\geq \kappa_{J,d})
	\geq
	\inf_{\bm{\theta}_d\in \mathbb{V}_{J,d}}\P(||\bm{\theta}_d+\bm{\eps}_d||_{J}\geq \kappa_{J,d}).
\end{align*}
From~$\overline{r}_{d}:=\inf_{q\in(0,J ]}\frac{r_{q,d}}{r^*_{q,d}} \geq r_d$ we now obtain
\begin{align*}
	\kappa_{J,d} = \left(
	r_d\sqrt{d\sigma^2_{J}}+d\lambda_{J}(0)\right)^{1/J}
	\leq \left(
	\overline{r}_{d}\sqrt{d\sigma^2_{J}}+d\lambda_{J}(0)\right)^{1/J} =: \kappa_d,
\end{align*}
so that~$$ \inf_{\bm{\theta}_d\in \mathbb{V}_{J,d}}\P(||\bm{\theta}_d+\bm{\eps}_d||_{J}\geq \kappa_{J,d}) \geq \inf_{\bm{\theta}_d\in \mathbb{V}_{J,d}}\P\del[1]{||\bm{\theta}_d+\bm{\eps}_d||_J\geq \kappa_d}
\to 
1,$$ where the convergence follows from Theorem~\ref{thm:minimax_p} (applied with~``$p = J$'').

\section{Supplementary general results \emph{not} imposing Gaussianity}\label{app:AUX}

Throughout Appendix~\ref{app:AUX} \emph{we do no longer assume that~$\varepsilon_i$ is standard normal} (although we still assume that the~$\varepsilon_i$ are i.i.d.), but we shall work under weaker conditions, which are imposed whenever needed. In particular, different (but overlapping) sets of conditions will be used depending on whether~$p \in (0, \infty)$ or~$p = \infty$. This allows us to establish more general results. 

We shall throughout denote the cdf of~$\varepsilon_i$ by~$F$, and we write~$\overline{F} = 1-F$. Since the~$\varepsilon_i$ are i.i.d.~throughout, an assumption imposed on the distribution of~$\varepsilon_1$ carries over to the distribution of every~$\varepsilon_{i}$ for~$i \in \N$. In formulating our assumptions we denote~$\varepsilon_1$ by~$\varepsilon$ for simplicity.

Note that in this more general framework, for a sequence of tests~$\varphi_d$ the statement whether or not~$\varphi_d \in \mathbb{T}_{\alpha}$ for a given~$\alpha \in (0, 1)$ can depend on~$F$. That is,~$\mathbb{T}_{\alpha}$ depends on~$F$, which we shall highlight by writing~$\mathbb{T}_{\alpha, F}$ in this section. Similarly, the consistency set~$\mathscr{C}(\varphi_d)$ of~$\varphi_d$ may depend on~$F$, which highlight by writing~$\mathscr{C}_F(\varphi_d)$.

\subsection{Consistency of~$p$-norm based tests with finite~$p > 0$}\label{appss:pfin}

We start with some auxiliary results.

\subsubsection{Assumptions and auxiliary results}\label{app:aux}

For every~$p \in (0, \infty)$ define the function~
\begin{equation}\label{eqn:ldef}
	\lambda_p:\R\to [0,\infty] \quad \text{ via } \quad s\mapsto\E|\eps+s|^p,
\end{equation}
(this notation was already used in Appendix~\ref{app:minimaxproofs} in the Gaussian case). Note that~$\lambda_p$ also depends on~$F$, but we do not highlight this notationally. Our main assumption on~$\varepsilon$ in Appendix~\ref{appss:pfin} is the following.

\begin{assumption}\label{ass:pfindist}
	The following conditions hold:	
	\begin{enumerate}
		\item $F(x) = 1-F(-x)$ for every~$x \in \R$;
		\item $F$ is twice continuously differentiable (with first and second derivative~$f$ and~$f'$, respectively);
		\item $f'$ restricted to~$(0, \infty)$ is non-positive, and is non-decreasing on~$(M, \infty)$ for some~$M > 0$.
	\end{enumerate} 	
\end{assumption}

\begin{remark}\label{rem:pfindist} Assumption~\ref{ass:pfindist} is obviously satisfied for normal distributions with mean zero (and positive variance). But it clearly holds much more generally.
\end{remark}

\begin{remark}\label{rem:fpbd}
	Note that under Assumption~\ref{ass:pfindist} the derivative~$f': \R \to \R$ is bounded.
\end{remark}

The following auxiliary results suitably extend statements contained in Lemma~3.2 of~\cite{ingster}, Lemma~4.2 of~\cite{pinelis2010asymptotic} and Sections~5.7 and~5.9 in the latter reference from the Gaussian case to distributions satisfying Assumption~\ref{ass:pfindist}. Recall the definition of~$g_p$ from Equation~\eqref{eqn:rhodef}.

\begin{lemma}\label{lem:suff0behav}
	Under Assumption~\ref{ass:pfindist} and if~$q \in (0, \infty)$ is such that~$\mathbb{E}(|\varepsilon|^q) < \infty$, then
	\begin{equation}\label{eqn:0behaveq}
		c_p^{-1} g_p(s) \leq \lambda_p(s)-\lambda_p(0) \leq c_p g_p(s) \quad \text{ for every } s \in \R,
	\end{equation}
	for every~$p \in (0, q]$ and a suitable constant~$c_p \in (0, \infty)$ (that may depend on~$F$). Furthermore, for every such~$p$ the function~$\lambda_p$ is real-valued, continuous and even.
\end{lemma}

\begin{proof}
	Let~$p \in (0, q]$. Note that~$\lambda_p$ is continuous and real valued because of~$\mathbb{E}(|\varepsilon|^p)^{1/p} \leq \mathbb{E}(|\varepsilon|^q)^{1/q} < \infty$ together with the Dominated Convergence Theorem (note that~$(a+b)^p \leq 2^p(a^p + b^p)$ for every~$a,b\geq 0$). Note furthermore that~$\lambda_p$ is even since~$f$ is even. This proves the statements in the last sentence of the lemma.
	
	We now claim that~\eqref{eqn:0behaveq} follows if we can show that
	\begin{equation}\label{eqn:0behav}
		\text{(a)}~~\lim_{s \to 0;  s \neq 0} s^{-2}\left(\lambda_p(s) - \lambda_p(0)\right)  \in (0, \infty),  ~ \text{ and } ~ \text{(b)}~~ \lambda_p(s) > \lambda_p(0) \text{ for every } s \neq 0.
	\end{equation}	
	To see that this claim is correct, first note that by the Dominated Convergence Theorem
	\begin{equation}\label{eqn:inflim}
		\lim_{s \to \pm\infty} |s|^{-p}\left(\lambda_p(s) - \lambda_p(0)\right) = \lim_{s \to \pm\infty} \E||s|^{-1}\eps+\mathrm{sign}(s)|^p = 1.
	\end{equation}
	Next, by~\eqref{eqn:0behav}, the function~$r(s) := (\lambda_p(s) - \lambda_p(0))/g_p(s)$,~$s \neq 0$, extends to a positive and continuous function on~$\R$, which we denote by~$r$ as well, and for which~\eqref{eqn:inflim} delivers~$$0 < a_0 := \inf_{s \in \R} r(s) \leq \sup_{s \in \R} r(s) =: a_1< \infty.$$ Hence~\eqref{eqn:0behaveq} follows. To verify~\eqref{eqn:0behav} set~$I(x, s) :=
	F(-x^{1/p}+s)
	+
	F(-x^{1/p}-s)$ for~$x \geq 0$ and~$s \in \R$, and use~$F(x) = 1-F(-x)$ to write
	\begin{equation}\label{eqn:intfomega}
		\lambda_p(s) = \int_{0}^{\infty}
		\P(|\varepsilon + s|^p > x)dx =  \int_{0}^{\infty}  I(x, s) dx.
	\end{equation}
	For every~$x \in \R$, using that~$f$ is even and the mean-value theorem,
	\begin{equation}\label{eqn:derivh}
		\begin{aligned}
			\partial_s I(x,s) &= f(x^{1/p} - s) - f(x^{1/p} + s) = -2s  f'(x^{1/p} + \tilde{s}),~\tilde{s} \in [-s, s], \\
			\partial^2_s I(x,s) &= - f'(x^{1/p} - s) - f'(x^{1/p} + s).
		\end{aligned}
	\end{equation}
	For every~$\epsilon > 0$, it thus follows that
	\begin{equation}\label{eqn:diffdom}
		\sup_{|s| \leq \epsilon, i = 1, 2}|\partial_s^i I(x,s)| \leq 2(\epsilon + 1) \sup_{|s| \leq \epsilon}|f'(x^{1/p} + s)|.
	\end{equation}
	Therefore, if, for every~$\epsilon > 0$, the function~$x \mapsto \sup_{|s| \leq \epsilon}|f'(x^{1/p} + s)|$ has a majorant~$m_{\epsilon}: (0, \infty) \to [0, \infty)$, say, that is Lebesgue integrable over~$(0, \infty)$, then, by the Dominated Convergence Theorem, we can differentiate twice under the integral in~\eqref{eqn:intfomega} at every~$s_0 \in \R$. By Remark~\ref{rem:fpbd}, we have~$|f'| \leq L$, say, and for~$\epsilon > 0$ set
	\begin{equation}
		m_{\epsilon}(x) :=
		\begin{cases}
			-f'(x^{1/p} - \epsilon) & \text{ if } x > (M+\epsilon)^p,
			\\
			L & \text{ if } x \in (0,  (M+\epsilon)^p].
		\end{cases}
	\end{equation}
	By assumption~$-f'$ is non-negative on~$(0, \infty)$ and non-increasing on~$(M, \infty)$. Thus,~$$0 \leq -f'(x^{1/p} + s) \leq -f'(x^{1/p} - \epsilon) = m_{\epsilon}(x) \text{ for every } x > (M+\epsilon)^p \text{ and every } |s|\leq \epsilon.$$ Hence, it remains to show that~$m_{\epsilon}$ is Lebesgue integrable. Since~$\mathbb{E}(|\varepsilon|^p)$ is finite, there exists a sequence~$b_m \in (1, \infty)$ such that~$b_m \uparrow \infty$ and~$b_m^p f(b_m) \to 0$. Because~$f'$ is bounded, and by monotone convergence, it suffices to show that the sequence~$\int_{(1+\epsilon)^p}^{(b_m + \epsilon)^p} -f'(x^{1/p} - \epsilon)dx = p \int_{1}^{b_m} -(x+\epsilon)^{p-1}f'(x)dx$ is bounded. But integration by parts shows that
	\begin{equation}\label{eqn:diffxpbd}
		\int_{1}^{b_m} -x^{p} f'(x)dx \leq f(1)-b_m^p f(b_m) + p\mathbb{E}(|\varepsilon|^p) \to f(1) + p\mathbb{E}(|\varepsilon|^p),
	\end{equation}
	where we used that~$\int_{1}^{b_m}x^{p-1}f(x) dx \leq \mathbb{E}(|\varepsilon|^p)$, and from which the boundedness condition follows. Differentiating twice under the integral in~\eqref{eqn:intfomega} now gives
	\begin{equation*}
		\lambda'_p(s) = \int_{0}^{\infty} f(x^{1/p} - s) - f(x^{1/p} + s) dx; ~  \lambda''_p(s) = - \int_{0}^{\infty} f'(x^{1/p} - s) + f'(x^{1/p} + s) dx.
	\end{equation*}
	By the Dominated Convergence Theorem both derivatives are continuous in~$s$ as~$f$ and~$f'$ are continuous (recall the majorant established above). 
	Thus~\eqref{eqn:0behav}(a) follows by L'Hospital' rule from~$\lambda_p'(0) = 0$ and~$\lambda_p''(0)$ positive (and finite) because~$-f'$ cannot vanish identically and is nonnegative everywhere on~$(0, \infty)$. From Parts~1 and~3 of Assumption~\ref{ass:pfindist} together with the previous display it readily follows that~$\lambda'_p(s) \geq 0$ for every~$s \geq 0$. But then~\eqref{eqn:0behav}(b) follows as well, because~\eqref{eqn:0behav}(a) already shows that~$\lambda_p(s) > \lambda_p(0)$ for every~$s$ in an open neighborhood of~$0$, and~$\lambda_p$ is even.
\end{proof}

\begin{lemma}\label{lem:pbounds}
	Under Assumption~\ref{ass:pfindist} and if~$q \in [2, \infty)$ is such that~$\mathbb{E}(|\varepsilon|^{q}) < \infty$, then the following holds:
	\begin{enumerate}[label=\textbf{\arabic*.}]
		\item  For every~$p \in (0, \infty)$ such that~$2 \lceil p/2 \rceil \leq q$ the function~$\Delta_p: \R \setminus (-1, 1) \to \R$ defined via
		\begin{equation}\label{eqn:pboundI}
			\Delta_p(s) := s^2 \left(|s|^{-p} \lambda_p(s) - 1\right)
		\end{equation}
		is bounded and continuous
		\item For every~$p \in (0, \infty)$ such that~$2 \lceil p \rceil \leq q$ there exists a constant~$C'_p \in (0, \infty)$, such that for every~$s \in \R$ we have
		\begin{equation}\label{eqn:pboundII}
			\left|\mathbb{V}ar|\eps+s|^p-\mathbb{V}ar|\eps|^p\right| 
			\leq
			C'_p\left[s^2\mathds{1}\{|s|\leq 1\}+(1+|s|^{2p-2})\mathds{1}\{|s|>1\}\right]. 
		\end{equation}
		\item For every~$p \in (0, \infty)$ such that~$2 \lceil 2p \rceil \leq q$, there exists a constant~$C''_{p} \in (0, \infty)$, such that for every~$s \in \R$ we have
		\begin{align}\label{eqn:pboundIII}
			\E\left[|\eps+s|^p-\E|\eps+s|^p\right]^{4} \leq C''_{p} \left[1+|s|^{4p - 2}\mathds{1}\{|s|>1\}\right].
		\end{align}
	\end{enumerate}
	
\end{lemma}
\begin{proof}
	\textbf{1.}: Fix~$p \in (0, \infty)$ such that~$\tilde{p} := \lceil p/2 \rceil \leq q/2$. By Lemma~\ref{lem:suff0behav} the function~$\lambda_p$ is real-valued, continuous and even. These properties carry over to~$\Delta_p$. To show that~$\Delta_p$ is bounded, it remains to verify that~$\limsup_{s\to \infty} \Delta_p(s) < \infty$ and that~$\liminf_{s\to \infty} \Delta_p(s) > - \infty$. Concerning the former, let~$s > 0$ be large enough such that~$|s^{-2\tilde{p}}\lambda_{2\tilde{p}}(s) - 1|<1/2$ holds (recall that~$s^{-2\tilde{p}}\lambda_{2\tilde{p}}(s) \to 1$ as~$s \to \infty$ was shown in~\eqref{eqn:inflim}). By Jensen's inequality and the mean-value theorem (both applied to~$x\mapsto x^{p/(2\tilde{p})}$) there exists a constant~$D > 0$ (independent of~$s$ in the range we consider) such that 
	\begin{equation*}
		s^{-p}\lambda_p(s)-1 \leq\left( \mathbb{E}\left((\varepsilon/s + 1)^{2\tilde{p}}\right)\right)^{p/(2\tilde{p})}-1 \leq D\left|
		\mathbb{E}\left((\varepsilon/s + 1)^{2\tilde{p}}\right)-1
		\right|.
	\end{equation*}
	Now, we write
	\begin{equation*}
		\mathbb{E}\left((\varepsilon/s + 1)^{2\tilde{p}}\right)-1 = \sum_{i = 1}^{2\tilde{p}} {2\tilde{p} \choose i}\frac{\mathbb{E}(\varepsilon^i)}{s^i} = \sum_{i = 1}^{\tilde{p}} {2\tilde{p} \choose 2i}\frac{\mathbb{E}(\varepsilon^{2i})}{s^{2i}},
	\end{equation*}
	the last equality following from symmetry of~$F$. Multiplying by~$s^2$ and letting~$s \to \infty$ we obtain~${2\tilde{p} \choose 2}\mathbb{E}(\varepsilon^{2}) < \infty$, which takes care of the limit superior. Next, we observe that
	\begin{equation*}
		s^{-p}\lambda_p(s)-1 \geq \mathbb{E}\left((\varepsilon/s + 1)^{p} \mathds{1}\{|\varepsilon|/s < 1/2\}\right)-1.
	\end{equation*}
	Expanding to the second order shows that for every~$x \in (-1/2, 1/2)$ the difference~$(x+1)^p - (1 + p x)$ equals~$\frac{p(p-1)}{2} (1+\tilde{x})^{p-2} x^2$ for an~$\tilde{x} \in (-1/2, 1/2)$. 
	In particular the difference is not smaller than~$-D' x^2$ for a constant~$D' > 0$ (treating separately the cases~$p \in (0, 1)$ and~$p \in [1, \infty)$). We can thus lower-bound~$s^{-p}\lambda_p(s)-1$ further by
	\begin{align*}
		&\mathbb{E}\left((p\varepsilon/s + 1) \mathds{1}\{|\varepsilon|/s < 1/2\}\right)-1 - D' \mathbb{E}((\varepsilon/s)^2) \\
		= &
		- \P(|\varepsilon|/s \geq 1/2) - D' \mathbb{E}((\varepsilon/s)^2) \geq - (4 + D') \mathbb{E}(\varepsilon^2) s^{-2};
	\end{align*}
	where we used~$\mathbb{E}(\varepsilon \mathds{1}\{|\varepsilon|/s < 1/2\}) = 0$ (by symmetry of~$F$) to obtain the equality, and the last inequality follows from Markov's inequality. This proves the statement concerning the limit inferior.
	
	\textbf{2.}: Fix~$p \in (0, \infty)$ such that~$2 \lceil p \rceil \leq q$. By symmetry of~$F$, it suffices to prove~\eqref{eqn:pboundII} for~$s \geq 0$. We start with the case~$s\geq 1$.  Writing~$\lambda_{\tilde{p}}(s) = s^{\tilde{p}}(s^{-2} \Delta_{\tilde{p}}(s) + 1)$  for~$\tilde{p} \in \{p, 2p\}$ shows that
	\begin{align*}
		\envert[1]{\mathbb{V}ar|\eps+s|^p-\mathbb{V}ar|\eps|^p}
		&=
		\envert[1]{\lambda_{2p}(s)-\lambda_{p}^2(s)-\lambda_{2p}(0)+\lambda_p^2(0)}\\
		&\leq
		s^{2p-2} \envert[1]{\Delta_{2p}(s) - 2\Delta_p(s) - s^{-2} \Delta_p^2(s) }+\envert[1]{\lambda_{2p}(0)-\lambda_p^2(0)}.
	\end{align*}
	By~\textbf{1.}, the upper bound is dominated by~$1+s^{2p-2}$ times a positive constant. Hence we are done in this case. In case~$s<1$, applying Lemma~\ref{lem:suff0behav} to~$\tilde{p} \in \{p, 2p\}$, we have~$\lambda_{\tilde{p}}(s)-\lambda_{\tilde{p}}(0)\leq c_{\tilde{p}}s^2$; furthermore, the monotonicity property discussed in the last paragraph of the proof of Lemma~\ref{lem:suff0behav} (and the mean-value theorem) shows that
	\begin{align*}
		0 \leq \lambda_{p}^2(s)-\lambda_{p}^2(0)
		\leq
		2\max_{s \in [0, 1]}\lambda_{p}(s)[\lambda_{p}(s)-\lambda_{p}(0)]
		\leq
		2\lambda_{p}(1)c_{p}s^2,  
	\end{align*}
	Hence, for~$s < 1$ we obtain (cf.~the penultimate display)
	\begin{align*}
		\envert[1]{\mathbb{V}ar|\eps+s|^p-\mathbb{V}ar|\eps|^p}
		\leq
		|\lambda_{2p}(s)-\lambda_{2p}(0)|+|\lambda_{p}^2(s)-\lambda_p^2(0)|
		\leq (c_{2p} + 2\lambda_{p}(1) c_p) s^2,
	\end{align*}
	and the statement in~\eqref{eqn:pboundII} thus follows.
	
	\textbf{3.}: Let~$p \in (0, \infty)$ be such that~$2 \lceil 2p \rceil \leq q$. The Dominated Convergence Theorem shows that the function~$s \mapsto \E\envert[1]{|\eps+s|^p-\E|\eps+s|^p}^{4}$ is continuous; furthermore, by symmetry of~$F$, this function is even. To show~\eqref{eqn:pboundIII}, it hence suffices to verify that the function just defined divided by~$s^{4p-2}$ remains bounded as~$s \to \infty$. To this end, let~$s \geq 1$ and write 
	\begin{align}\label{eqn:binom3}
		\E\left[|\eps+s|^p-\E|\eps+s|^p\right]^{4} &= 
		\sum_{i = 0}^4 {4 \choose i} \lambda_{ip}(s) \left(-\lambda_{p}(s)\right)^{4-i},
	\end{align}
	where we set~$\lambda_0 \equiv 1$. Writing~$\lambda_{\tilde{p}}(s) = s^{\tilde{p}}(s^{-2} \Delta_{\tilde{p}}(s) + 1)$  for~$\tilde{p} \in \{ip: i = 0, 1, \hdots, 4\}$ and~$\Delta_0 \equiv 0$, shows that upon dividing~\eqref{eqn:binom3} by~$s^{4p-2}$ we obtain
	\begin{align*}
		&\sum_{i = 0}^4 {4 \choose i} \left(\Delta_{ip}(s) + s^2\right) \left(-(s^{-2}\Delta_p(s) + 1)\right)^{4-i} \\ = ~& \sum_{i = 0}^4 {4 \choose i} \Delta_{ip}(s) \left(-(s^{-2}\Delta_p(s) + 1)\right)^{4-i} + s^2[s^{-2}\Delta_p(s)]^4,
	\end{align*}
	where we used the Binomial formula in the last equality. We now conclude with~\textbf{1.}.
\end{proof}

From now on convergence in distribution as~$d \to \infty$ will be denoted by~``$\rightsquigarrow$.''

\begin{lemma}\label{lem:asyn}
	Under Assumption~\ref{ass:pfindist} and if~$q \in [2, \infty)$ is such that~$\mathbb{E}(|\varepsilon|^{q}) < \infty$, then the following holds for every~$p \in (0, \infty)$ such that~$2 \lceil 2p \rceil \leq q$: for every~$\bm{\vartheta} \in \bm{\Theta}$ and every subsequence~$d'$ of~$d$ along which the sequence~$d^{-1/2} \sum_{i = 1}^d g_p(\theta_{i,d})$ is bounded, we have 
	\begin{equation}\label{eqn:asyn}
		\frac{\|\bm{\theta}_{d'} + \bm{\varepsilon}_{d'}\|_p^p - \sum_{i = 1}^{d'} \lambda_p(\theta_{i,d'})}{\sqrt{d' \mathbb{V}ar|\varepsilon|^p}} \rightsquigarrow \mathbb{N}(0, 1);
	\end{equation}
	in particular, it holds that
	\begin{equation}\label{eqn:asynNULL}
		\frac{\|\bm{\varepsilon}_d\|_p^p - d \mathbb{E}|\varepsilon|^p}{\sqrt{d \mathbb{V}ar|\varepsilon|^p}} \rightsquigarrow \mathbb{N}(0, 1),
	\end{equation}
	so that for~$\alpha \in (0, 1)$ a sequence of critical values~$\kappa_d$ satisfies~$\{p, \kappa_d\} \in \mathbb{T}_{\alpha, F}$ if and only if
	\begin{equation}\label{eqn:CVpreal}
		\kappa_d = \left[\left(\Phi^{-1}(1-\alpha) + o(1)\right)\sqrt{d \mathbb{V}ar|\varepsilon|^p} + d \mathbb{E}|\varepsilon|^p \right]^{1/p}.
	\end{equation}
\end{lemma}

\begin{proof}
	Fix~$p \in (0, \infty)$ such that~$2 \lceil 2p \rceil \leq q$. We give the proof under the assumption that~$d' \equiv d$, which only simplifies the notation. To this end, we verify Lyapunov's condition (with fourth moments) for 
	\begin{align*}
		\xi_{d,i}:= |\theta_{i,d}+\eps_i|^p-\lambda_p(\theta_{i,d}),~d \in \N,~i = 1, \hdots, d.
	\end{align*}
	From \textbf{3.}~of Lemma~\ref{lem:pbounds} we obtain
	\begin{align*}
		\sum_{i=1}^d\E\envert[1]{|\theta_{i,d}+\eps_i|^p-\E|\theta_{i,d}+\eps_i|^p}^4
		\leq
		C''_p \sbr[2]{d+\sum_{i=1}^d|\theta_{i,d}|^{4p-2}\mathds{1}\{|\theta_{i,d}|>1\}}.
	\end{align*}
	The inequality~$\| \cdot \|_r^r \leq \| \cdot \|_1^r$ for~$r\geq 1$, applied with~$r^* := 1\vee ((4p-2)/p)$ delivers
	\begin{align*}
		\sum_{i=1}^d|\theta_{i,d}|^{4p-2}\mathds{1}\{|\theta_{i,d}|>1\}
		&\leq
		\sum_{i=1}^d|\theta_{i,d}|^{pr^*}\mathds{1}\{|\theta_{i,d}|>1\} 
		\leq d^{\frac{r^*}{2}} \del[3]{d^{-1/2}\sum_{i=1}^dg_p(\theta_{i,d})
		}^{r^*}.  
	\end{align*}
	Since~$d^{-1/2} \sum_{i = 1}^d g_p(\theta_{i,d})$ is bounded, we obtain 
	\begin{align*}
		\sum_{i=1}^d\E||\theta_{i,d}+\eps_i|^p-\E|\theta_{i,d}+\eps_i|^p|^4 \leq C_p''\sbr[4]{d+ O\left(d^{r^*/2}\right)} = o\left(d^2\right).
	\end{align*}
	Hence, the Lyapunov condition and~\eqref{eqn:asyn} follow, upon showing that
	\begin{align}\label{eq:varestim}
		\frac{\sum_{i=1}^d\mathbb{V}ar|\theta_{i,d}+\eps_i|^p}{d\mathbb{V}ar|\varepsilon|^p}
		=1+\frac{\sum_{i=1}^d\mathbb{V}ar|\theta_{i,d}+\eps_i|^p-d\mathbb{V}ar|\eps|^p}{d\mathbb{V}ar|\varepsilon|^p}
		\to 1.
	\end{align}
	From \textbf{2.} of Lemma~\ref{lem:pbounds} we obtain
	\begin{equation}\label{eqn:vbdII}
		\left|\sum_{i=1}^d\mathbb{V}ar|\theta_{i,d}+\eps_i|^p-d\mathbb{V}ar|\eps|^p\right|\leq 2 C'_p \sum_{i=1}^d g_{p \vee (2p-2)}(\theta_{i,d}). 
	\end{equation}
	On the one hand, if~$p\in (0, 2]$, then~$2p-2\leq p$, and since~$d^{-1/2}\sum_{i=1}^dg_p(\theta_{i,d})$ is bounded we obtain~\eqref{eq:varestim} from~\eqref{eqn:vbdII}. On the other hand, if~$p\in (2, \infty)$, then~$2p-2>p$ and the inequality~$\| \cdot \|_r^r \leq \| \cdot \|_1^r$ (for~$r \geq 1$) applied with~$r^{**} := (2p-2)/p$ delivers
	\begin{equation}\label{eqn:vbdIIL}
		\begin{aligned}
			\sum_{i=1}^d g_{2p-2}(\theta_{i,d}) &= \sum_{i: |\theta_{i,d}|\leq 1} \theta^2_{i,d} + \sum_{i: |\theta_{i,d}|> 1} |\theta_{i,d}|^{2p-2} \\ &\leq \sum_{i=1}^dg_{p}(\theta_{i,d}) + \left(\sum_{i=1}^dg_{p}(\theta_{i,d})\right)^{(2p-2)/p}.
		\end{aligned}
	\end{equation}
	Boundedness of~$d^{-1/2}\sum_{i=1}^dg_p(\theta_{i,d})$ again delivers~\eqref{eq:varestim}. The statement in~\eqref{eqn:asynNULL} follows from what has been established by setting~$\bm{\vartheta}$ the zero array. The final assertion in the lemma follows immediately from that statement. 
\end{proof}

\subsubsection{Characterizing the consistency set of~$p$-norm based tests with finite~$p > 0$}\label{app:conspreal}

We are now ready to characterize the consistency set of a~$p$-norm based test in the general (not necessarily normal) case. This is done in the following result. Note that the range of~$p > 0$ for which a statement is made in the following theorem depends on the ``highest'' moment that exists for~$\varepsilon$. In particular, if~$\mathbb{E}|\varepsilon|^q < \infty$ for all~$q \in (0, \infty)$, as in the normal case, then one obtains a statement for all~$p \in (0, \infty)$ which coincides with the statement made in Theorem~\ref{thm:conspreal}. If, on the other hand, not all moments exist (but at least the second moment exists), then one still obtains a corresponding statement, but for a limited range of~$p$.

\begin{theorem}\label{thm:consprealGEN}
	Under Assumption~\ref{ass:pfindist} and if~$q \in [2, \infty)$ is such that~$\mathbb{E}(|\varepsilon|^{q}) < \infty$, then the following holds for every~$p \in (0, \infty)$ such that~$2 \lceil 2p \rceil \leq q$: for~$\kappa_{d}$ such that~$\{p, \kappa_d\} \in \mathbb{T}_{\alpha, F}$,~$\alpha \in (0, 1)$, 
	\begin{equation}\label{eqn:cprealGEN}
		\bm{\vartheta} \in \mathscr{C}_F(\{p,\kappa_{d} \}) \quad \Leftrightarrow \quad 
		\frac{\sum_{i = 1}^d g_p(\theta_{i,d})}{\sqrt{d}} \to \infty;
	\end{equation}
	furthermore~$\{p, \kappa_d\}$ has asymptotic power~$\alpha$ against~$\bm{\vartheta} \in \bm{\Theta}$ if and only if the sum on the right in~\eqref{eqn:cprealGEN} converges to~$0$. 
\end{theorem}

\begin{proof}
	
	Let~$\alpha \in (0, 1)$,~$p$ be such that~$2 \lceil 2p \rceil \leq q$, and~$\bm{\vartheta} \in \bm{\Theta}$. We start with two observations:
	
	First, note that~$\|\bm{\theta}_d + \bm{\varepsilon}_d\|_p \geq \kappa_d$ is equivalent to~
	\begin{equation}\label{eqn:split}
		\frac{\|\bm{\theta}_d + \bm{\varepsilon}_d\|_p^p - \sum_{i = 1}^d \lambda_p(\theta_{i,d})}{\sqrt{d \mathbb{V}ar|\varepsilon|^p}} + \frac{\sum_{i = 1}^d\left( \lambda_p(\theta_{i,d}) - \lambda_p(0)\right)}{\sqrt{d \mathbb{V}ar|\varepsilon|^p}} \geq \frac{\kappa^p_d- d\lambda_p(0)}{\sqrt{d \mathbb{V}ar|\varepsilon|^p}} =: \overline{\kappa}_d,
	\end{equation}
	where~$\overline{\kappa}_d  \to \Phi^{-1}(1-\alpha) =: z_{1-\alpha}$ follows from~$\{p, \kappa_d\} \in \mathbb{T}_{\alpha, F}$ and~\eqref{eqn:CVpreal}. 
	
	Second, by Lemma~\ref{lem:suff0behav}, there exists a positive real number~$c_p$ such that
	\begin{equation}\label{eqn:sandwich}
		0 \leq c_p^{-1}
		\frac{\sum_{i = 1}^d g_p(\theta_{i,d})}{\sqrt{d \mathbb{V}ar|\varepsilon|^p}}
		\leq b_d := \frac{\sum_{i = 1}^d\left( \lambda_p(\theta_{i,d}) - \lambda_p(0)\right)}{\sqrt{d \mathbb{V}ar|\varepsilon|^p}}	\leq 
		c_p
		\frac{\sum_{i = 1}^d g_p(\theta_{i,d})}{\sqrt{d \mathbb{V}ar|\varepsilon|^p}} < \infty.
	\end{equation}
	
	We now prove the equivalence in~\eqref{eqn:cprealGEN}:
	
	Let~$\bm{\vartheta} \in \bm{\Theta}$ be such that the sequence~$\sum_{i = 1}^d g_p(\theta_{i,d})/\sqrt{d}$ does not diverge to~$\infty$. We show that~$\bm{\vartheta} \notin \mathscr{C}_F(\{p, \kappa_d\})$. From~\eqref{eqn:sandwich} it follows that there exists a subsequence~$d'$, say, along which~$b_d \to b \in [0, \infty)$, say. From~\eqref{eqn:asyn} in Lemma~\ref{lem:asyn} (applied along the subsequence~$d'$) it follows that the sequence of random variables to the left in~\eqref{eqn:split},~$X_d$, say, converges in distribution to~$\mathbb{N}(b, 1)$ along~$d'$. The Portmanteau Theorem hence implies
	\begin{equation*}
		\liminf_{d \to \infty} \mathbb{P}(\|\bm{\theta}_d + \bm{\varepsilon}_d\|_p \geq \kappa_{d}) \leq \lim_{d' \to \infty}
		\mathbb{P}(X_{d'} - \overline{\kappa}_{d'} \geq 0) = \overline{\Phi}(z_{1-\alpha}-b) < 1,
	\end{equation*}
	that is~$\bm{\vartheta} \notin \mathscr{C}_F(\{p, \kappa_d\})$.
	
	Next, let~$\bm{\vartheta} \in \bm{\Theta}$ be such that the sequence~$\sum_{i = 1}^d g_p(\theta_{i,d})/\sqrt{d} \to \infty$. We show that~$\bm{\vartheta} \in \mathscr{C}_F(\{p, \kappa_d\})$. We first claim that the sequence of random variables
	\begin{equation}\label{eqn:tight}
		\frac{\|\bm{\theta}_d + \bm{\varepsilon}_d\|_p^p - \sum_{i = 1}^d \lambda_p(\theta_{i,d})}{\sum_{i = 1}^d g_p(\theta_{i,d})} = 
		\frac{\sum_{i = 1}^d[ |\theta_{i,d} + \varepsilon_i|^p - \mathbb{E}(|\theta_{i,d} + \varepsilon_i|^p)]}{\sum_{i = 1}^d g_p(\theta_{i,d})}
	\end{equation}
	converges to~$0$ in probability (the quotients are well defined for~$d$ large enough). To prove this claim, since the expectation of the random variables under consideration are all~$0$, it is enough to verify that the sequence of their variances converges to~$0$. By~\textbf{2.} of Lemma~\ref{lem:pbounds}, we can bound these variances via
	\begin{equation}\label{eqn:upnew}
		0 \leq \frac{ \sum_{i = 1}^d \mathbb{V}ar\left(
			|\theta_{i,d} + \varepsilon_i|^p
			\right)}{\left(\sum_{i = 1}^d g_p(\theta_{i,d})\right)^2}
		\leq
		\frac{d\mathbb{V}ar|\eps|^p}{\left(\sum_{i = 1}^d g_p(\theta_{i,d})\right)^2}
		+
		\frac{2C_p'\sum_{i = 1}^d
			g_{p \vee (2p-2)}(\theta_{i,d})}{\left(\sum_{i = 1}^d g_p(\theta_{i,d})\right)^2}.
	\end{equation}
	The first ratio on the far right-hand side in~\eqref{eqn:upnew} converges to~$0$. To see that also the second ratio converges to~$0$ we argue as around~\eqref{eqn:vbdII}: if~$p \in (0, 2]$ then~$2p-2 \leq p$ and we conclude with~$\sum_{i = 1}^d g_p(\theta_{i,d}) \to \infty$. If~$p \in (2, \infty)$ we can use the bound in~\eqref{eqn:vbdIIL} and conclude in the same way.
	
	With this in mind, we now show that~$\mathbb{P}(\|\bm{\theta}_d + \bm{\varepsilon}_d\|_p \geq \kappa_d) \to 1$ if~$\sum_{i = 1}^d g_p(\theta_{i,d})/\sqrt{d} \to \infty$. Let~$d'$ be a subsequence of~$d$. By~\eqref{eqn:sandwich}, there exists a subsequence~$d''$ of~$d'$ along which~$b_d^* := \sum_{i = 1}^d(\lambda_{p}(\theta_{i,d}) - \lambda_p(0))/\sum_{i = 1}^d g_p(\theta_{i,d})$ converges to~$b^* \in (0, \infty)$, say. Now, we re-write the rejection event in~\eqref{eqn:split} as the event that the random variable
	\begin{equation*}
		\frac{\|\bm{\theta}_d + \bm{\varepsilon}_d\|_p^p - \sum_{i = 1}^d \lambda_p(\theta_{i,d})}{\sum_{i = 1}^d g_p(\theta_{i,d})} + b_d^* - \overline{\kappa}_d \frac{\sqrt{d \mathbb{V}ar|\varepsilon|^p}}{\sum_{i = 1}^d g_p(\theta_{i,d})} 
	\end{equation*}
	is non-negative. By the claim established above, and since~$\overline{\kappa}_d$ converges, this sequence of random variables converges along~$d''$ in probability to~$b^* > 0$. Hence, the consistency follows.
	
	The second statement in the theorem follows by slightly modifying the argument after Equation~\eqref{eqn:sandwich}.
\end{proof}

Given Theorem~\ref{thm:consprealGEN} we could now obtain versions of Corollary~\ref{cor:rewrite} and Theorems~\ref{thm:incl} and~\ref{thm:cons_prop} also under the more general Assumption~\ref{ass:pfindist}, but for a range of powers that depends on the ``highest'' moment of~$\varepsilon$, which is required to be greater than~$2$. In particular, if all moments exist as in the discussion before Theorem~\ref{thm:consprealGEN}, then the statements carry over identically. Instead of spelling out all the details, we illustrate one such generalization for the important monotonicity statement in Theorem~\ref{thm:incl}, but do not provide details for other results. The statement is as follows:

\begin{theorem}\label{thm:inclGEN}
	Under Assumption~\ref{ass:pfindist} the following holds: For~$0 < p < q < \infty$ such that~$\mathbb{E}|\varepsilon|^{2 \lceil 2q\rceil} < \infty$, and sequences of tests~$\{p, \kappa_{d,p}\}$ and~$\{q, \kappa_{d,q}\}$ with asymptotic sizes in~$(0, 1)$, we have~$$\mathscr{C}_F(\{p, \kappa_{d, p}\}) \subsetneqq \mathscr{C}_F(\{q, \kappa_{d, q}\}).$$
\end{theorem}

\begin{proof}
	It follows from Theorem~\ref{thm:consprealGEN} that for~$\tilde{p} \in \{p, q\}$ we have~$$\mathscr{C}_F(\{\tilde{p}, \kappa_{d, \tilde{p}}\}) = \mathscr{C}_{\Phi}(\{\tilde{p}, \tilde{\kappa}_{d, \tilde{p}}\}) = \mathscr{C}(\{\tilde{p}, \tilde{\kappa}_{d, \tilde{p}}\}),$$ where~$\tilde{\kappa}_{d, \tilde{p}}$ is a sequence of critical values such that~$\{\tilde{p}, \tilde{\kappa}_{d, \tilde{p}}\}$ has asymptotic size in~$(0, 1)$ under standard normal errors. The statement hence follows from Theorem~\ref{thm:incl}.
\end{proof}

\subsection{Consistency of~supremum-norm based tests}\label{app:conspinf}

We first summarize some assumptions and observations needed in Appendix~\ref{app:conspinf}. The assumptions differ from Assumption~\ref{ass:pfindist} upon which the results in Appendix~\ref{appss:pfin} are based, in that we do not need differentiability of~$F$. However, in two of the results we need to further restrict its tail behavior. This is because the analysis of the supremum-norm based test relies on results from extreme-value theory. The standard normal distribution satisfies all of the conditions imposed. 

The symmetry condition on~$F$ in Assumption~\ref{ass:pfindist} will also be used in the context of supremum-norm based tests.

\begin{assumption}\label{as:symmetry}
	We have~$F(x) = 1-F(-x)$ for every~$x \in \R$.
\end{assumption}

\begin{assumption}\label{as:cont}
	The cdf~$F$ is continuous.
\end{assumption}

Under Assumptions~\ref{as:symmetry} and~\ref{as:cont} we have
\begin{equation}\label{eqn:tailabsgone}
	2\overline{F}(x)  = \P(\varepsilon \geq x) + \P(\varepsilon \leq -x) = \P(|\varepsilon| \geq x) \quad \text{ for all } x \geq 0.
\end{equation}

The following condition will be used whenever we rely on results in~\cite{bogachev2006limit}, where a discussion of this assumption can be found in Section~5; cf.~also~\cite{regvar} for a detailed account of regularly varying functions. Note that in the following assumption we implicitly impose the condition that the support of~$\varepsilon$ is unbounded.

\begin{assumption}\label{as:NRV}
	The (log-tail distribution) function
	\begin{equation}\label{eqn:logtail}
		h(x) := - \log\left(\P(|\varepsilon| \geq x)\right), ~x \in \R,
	\end{equation}
	is normalized regularly varying at infinity with index~$\rho \in (0, \infty)$; that is, for every~$\epsilon > 0$ the functions~$h(x)/x^{\rho-\epsilon}$ and~$h(x)/x^{\rho+\epsilon}$ are ultimately (i.e., for~$x$ large enough) increasing and decreasing, respectively.
\end{assumption}

\begin{remark}\label{rem:NRVnormal}
	That Assumption~\ref{as:NRV} holds with~$\rho = 2$ in case~$\varepsilon$ is standard normally distributed can easily be checked making use of~\eqref{eqn:tailabsgone},~$\Phi'(x)/(x\overline{\Phi}(x)) \to 1$ as~$x \to \infty$, and the characterization for normalized regular variation given in Lemma~5.2 of~\cite{bogachev2006limit} (cf.~also \citet[p.~15]{regvar}).
\end{remark}

A generalized inverse of the non-decreasing function~$h$ defined in~\eqref{eqn:logtail} is defined via~
\begin{equation}\label{eqn:invlogtail}
	h^{\leftarrow}(x) := \inf\{z \in \R: h(z) > x\},
\end{equation}
noting that the set over which the infimum is taken is non-empty for every~$x \in \R$ as~$h(x) \to \infty$ as~$x \to \infty$.

As the first main result in this section, we now provide some statements equivalent to~$\bm{\vartheta} \in \mathscr{C}_F(\{\infty, \kappa_d\})$ under very weak assumptions on~$F$. In the fourth statement in the following proposition we interpret~$1/0 = \infty$,~$a + \infty = \infty$ for every~$a \in (-\infty, \infty]$.

\begin{proposition}\label{prop:abs_noabs} 
	Suppose Assumptions~\ref{as:symmetry} and~\ref{as:cont} hold and let~$\kappa_{d}$ be a sequence of real numbers such that~$\{\infty, \kappa_d\} \in \mathbb{T}_{\alpha, F}$, for an $\alpha \in (0, 1)$. Then, the following statements are equivalent:
	\begin{enumerate}
		\item $\bm{\vartheta} \in \mathscr{C}_F(\{\infty, \kappa_d\})$.
		\item $\P(\max_{i = 1, \hdots, d} (\varepsilon_i + |\theta_{i,d}|) \leq \kappa_d) = \prod_{i = 1}^d F(\kappa_d - |\theta_{i,d}|) \to 0$.
		\item Every subsequence~$d'$ of~$d$ has a subsequence~$d''$ along which~$\overline{F}(\kappa_{d} - \|\bm{\theta}_d\|_{\infty}) \to 1$ or $\sum_{i = 1}^d \overline{F}\left(
		\kappa_{d} - |\theta_{i,d}|
		\right) \to \infty$.
		\item $\sum_{i = 1}^d \overline{F}\left(
		\kappa_{d} - |\theta_{i,d}|
		\right)/F\left(
		\kappa_{d} - |\theta_{i,d}|
		\right) \to \infty$.
	\end{enumerate}
	
\end{proposition}

\begin{proof}
	Before we establish the equivalences in the lemma, we note some trivial equivalences to~$\bm{\vartheta} \in \mathscr{C}_F(\{\infty, \kappa_d\})$ under Assumptions~\ref{as:symmetry} and~\ref{as:cont}: by definition~$\bm{\vartheta} \in \mathscr{C}_F(\{\infty, \kappa_d\})$ is  equivalent to~$\P(\max_{i = 1, \hdots, d} |\varepsilon_i + \theta_{i,d}| \geq \kappa_d) \to 1$, which is equivalent to~$\P(\max_{i = 1, \hdots, d} |\varepsilon_i + |\theta_{i,d}|| \geq \kappa_d) \to 1$, because~$F$ is symmetric. Furthermore,~$\bm{\vartheta} \in \mathscr{C}_F(\{\infty, \kappa_d\})$ is equivalent to~$\P(\max_{i = 1, \hdots, d} |\varepsilon_i + \theta_{i,d}| \leq \kappa_d) \to 0$ because~$F$ is continuous, which is, by symmetry of~$F$, equivalent to~$\P(\max_{i = 1, \hdots, d} |\varepsilon_i + |\theta_{i,d}|| \leq \kappa_d) \to 0$.
	
	$\bm{(1 \Rightarrow 2)}$: We need to show that
	\begin{align}\label{Part1}
		a_d := \P\del[2]{\max_{1\leq i\leq d}|\varepsilon_i + |\theta_{i,d}|| \leq \kappa_d} \to 0 ~~ \text{ implies } ~~ b_d := \prod_{i = 1}^d F(\kappa_d - |\theta_{i,d}|) \to 0.
	\end{align}
	To do so, we argue that if~$a_d \to 0$, then any subsequence $d'$ of $d$ possesses a subsequence  along which $b_d \to 0$. Fix a subsequence~$d'$ of~$d$. The inequality~$b_d
	\leq F(\kappa_d - \|\bm{\theta}_d\|_{\infty})$ shows that in case~$\liminf_{d' \to \infty} F(\kappa_{d'} - \|\bm{\theta}_{d'}\|_{\infty}) = 0$ we are done. Hence, we may assume without loss of generality (otherwise pass to a subsequence) that for some~$\zeta \in (0, \infty)$ we have~
	\begin{equation}\label{eqn:infF}
		F(\kappa_{d'} - |\theta_{i,d'}|) \geq F(\kappa_{d'} - \|\bm{\theta}_{d'}\|_{\infty}) \geq \zeta/2 > 0 \text{ for every } d',  \text{ and every } i = 1, \hdots, d'.
	\end{equation}
	Since~$\{\infty, \kappa_d\} \in \mathbb{T}_{\alpha, F}$, we have~$\alpha_d := \P(\|\bm{\varepsilon}_d\|_{\infty} \geq \kappa_d) \to \alpha \in (0, 1)$, and we can hence assume without loss of generality that for some~$\epsilon > 0$ it holds that~$\alpha_{d'} \in (\epsilon, 1-\epsilon)$ for every~$d'$ (otherwise pass to a subsequence). Then, for every~$d'$
	\begin{equation}\label{eqn:supF}
		F(-\kappa_{d'} - |\theta_{i,d'}|)
		\leq
		F(-\kappa_{d'}) = \overline{F}(\kappa_{d'}) = \frac{1-(1-\alpha_{d'})^{1/d'}}{2} \leq \frac{1-\epsilon^{1/d'}}{2} \leq -\frac{\log(\epsilon)}{2d'},
	\end{equation}
	where we used symmetry of~$F$ to get the first equality,~\eqref{eqn:tailabsgone} and~$\P(|\varepsilon|\leq \kappa_d) = (1-\alpha_d)^{1/d}$ to get the second equality, and~$1-x^{-1} \leq \log(x)$ for every~$x > 0$ to get the last inequality.
	Combining~\eqref{eqn:infF} and~\eqref{eqn:supF}, we can write~$$\P\del[1]{|\varepsilon_{i} + |\theta_{i,d'}||\leq \kappa_{d'}} = 
	F(\kappa_{d'} - |\theta_{i,d'}|) - F(-\kappa_{d'} - |\theta_{i,d'}|)
	$$ as
	\begin{equation*}
		F(\kappa_{d'} - |\theta_{i,d'}|)
		\left[
		1-\frac{F(-\kappa_{d'} - |\theta_{i,d}|)}{F(\kappa_{d'} - |\theta_{i,d'}|)}
		\right]
		\geq
		F(\kappa_{d'} - |\theta_{i,d'}|)\left[1+\frac{\log(\epsilon)}{d'\zeta}\right].
	\end{equation*}
	By~\eqref{eqn:infF}, there exists a~$\underline{d} \in \N$ such that the lower bound just derived is greater than~$0$ for every~$d' \geq \underline{d}$. For all~$d' \geq \underline{d}$, it so follows that
	\begin{align*}
		b_{d'} \left[1+\frac{\log(\epsilon)}{d'\zeta}\right]^{d'}
		\leq
		\prod_{i=1}^{d'}
		\P\del[1]{|\varepsilon_{i} + |\theta_{i,d'}||\leq \kappa_{d'}}
		=
		a_{d'}
		\to 0,
	\end{align*}
	and~$b_{d'} \to 0$ thus follows from~$[1+\frac{\log(\epsilon)}{d'\zeta}]^{d'} \to \epsilon^{1/\zeta}$ as~$d' \to \infty$.
	
	$\bm{(2 \Rightarrow 3)}$: Upon taking the logarithm (interpreting~$-\log(0) = \infty$ and setting~$a + \infty = \infty$ for every~$a \in (-\infty, \infty]$) we obtain from~$\prod_{i = 1}^d F(\kappa_d - |\theta_{i,d}|) \to 0$ that
	\begin{equation}\label{eqn:aequisup}
		\sum_{i = 1}^d -\log\left(1-\overline{F}\left(
		\kappa_{d} - |\theta_{i,d}|
		\right)
		\right) \to \infty.
	\end{equation}
	Let~$d'$ be a subsequence of~$d$. In case~$\limsup_{d' \to \infty} \overline{F}\left(\kappa_{d'} - \|\bm{\theta}_{d'}\|_{\infty}
	\right) = 1$ we are obviously done. We may thus assume (otherwise pass to a subsequence) that for some~$\epsilon \in (0, 1)$ we have~$\overline{F}\left(\kappa_{d'} - \|\bm{\theta}_{d'}\|_{\infty}
	\right) \leq 1-\epsilon$ for every~$d'$. Then, for every~$d'$ and every~$i = 1, \hdots, d'$,
	\begin{equation*}
		\epsilon  \leq F\left(\kappa_{d'} - |\theta_{i,d'}|
		\right).
	\end{equation*}
	Together with~$\frac{-x}{1+x} \geq -\log(1+x)$ for~$x > -1$ we thus obtain
	\begin{equation*}
		\sum_{i = 1}^{d'} -\log\left(1 - 
		\overline{F}\left(
		\kappa_{d'} - |\theta_{i,d'}|
		\right)
		\right) \leq \frac{1}{\epsilon}
		\sum_{i = 1}^{d'} \overline{F}\left(
		\kappa_{d'} - |\theta_{i,d'}|
		\right),
	\end{equation*}
	and~\eqref{eqn:aequisup} thus implies~$\sum_{i = 1}^{d'} \overline{F}\left(
	\kappa_{d'} - |\theta_{i,d'}|
	\right) \to \infty$.
	
	$\bm{(3 \Rightarrow 1)}$: Let~$d'$ be a subsequence of~$d$. We start with two preliminary observations: On the one hand, if there exists a subsequence~$d''$ of~$d'$ along which~$\overline{F}\left(\kappa_{d} - \|\bm{\theta}_{d}\|_{\infty}
	\right) \to 1$, it follows that
	\begin{equation*}
		\P\left(\max_{i = 1, \hdots, d''}( \varepsilon_i + |\theta_{i,d''}|) \leq \kappa_{d''}\right) = \prod_{i = 1}^{d''} F\left(
		\kappa_{d''} - |\theta_{i,d''} |
		\right) \leq F\left(\kappa_{d''} - \|\bm{\theta}_{d''}\|_{\infty}
		\right)\to 0.
	\end{equation*}

	On the other hand, i.e., if~$\limsup_{d' \to \infty} \overline{F}\left(\kappa_{d'} - \|\bm{\theta}_{d'}\|_{\infty}
	\right)<1$, there exists a subsequence~$d''$ of~$d'$ along which~$\sum_{i = 1}^d \overline{F}\left( \kappa_{d} - |\theta_{i,d}|
	\right) \to \infty$, and the inequality~$-\log(1+x) \geq -x$ for every~$x \geq -1$ thus implies 
	\begin{equation*}
		\sum_{i = 1}^d -\log\left(1 - 
		\overline{F}\left(
		\kappa_{d} - |\theta_{i,d}|
		\right)
		\right) \geq \sum_{i = 1}^d 
		\overline{F}\left(
		\kappa_{d} - |\theta_{i,d}|
		\right),
	\end{equation*}
	from which~\eqref{eqn:aequisup} and thus~~$\prod_{i = 1}^{d''} F(\kappa_{d''} - |\theta_{i,d''}|) \to 0$ follows, from which we again conclude that~$\P\left(\max_{i = 1, \hdots, d''}( \varepsilon_i + |\theta_{i,d''}|) \leq \kappa_{d''}\right) \to 0$.
	
	In any case, any subsequence~$d'$ of~$d$ admits a subsequence~$d''$ along which
	\begin{equation*}
		0 = \lim_{d'' \to \infty} \P\left(\max_{i = 1, \hdots, d''} (\varepsilon_i + |\theta_{i,d''}|) \leq \kappa_{d''}\right) \geq \lim_{d'' \to \infty} \P\left(\max_{i = 1, \hdots, d''} |\varepsilon_i + |\theta_{i,d''}|| \leq \kappa_{d''}\right) \geq 0.
	\end{equation*}
	Thus,~$\P(\max_{i = 1, \hdots, d} |\varepsilon_i + |\theta_{i,d}|| \leq \kappa_d) \to 0$, which (cf.~the discussion in the first paragraph of this proof) is equivalent to~$\bm{\vartheta} \in \mathscr{C}_F(\{\infty, \kappa_d\})$.
	
	$\bm{(3 \Leftrightarrow 4)}$: We use the general fact (making use of the convention spelled out before the statement of the proposition) that a triangular array~$\{x_{i,d} \in [0, 1]: d \in \N, i = 1, \hdots, d\}$ satisfies~$	\sum_{i = 1}^d x_{i,d}/(1-x_{i,d}) \to \infty$ if and only if for every subsequence~$d'$ of~$d$ there exists a subsequence~$d''$ of~$d'$, along which
	\begin{equation*}
		\max_{i =1}^d x_{i,d} \to 1 \quad \text{ or }  \quad \sum_{i = 1}^{d}x_{i,d} \to \infty.
	\end{equation*}
	
\end{proof}

The following result gives a sufficient condition for a~$p_d$-norm based test with~$p_d$ diverging to~$\infty$ suitably quickly to be consistent against~$\bm{\vartheta}$. The condition is formulated in terms of the consistency behavior of supremum-norm based tests. In this sense, the result links the consistency set of~$p_d$-norm based tests with diverging~$p_d$ to that of the supremum-norm based test. 

\begin{proposition}\label{prop:pdsup}
	Suppose Assumption~\ref{as:NRV} holds, and let the sequence~$p_d \in (0, \infty)$
	and the sequence of critical values~$\kappa_d$ be such that the sequence of tests~$\{p_d, \kappa_d\} := \mathds{1}\{\|\cdot \|_{p_d} \geq \kappa_d\}$ is in~$\mathbb{T}_{\alpha, F}$,~$\alpha \in (0, 1)$. Under the condition that
	\begin{equation}\label{eqn:fasterlogd}
		\liminf_{d \to \infty} \frac{p_d}{\rho \log(d)} > 1,
	\end{equation}
	the sequence of tests~$\{p_d, \kappa_d\}$ is consistent against~$\bm{\vartheta} \in \bm{\Theta}$ if~\emph{every} supremum-norm based test with asymptotic size in~$(0, \alpha]$ is consistent against~$\bm{\vartheta}$. That is
	\begin{equation}\label{eqn:fasterlogd2}
		\bigcap \left\{ \mathscr{C}_F(\{\infty, \kappa_{d, \infty}\}) :  \{\infty, \kappa_{d, \infty}\}\in \mathbb{T}_{\tilde{\alpha}, F},~\tilde{\alpha} \in (0, \alpha] \right\} \subseteq \mathscr{C}_F(\{p_d, \kappa_d\});
	\end{equation}
	in fact, there exists an~$\alpha^* \in (0, \alpha]$ such that
	\begin{equation}\label{eqn:fasterlogd3}
		\bigcap \left\{ \mathscr{C}_F(\{\infty, \kappa_{d, \infty}\}) :  \{\infty, \kappa_{d, \infty}\}\in \mathbb{T}_{\tilde{\alpha}, F},~\tilde{\alpha} \in [\alpha^*, \alpha] \right\} \subseteq \mathscr{C}_F(\{p_d, \kappa_d\}),
	\end{equation}
	which can be chosen as~$\alpha^* = \alpha$ in case the limit inferior in~\eqref{eqn:fasterlogd} is~$\infty$.
\end{proposition}

\begin{remark}\label{rem:consnotdep}
	If~$F$ is such that the consistency set of any supremum-norm based test with asymptotic size in~$(0, 1)$ neither depends on the sequence of critical values~$\kappa_{d, \infty}$ nor on the actual value of the asymptotic size (which turns out to be the case under normality of the errors, cf.~Theorem~\ref{thm:conspinf}), then the intersection to the left in~\eqref{eqn:fasterlogd2} (and thus also in~\eqref{eqn:fasterlogd3}) coincides with any member of this set. We then see that  
	in case~$p_d$ satisfies~\eqref{eqn:fasterlogd}, every~$p_d$-norm based test with asymptotic size~$\alpha$ dominates (in terms of consistency) any supremum-norm based test with asymptotic size in~$(0, 1)$. 
\end{remark}

\begin{remark}
	Note that the intersection in~\eqref{eqn:fasterlogd3} actually equals
	\begin{equation*}
		\bigcap \left\{ \mathscr{C}_F(\{\infty, \kappa_{d, \infty}\}) :  \{\infty, \kappa_{d, \infty}\}\in \mathbb{T}_{\alpha^*, F} \right\} \subseteq \mathscr{C}_F(\{p_d, \kappa_d\}),
	\end{equation*}
	which follows since for~$\alpha^* \leq \tilde{\alpha}_1 <\tilde{\alpha}_2 \leq \alpha$ the sequence of critical values corresponding to~$\tilde{\alpha}_1$ are eventually larger than those of~$\tilde{\alpha}_2$.
\end{remark}

\begin{proof}
	Note that for~$d$ large enough~$h^{\leftarrow}(\log(d)) > 0$, recall~\eqref{eqn:invlogtail} for a definition of the generalized inverse~$h^{\leftarrow}$, which we can hence without loss of generality take for granted throughout. We start with two preliminary observations: 
	
	First, we note that Proposition~10.1 in \cite{bogachev2006limit} establishes that 
	\begin{equation}\label{eqn:boga1}
		\frac{\rho \log(d)}{h^{\leftarrow}(\log(d))}\left(\|\bm{\varepsilon}_d\|_{\infty} - h^{\leftarrow}(\log(d))\right) \rightsquigarrow \Lambda,
	\end{equation}
	where~$\Lambda(x) = \exp\left(-\exp(-x)\right)$,~$x \in \R$, denotes Gumbel's double exponential cdf. 
	
	Secondly, we fix an~$a \in (0, 1)$ such that~
	\begin{equation}\label{eqn:bseq}
		\liminf_{d \to \infty} p_d/\left(\rho \log(d)/a\right) > 1;
	\end{equation}
	such an~$a$ exists due to~\eqref{eqn:fasterlogd}. Abbreviating~$r_d := \rho \log(d)/a$, Part~(c) of Theorem 2.7 in \cite{bogachev2006limit} (applied with ``$t = r_d$'',~``$\alpha = a$'', and~``$N(t) = d$,'' and noting that~$h$ is eventually continuous and strictly increasing, cf.~the discussion after Condition~5.1 in~\cite{bogachev2006limit} and also Theorem~1.5.5 in~\cite{regvar}) then establishes
	\begin{equation}\label{eqn:boga2}
		\frac{\rho \log(d)}{h^{\leftarrow}(\log(d))}\left(\|\bm{\varepsilon}_{d}\|_{r_d} - h^{\leftarrow}(\log(d))\right) \rightsquigarrow \Lambda_{a},
	\end{equation}
	where~$\Lambda_{a}$ is the distribution of~$a\log(Z_a)$, for~$Z_a$ a random variable with a stable law with characteristic exponent~$a$ and skewness parameter~$1$ (the characteristic function of~$Z_a$ can be found in Equation 2.9 of~\cite{bogachev2006limit}); in particular the cdf corresponding to~$\Lambda_a$ is continuous.
	
	Given the two preliminary observations made, we now argue as follows: By~\eqref{eqn:bseq}, and since we are only concerned with asymptotic statements, we may for simplicity of notation and without loss of generality assume that~$p_d \geq r_d \geq 1$ holds for every~$d$. With this in mind, we have, for every~$d\in \N$,
	\begin{equation}\label{eqn:normin}
		\|\cdot\|_{\infty} \leq \|\cdot\|_{p_d} \leq \|\cdot\|_{r_d}.
	\end{equation}
	From~$\{p_d, \kappa_d\} \in \mathbb{T}_{\alpha, F}$ (cf.~also~\eqref{eqn:size}) and~\eqref{eqn:normin} we can hence conclude that 
	\begin{equation}\label{eqn:boga3}
		\limsup_{d \to \infty } \P\left(\|\bm{\varepsilon}_d\|_{\infty} \geq \kappa_d \right) \leq \alpha \leq \liminf_{d \to \infty } \P\left(\|\bm{\varepsilon}_d\|_{r_d} \geq \kappa_d \right).
	\end{equation}
	Together with~\eqref{eqn:boga1} and~\eqref{eqn:boga2}, continuity of~$\Lambda$ and~$\Lambda_a$, Polya's theorem,~$\alpha \in (0, 1)$, and denoting~$$u_d := \frac{\rho \log(d)}{h^{\leftarrow}(\log(d))}\left(\kappa_d - h^{\leftarrow}(\log(d))\right),$$ this implies 
	\begin{equation}\label{eqn:bogainf}
		-\infty < \Lambda^{-1}(1-\alpha) \leq \liminf_{d \to \infty} u_d \leq  \limsup_{d \to \infty} u_d \leq \Lambda_a^{-1}(1-\alpha) < \infty.
	\end{equation}
	Set~$u_d' := \sup_{m \geq d} u_m \to \limsup_{d \to \infty} u_d =: \overline{u}$. By~\eqref{eqn:boga1}, continuity of~$\Lambda$, Polya's theorem and~\eqref{eqn:bogainf} the supremum-norm based test with sequence of critical values
	\begin{equation}\label{eqn:cvlb}
		\kappa'_d := h^{\leftarrow}(\log(d)) + u_d' \frac{h^{\leftarrow}(\log(d))}{\rho \log(d)} \geq  h^{\leftarrow}(\log(d)) + u_d \frac{h^{\leftarrow}(\log(d))}{\rho \log(d)} = \kappa_d
	\end{equation}
	has asymptotic size~$1-\Lambda(\overline{u}) \in [\alpha^*, \alpha]$ where~$\alpha^* := 1-\Lambda(\Lambda_a^{-1}(1-\alpha)) \in (0, \alpha]$, the inclusion following from~\eqref{eqn:bogainf}. That is~$\{\infty, \kappa_d'\} \in \mathbb{T}_{\tilde{\alpha}, F}$ for some~$\tilde{\alpha} \in [\alpha^*, \alpha]$. Now, let~$\bm{\vartheta}$ be an element of the intersection in~\eqref{eqn:fasterlogd3}. In particular, it follows that~$\bm{\vartheta} \in \mathscr{C}(\{\infty, \kappa_d'\})$, i.e,~$\P\left(
	\|\bm{\theta}_d + \bm{\varepsilon}_d\|_{\infty} \geq \kappa_d'
	\right) \to 1$.
	From~\eqref{eqn:normin} and~\eqref{eqn:cvlb} it follows that~$\P\left(
	\|\bm{\theta}_d + \bm{\varepsilon}_d\|_{p_d} \geq \kappa_d
	\right) \geq \P\left(
	\|\bm{\theta}_d + \bm{\varepsilon}_d\|_{\infty} \geq \kappa_d'
	\right)$, which establishes the inclusion in~\eqref{eqn:fasterlogd3}
	(and thus also the weaker statement in~\eqref{eqn:fasterlogd2}). To prove the last statement, it suffices to note that~$\alpha^* \to \alpha$ for~$a \to 0$ because~$\Lambda_a \to \Lambda$ uniformly as~$a \to 0$, as established in Theorem~10.2 of~\cite{bogachev2006limit} (together with~$\Lambda$ being strictly increasing and continuous).
\end{proof}

The following result is an immediate consequence of Proposition~10.1 in \cite{bogachev2006limit}; cf.~Equation~\eqref{eqn:boga1} in the proof of Proposition~\ref{prop:pdsup}.

\begin{lemma}\label{lem:cvsup}
	Suppose Assumption~\ref{as:NRV} holds. Then, a sequence of critical values~$\kappa_d$ satisfies~$\{\infty, \kappa_d\} \in \mathbb{T}_{\alpha, F}$ if and only if
	\begin{equation}
		\kappa_d = h^{\leftarrow}(\log(d))\left(1 + \frac{\Lambda^{-1}(1-\alpha) + o(1)}{\rho \log(d)}\right),
	\end{equation}
	where~$\Lambda(x) = \exp\left(-\exp(-x)\right)$,~$x \in \R$, denotes Gumbel's double exponential cdf. 
\end{lemma}

\subsection{Question~\ref{q:existdom} and related results in the non-Gaussian case}\label{app:existdomgen}

We now show that Question~\ref{q:existdom} can be answered in the affirmative also in the non-Gaussian case under the condition that~$\varepsilon$ has moments of all orders (one can of course also formulate related results concerning tests that dominate all~$p$-norm based tests up to a certain order, in case~$\varepsilon$ only permits absolute moments up to some order, but we do not spell out the details). 

The following theorem contains a general version of the first part of Theorem~\ref{thm:domp}. The theorem also contains a generalized version of the second part of Theorem~\ref{thm:domp}, which is based on Proposition~\ref{prop:pdsup}, and which, by Remark~\ref{rem:consnotdep}, simplifies to the second statement in Theorem~\ref{thm:domp} under additional conditions.
\begin{theorem}\label{thm:dompgen}
	Suppose Assumption~\ref{ass:pfindist} holds and that~$\mathbb{E}(|\varepsilon|^p) < \infty$ for every~$p \in (0, \infty)$. Let~$p_d$,~$m_d$,~$\alpha$,~$\mathcal{A}$ and~$\mathbb{M}$ be as in the statement of Theorem~\ref{thm:domp}. For every~$d \in \N$ and every~$j = 1, \hdots, m_d$, choose~$\kappa_{j, d} > 0$ and~$c_d \in (0, 1]$ such that
	\begin{equation}\label{eqn:exactsizegen}
		\mathbb{P}\left( \|\bm{\varepsilon}_d \|_{p_j} \geq \kappa_{j, d} \right) = \alpha_{j,d} \quad \text{ and } \quad \mathbb{E}\left(\psi_d(\bm{\varepsilon}_d)\right) = \alpha,
	\end{equation}
	where
	\begin{equation}\label{eq:max_testgen}
		\psi_d(\cdot) := \mathds{1}\left\{ \max_{j = 1, \hdots, m_d} \kappa_{j,d}^{-1}
		\| \cdot \|_{p_j} \geq c_d
		\right\}.
	\end{equation}
	Then, the following holds.
	\begin{enumerate}
		\item The sequence of tests~$\psi_d$ has the property 
		\begin{equation}\label{eqn:bettergen} 
			\mathscr{C}_F(\{p, \kappa_{d}\}) \subseteq \mathscr{C}_F(\psi_d), \text{ for every } p \in (0, \infty) \text{ and every } \{p, \kappa_d\} \in \mathbb{T}_{\alpha, F}.
		\end{equation}	
		\item If Assumption~\ref{as:NRV} is satisfied and
		\begin{equation}\label{eqn:fasterlogd4}
			\liminf_{d \to \infty} \frac{p_{m_d}}{\rho \log(d)} > 1,
		\end{equation}
		then there exists a~$\delta^* \leq \delta := \lim_{d \to \infty} \alpha_{m_d, d}$ such that~$\psi_d$ is consistent against any~$\bm{\vartheta}$ that every supremum-norm based test with asymptotic size in $[\delta^*, \delta]$ is consistent against, i.e.,
		\begin{equation}
			\bigcap \left\{ \mathscr{C}_F(\{\infty, \kappa_{d, \infty}\}) :  \{\infty, \kappa_{d, \infty}\}\in \mathbb{T}_{\tilde{\alpha}, F},~\tilde{\alpha} \in [\delta^*, \delta] \right\} \subseteq \mathscr{C}_F(\psi_d);
		\end{equation}
		furthermore, in case the limit inferior in~\eqref{eqn:fasterlogd4} is~$\infty$ one can choose~$\delta^* = \delta$.
	\end{enumerate}
\end{theorem}

\begin{proof}
	We prove the two statements in the theorem separately:
	
	\textbf{Part 1:} 
	That~$\kappa_{j,d} > 0$ as requested in~\eqref{eqn:exactsizegen} exists follows from the cdf of~$\|\bm{\varepsilon}_d \|_{p}$ being continuous, non-decreasing on~$(0, \infty)$ and zero everywhere else, together with~$\alpha_{j,d} \in (0, 1)$. To show that~$c_d \in (0, 1]$ as in~\eqref{eqn:exactsizegen}  exists, fix~$d \in \N$ and note that by~\eqref{eqn:exactsizegen} and a union bound 
	\begin{equation*}
		\mathbb{P}\left( \max_{j = 1, \hdots, m_d}
		\kappa_{j, d}^{-1} \| \bm{\varepsilon}_d \|_{p_j} \geq 1
		\right)
		\leq
		\sum_{j = 1}^{m_d} \mathbb{P}
		\left(
		\| \bm{\varepsilon}_d\|_{p_j} \geq \kappa_{j, d}
		\right) = \sum_{j = 1}^d \alpha_{j,d} = \alpha;
	\end{equation*}
	Furthermore, observe that the function
	\begin{equation}
		c \mapsto \mathbb{P}\left( \max_{j = 1, \hdots, m_d} \kappa_{j, d}^{-1}
		\| \bm{\varepsilon}_d \|_{p_j} \geq c
		\right)
	\end{equation}
	is continuous and non-increasing on~$(0, \infty)$. Since the limit of this function as~$c \to 0$ ($c \to \infty$) is~$1 > \alpha$ (is~$0 < \alpha$), the existence of~$c_d \leq 1$ as required in~\eqref{eqn:exactsizegen} follows. 
	
	Next, let~$p \in (0, \infty)$ and~$\{p, \kappa_d\} \in \mathbb{T}_{\alpha, F}$. Since~$\mathbb{M} \subseteq \N$ is unbounded, and because~$p_j$ is strictly increasing and unbounded, there exists a~$J \in \mathbb{M}$ (not depending on~$d$) such that~$p_J \geq p$. By assumption,~$m_j\to \infty$ as~$j \to \infty$, from which it follows that eventually~$m_d \geq J$. For all such~$d$,~$c_d \leq 1$ shows that
	\begin{equation}\label{eqn:lowpsigen}
		\psi_d \geq \mathds{1}\{ \| \cdot \|_{p_J} \geq c_d \kappa_{J, d}\} \geq \mathds{1}\{\| \cdot \|_{p_J} \geq \kappa_{J, d}\}.
	\end{equation}
	The sequence~$\kappa_{J, d}$ was chosen such that the test at the far-right in~\eqref{eqn:lowpsigen} has null-rejection probability~$\alpha_{J, d}$. Furthermore, we obtain~$\lim_{d \to \infty} \alpha_{J, d} =: \underline{\alpha}_J \in (0, 1)$ from~$J \in \mathbb{M}$,~\eqref{eqn:Aarray}, and
	~$\underline{\alpha}_J \leq \alpha < 1$.  Hence,~$\{p_J, \kappa_{J, d}\} \in \mathbb{T}_{\underline{\alpha}_J, F}$. Theorem~\ref{thm:inclGEN} now shows that
	\begin{equation}
		\mathscr{C}_F(\{p, \kappa_d\}) \subseteq \mathscr{C}_F(\{p_J, \kappa_{J, d}\}) \subseteq \mathscr{C}_F(\psi_d),
	\end{equation}
	the second inclusion following from~\eqref{eqn:lowpsigen}.
	
	\textbf{Part 2:} Note that
	\begin{equation}\label{eqn:lowpsi2}
		\psi_d \geq \mathds{1}\{ \| \cdot \|_{p_{m_d}} \geq c_d \kappa_{m_d, d}\} \geq \mathds{1}\{\| \cdot \|_{p_{m_d}} \geq \kappa_{m_d, d}\}.
	\end{equation}
	From~\eqref{eqn:exactsizegen} and~\eqref{eqn:Aarray} it follows that~$\{p_{m_d}, \kappa_{m_d, d}\}$, say, the sequence of tests to the far-right in~\eqref{eqn:lowpsi2} has asymptotic size~$\delta \in (0, 1)$. The condition in~\eqref{eqn:fasterlogd4} together with Proposition~\ref{prop:pdsup} allows us to conclude that for some~$\delta^* \in (0, \delta]$ we have
	\begin{equation*}
		\bigcap \left\{ \mathscr{C}_F(\{\infty, \kappa_{d, \infty}\}) :  \{\infty, \kappa_{d, \infty}\}\in \mathbb{T}_{\tilde{\alpha}, F},~\tilde{\alpha} \in [\delta^*, \delta] \right\} \subseteq \mathscr{C}_F(\{p_{m_d}, \kappa_{m_d, d}\}) \subseteq \mathscr{C}_F(\psi_d),
	\end{equation*}
	where the last inclusion followed from~\eqref{eqn:lowpsi2}. The final statement in Part~2 follows from the last statement in Proposition~\ref{prop:pdsup}.
\end{proof}

The following result quantifies the closeness of the power function of~$\psi_d$ to any~$p_j$-norm based test used in its construction. It is a general version of Theorem~\ref{thm:opt_sum} in the non-Gaussian case.

\begin{theorem}\label{thm:opt_sum_gen}
	In the context of Theorem~\ref{thm:dompgen}, fix~$j \in \mathbb{M}$ and set 
	\begin{equation}\label{eqn:alphabardefgen}
		\lim_{d \to \infty} \alpha_{j, d} =: \underline{\alpha}_j > 0.
	\end{equation}
	For any sequence of critical values~$\kappa_{d}$ such that~$\{p_j, \kappa_{d}\} \in \mathbb{T}_{\alpha, F}$, the sequence of tests~$\psi_d$ as defined in~\eqref{eq:max_testgen} satisfies
	\begin{align}\label{eq:powerfunctiongen}
		\limsup_{d\to\infty}\sup_{\bm{\theta}_d\in\R^d} \left[\P\del[1]{\|\bm{\theta}_d+\bm{\eps}_d\|_{p_j} \geq \kappa_{d}}-\E(\psi_d(\bm{\theta}_d+\bm{\eps}_d))\right] \leq \frac{
			\Phi^{-1}(1-\underline{\alpha}_j) - \Phi^{-1}(1-\alpha)}{\sqrt{2\pi}},
	\end{align}
	for~$\Phi$ the cdf of the standard normal distribution, and
	\begin{equation}\label{eq:powerfunction2gen}
		\liminf_{d\to\infty}\sup_{\bm{\theta}_d\in\R^d} \left[\E(\psi_d(\bm{\theta}_d+\bm{\eps}_d)) - \P\del[1]{\|\bm{\theta}_d+\bm{\eps}_d\|_{p_j}
			\geq \kappa_{d}}\right] \geq 1-\alpha.
	\end{equation}
\end{theorem}

\begin{proof}
	We make some initial observations. Unboundedness of~$\mathbb{M}$ together with~\eqref{eqn:Aarray} implies~$\underline{\alpha}_j \in (0, \alpha)$, and from~\eqref{eqn:exactsizegen} it follows that~$\{p_j, \kappa_{j, d}\} \in \mathbb{T}_{\underline{\alpha}_j, F}$. Furthermore, for every~$d \in \N$ and~$\bm{\theta}_d \in \R^d$, the definition of~$\psi_d$ in Equation~\eqref{eq:max_testgen} and~$c_d \in (0, 1]$ readily shows that
	\begin{equation}\label{eqn:pjupbd}
		\begin{aligned}
			\mathbb{P}\left(\|\bm{\theta}_d + \bm{\varepsilon}_d\|_{p_j} \geq\kappa_d\right) 
			- 
			\mathbb{E}(\psi_d(\bm{\theta}_d + \bm{\varepsilon}_d)) 
			&\leq 
			\mathbb{P}\left(\|\bm{\theta}_d + \bm{\varepsilon}_d\|_{p_j} \geq\kappa_d\right)
			-
			\mathbb{P}\left(\|\bm{\theta}_d + \bm{\varepsilon}_d\|_{p_j} \geq \kappa_{j, d}\right).
		\end{aligned}
	\end{equation}
	
	Next, let~$\bm{\vartheta} \in \bm{\Theta}$ and a subsequence~$d'$ of~$d$ be such that the limit superior to the left in~\eqref{eq:powerfunction2gen} coincides with
	\begin{equation}\label{eqn:limsuprepl}
		\lim_{d'\to\infty}  \left[\P\del[1]{\|\bm{\theta}_{d'}+\bm{\eps}_{d'}\|_{p_j} \geq \kappa_{d}}-\E(\psi_{d'}(\bm{\theta}_{d'}+\bm{\eps}_{d'}))\right] =: \mathfrak{s}.
	\end{equation}
	It remains to verify that~$\mathfrak{s} \leq \phi(0)(
	\Phi^{-1}(1-\underline{\alpha}_j) - \Phi^{-1}(1-\alpha))$. Without loss of generality (pass to a further subsequence if necessary) we can assume that $\sum_{i = 1}^{d'} g_{p_j}(\theta_{i,d'})/\sqrt{d'} \to a \in [0, \infty]$. We consider two cases:
	
	\textbf{Case 1:} If~$a = \infty$, we may define~$\bm{\vartheta}^* \in \bm{\Theta}$ such that~$\bm{\theta}_{d'} = \bm{\theta}_{d'}^*$ holds along~$d'$ and such that~$\sum_{i = 1}^{d} g_{p_j}(\theta^*_{i,d})/\sqrt{d} \to \infty$.
	By Theorem~\ref{thm:consprealGEN},~$\bm{\vartheta}^* \in \mathscr{C}_F(\{p_j, \kappa_{j, d}\})$ and~$\bm{\vartheta}^* \in \mathscr{C}_F(\{p_j, \kappa_{d}\})$, from which it follows that
	\begin{equation}
		\mathbb{P}\left(\|\bm{\theta}^*_d + \bm{\varepsilon}_d\|_{p_j} \geq \kappa_{d}\right)
		-
		\mathbb{P}\left(\|\bm{\theta}_d^* + \bm{\varepsilon}_d\|_{p_j} \geq \kappa_{j, d}\right) \to 0.
	\end{equation}
	In particular, it follows that (as~$d' \to \infty$)
	\begin{equation}
		\mathbb{P}\left(\|\bm{\theta}_{d'} + \bm{\varepsilon}_{d'}\|_{p_j} \geq \kappa_{d'}\right)
		-
		\mathbb{P}\left(\|\bm{\theta}_{d'} + \bm{\varepsilon}_{d'}\|_{p_j} \geq \kappa_{j, d'}\right) \to 0.
	\end{equation}
	Hence, in this case~$\mathfrak{s} \leq 0$ by~\eqref{eqn:pjupbd}, and we are done as~$\underline{\alpha}_j < \alpha$.
	
	\textbf{Case 2:} Suppose now that~$a \in [0, \infty)$. Lemma~\ref{lem:asyn} shows that
	\begin{equation}
		\frac{\|\bm{\theta}_{d'} + \bm{\varepsilon}_{d'}\|_p^p - \sum_{i = 1}^{d'} \lambda_p(\theta_{i,d'})}{\sqrt{d' \mathbb{V}ar|\varepsilon|^p}} \rightsquigarrow \mathbb{N}(0, 1)
	\end{equation}
	and that
	\begin{equation}
		\frac{\kappa_{d'}^{p_j} - {d'} \lambda_{p_j}(0)}{\sqrt{d' \mathbb{V}ar(|\varepsilon|^{p_j})}} \to \Phi^{-1}(1-\alpha) ~\text{ and } ~ \frac{\kappa_{j,d'}^{p_j} - d' \lambda_{p_j}(0)}{\sqrt{d' \mathbb{V}ar(|\varepsilon|^{p_j})}} \to \Phi^{-1}(1-\underline{\alpha}_j).
	\end{equation}
	Furthermore, since~$a$ is finite and by Lemma~\ref{lem:suff0behav}, along a subsequence~$d''$ of~$d'$ we have
	\begin{equation}
		\frac{\sum_{i = 1}^{d''} \lambda_{p_j}(\theta_{i,d''}) - \lambda_{p_j}(0)}{\sqrt{d'' \mathbb{V}ar(|\varepsilon|^{p_j})}} \to b \in [0, \infty).
	\end{equation}
	Combining the observations in the previous three displays using a decomposition similar to~\eqref{eqn:split} (and using Polya's theorem), we obtain 
	\begin{align*}
		\mathbb{P}\left(\|\bm{\theta}_{d''} + \bm{\varepsilon}_{d''}\|_{p_j} \geq\kappa_{d''}\right) &\to \overline{\Phi}( \Phi^{-1}(1-\alpha) - b) \\
		\mathbb{P}\left(\|\bm{\theta}_{d''} + \bm{\varepsilon}_{d''}\|_{p_j} \geq\kappa_{j,d''}\right) &\to \overline{\Phi}( \Phi^{-1}(1-\underline{\alpha}_j) - b),
	\end{align*}
	so that the upper bound in~\eqref{eqn:pjupbd}  converges along~$d''$ to
	\begin{equation*}
		\Phi( \Phi^{-1}(1-\underline{\alpha}_j) - b) -
		\Phi( \Phi^{-1}(1-\alpha) - b) \leq \phi(0)\left(
		\Phi^{-1}(1-\underline{\alpha}_j) - \Phi^{-1}(1-\alpha)\right),
	\end{equation*}
	the inequality following from the mean-value theorem.
	
	To prove the statement in~\eqref{eq:powerfunction2gen}, it suffices to exhibit an array~$\bm{\vartheta} \in \mathscr{C}_F(\psi_d)$ such that
	\begin{equation}\label{eqn:pw0}
		\P\del[1]{\|\bm{\theta}_d+\bm{\eps}_d\|_{p_j}
			\geq \kappa_{d}} \to \alpha.
	\end{equation}
	To this end, define~$\bm{\vartheta}$ via~$\bm{\theta}_d := (d^{1/(4p_j)}, 0, \hdots, 0)$, $d \in \N$. Theorem~\ref{thm:consprealGEN} verifies~\eqref{eqn:pw0}, and shows that~$\bm{\vartheta} \in \mathscr{C}_F(\{4p_j, \kappa_d^* \} )$ for~$\{4p_j, \kappa_d^* \} \in \mathbb{T}_{\alpha, F}$,~$\alpha \in (0, 1)$, from which it follows from Theorem~\ref{thm:dompgen} that~$\bm{\vartheta} \in \mathscr{C}(\psi_d)$. 
\end{proof}

\end{document}